\documentclass{jgm} 
\usepackage[utf8]{inputenc}
\usepackage{amsmath}
\usepackage{amssymb}
\usepackage{mathtools}
\mathtoolsset{showonlyrefs}
\usepackage[disable]{todonotes}
\usepackage{graphicx}
\usepackage{etoolbox}
\usepackage{tikz,tikz-cd}
\usepackage[colorlinks]{hyperref}
\hypersetup{urlcolor=blue, citecolor=red}
\def\currentvolume{9}
 \def\currentissue{3}
  \def\currentyear{2017}
   \def\currentmonth{September}
    \def\ppages{335--390}
     \def\DOI{10.3934/jgm.2017014}

\newtheorem{theorem}{Theorem}[section]
\newtheorem{corollary}{Corollary}

\newtheorem{lemma}[theorem]{Lemma}
\newtheorem{proposition}{Proposition}

\newtheorem{conjecture}{Conjecture}

\theoremstyle{definition}
\newtheorem{problem}{Problem}

\newtheorem{definition}[theorem]{Definition}
\newtheorem{remark}{Remark}

\newtheorem{example}{Example}

\newcommand{\rednotes}[1]{}
\newcommand{\reviewerone}[1]{}
\newcommand{\reviewertwo}[1]{}

\newtoggle{responseletter}
\togglefalse{responseletter}
\newtoggle{arxiv}
\toggletrue{arxiv}
\newtoggle{aims}
\toggletrue{aims}
\newtoggle{final}
\toggletrue{final}

\providecommand{\divv}{\nabla\!\cdot}

\DeclareMathOperator{\tr}{tr}
\newcommand{\ee}{\mathrm{e}}

\newcommand{\pair}[1]{\left\langle #1 \right\rangle}

\providecommand{\norm}[1]{\lVert#1\rVert}
\providecommand{\abs}[1]{\lvert#1\rvert}
\providecommand{\vect}[1]{\boldsymbol{#1}}

\newcommand{\ud}{\mathrm{d}}

\newcommand{\pd}{\partial}

\newcommand{\RR}{{\mathbb R}}
\newcommand{\TT}{{\mathbb T}}

\newcommand{\vol}{\ud x}

\newcommand{\Diff}{\mathrm{Diff}}

\newcommand{\Diffvol}{{\Diff_{\vol}}}

\newcommand{\LieD}{\mathcal{L}}

\newcommand{\Dens}{\mathrm{Dens}}

\newcommand{\hess}{\mathrm{Hess}}

\newcommand*\GL{\mathrm{GL}}
\newcommand*\id{\mathrm{id}}
\newcommand*\gl{\mathfrak{gl}}
\newcommand*\SO{\mathrm{SO}}
\newcommand*\OO{\mathrm{O}}

\providecommand{\oo}{\mathfrak{o}}

\newcommand*\g{\mathfrak{g}}

\newcommand{\Sym}[1]{\mathrm{P}(#1)}
\newcommand{\TSym}[1]{\mathrm{S}(#1)}
\newcommand{\normal}{\mathcal{N}}
\newcommand{\Ver}{\mathrm{Ver}}
\newcommand{\Hor}{\mathrm{Hor}}

\newcommand{\costfunc}{J}
\DeclareMathOperator{\cov}{cov}

\newcommand{\diag}[1]{\mathrm{D}(#1)}
\newcommand{\poly}[1]{\mathrm{poly}^+_{#1}}
\newcommand{\Orb}[1]{\mathrm{Orb}(#1)}
\newcommand{\diagord}[1]{\mathrm{D}^<(#1)}

\title[Geometry of Matrix Decompositions] 
      {Geometry of Matrix Decompositions Seen Through Optimal Transport and Information Geometry}

\author[Klas Modin]{}


\email{klas.modin@chalmers.se}

\thanks{This project has received funding from the European Union’s Horizon 2020 research and innovation programme under grant agreement No 661482, and from the Swedish Foundation for Strategic Research under grant agreement ICA12-0052.}

\begin{document}
\maketitle

\centerline{\scshape Klas Modin$^*$}
\medskip
{\footnotesize
 \centerline{Department of Mathematical Sciences}
   \centerline{Chalmers University of Technology and University of Gothenburg}
   \centerline{SE-412~96 Gothenburg, Sweden}
} 



\begin{abstract}
The space of probability densities is an infinite-dimensional Riemannian manifold, with Riemannian metrics in two flavors: Wasserstein and Fisher--Rao.
The former is pivotal in optimal mass transport (OMT), whereas the latter occurs in information geometry---the differential geometric approach to statistics.
The Riemannian structures restrict to the submanifold of multivariate Gaussian distributions, where they induce Riemannian metrics on the space of covariance matrices.

Here we give a systematic description of classical matrix decompositions (or factorizations) in terms of Riemannian geometry and compatible principal bundle structures.
Both Wasserstein and Fisher--Rao geometries are discussed.
The link to matrices is obtained by considering OMT and information geometry in the category of linear transformations and multivariate Gaussian distributions.
This way, OMT is directly related to the polar decomposition of matrices, whereas information geometry is directly related to the $QR$, Cholesky, spectral, and singular value decompositions.
We also give a coherent description of gradient flow equations for the various decompositions;
most flows are illustrated in numerical examples.

The paper is a combination of previously known and original results.
As a survey it covers
the Riemannian geometry of OMT and polar decompositions (smooth and linear category),
entropy gradient flows, and the Fisher--Rao metric and its geodesics on the statistical manifold of multivariate Gaussian distributions.
The original contributions include new gradient flows associated with various matrix decompositions, new geometric interpretations of previously studied isospectral flows, and a new proof of the polar decomposition of matrices based an entropy gradient flow.
\end{abstract}

\section{Introduction}

The influence of matrix decompositions in scientific computing cannot be overestimated.
Numerical linear algebra, the subject treating computer algorithms for matrix decompositions, is part of the curriculum of almost every mathematics department.
A typical course follows an \emph{algorithmic} approach, based on algebra, combinatorics, and some analysis.
It is also possible, although far less common, to follow a \emph{geometric} approach, based on Riemannian geometry and Lie group theory; this
point of view reveals hidden dynamical features of matrix decompositions. 
The best known examples are perhaps the \emph{isospectral flows}~\cite{Wa1984}, in particular the \emph{Toda flow}~\cite{To1970} with its connection to the $QR$~algorithm~\cite{Sy1982,DeNaTo1983}, and the work by Brockett~\cite{Br1991b} who produced a gradient flow that diagonalize matrices.
There are also other examples, known to specialists but less known among most practitioners of numerical linear algebra~\cite{Ru1954,Ru1958,Bl1985,ChNo1988,Bl1990,ChDr1991,MoVe1991,DeDeLiTo1991,Ch1992,HeMo1992,HeMoPe1994,Ch1995,GeMa1997}.
Overviews and further references are available in survey papers~\cite{Wa1984,Ch1994,Ch2008b,To2013}.
\reviewertwo{
  “There are also other examples. . . ”: references?  \\[2ex]
  DONE!
}

In this paper we present a systematic Riemannian geometric description of classical matrix decompositions.
We thereby give a Lagrangian perspective of decompositions that have previously been studied as Hamiltonian problems, such as the integrable Hamiltonian structure of the Toda flow and its generalizations \cite{Mo1975,Ad1978,ReSe1979,Sy1980}.
We also show how entropy gradient flows, traditionally studied in infinite-dimensional settings, are connected to matrix decompositions.
The examples treated are well suited for a first course on Riemannian geometry, as they provide students with new perspectives on topics familiar from numerical linear algebra. 
In addition, our approach gently introduces the more advanced subjects \emph{optimal mass transport} (see Vilani~\cite{Vi2009}) and \emph{information geometry} (see Amari and Nagaoka~\cite{AmNa2000}).
Throughout the paper we avoid most aspects of analysis, focusing on geometry.

\reviewerone{
  The biggest issue I had with the paper was whether it was a “survey” paper or original research. Clearly it is some combination of both. The abstract describes it as a survey, and there is a large amount of background material, but Sections 2.3–2.4 and about half of Section 3 seem to be new (the remainder of Section 3 which is not new is largely due to the author in another paper). As a survey paper it lacks references (only 23 compared to the 39 that appear in “Generalized Hunter- Saxton...”) and could also use more historical context and explicit examples. As original research it seems to be rather casual about distinguishing between the author’s new work and the material that it is summarizing. This could be resolved with an explicit description in the beginning of the paper about what is old and what is new.
  \\[2ex]
  DONE! A sub-section detailing the contributions has been added.
}

The first step in our geometric approach to matrix decompositions is to provide the space of Gaussian distributions with a Riemannian metric.
There are two standard choices: Wasserstein and Fisher--Rao.
The former occurs in optimal mass transport (OMT) and gives the polar decomposition of matrices.
The latter occurs in information geometry and gives the $QR$, Cholesky, spectral, and singular value decompositions. 
In \autoref{sec:wasserstein} we describe the Wasserstein geometry, through optimal mass transport, and the link to polar decompositions.
We consider two different categories of transformations: diffeomorphisms (infinite-dimensional) and general linear transformations (finite-dimensional).
In \autoref{sec:fisher_rao} we describe the Fisher--Rao geometry, through information geometry, and we show how it gives rise to the aforementioned matrix decompositions.
Throughout these sections we also give a coherent approach to vertical and horizontal Riemannian gradient flows, to recover the matrix decompositions.



We now continue by introducing the basic ingredients of the paper, without specifying any Riemannian structure.




Let $x=(x^1,\ldots,x^n)$ denote Euclidean coordinates on $\RR^n$.
A multivariate Gaussian distribution with mean zero is a distribution with probability density of the form
\begin{equation}\label{eq:gaussians}
  p(x,\Sigma)\,\vol \coloneqq \frac{1}{\sqrt{\det(\Sigma)(2\pi)^n}}\exp(-\frac{1}{2}x^{\top}\Sigma^{-1}x)\,\vol,
\end{equation}
for some symmetric, positive definite \emph{covariance matrix} $\Sigma$.
In the language of information geometry (\emph{cf.} Amari and Nagaoka~\cite{AmNa2000}), the set $\normal_n$ of all such probability distributions constitutes a \emph{statistical manifold}.
That is, $\normal_n$ is a submanifold of the infinite-dimensional manifold
\begin{equation}
  \Dens(\RR^n) = \{ p\, \vol \mid p\in C^\infty(M),\int_{\RR^{n}}p\,\vol = 1 \}
\end{equation}
of all (smooth) probability densities on $\RR^n$.
  There are several ways to equip $\Dens(\RR^n)$ with an infinite-dimensional manifold structure, see for example~\cite{EbMa1970,Ha1982,KrMi1997b}.
  However, as mentioned we shall not go into analysis. 


Consider the space of all symmetric $n$-by-$n$ matrices
\begin{equation}
  \TSym{n} = \left\{  S\in \RR^{n\times n}; S^{\top} = S \right\}.
\end{equation}
Denote the subset of those matrices that are also positive definite by
\begin{equation}
  \Sym{n} = \left\{ \Sigma\in \TSym{n}; x^\top \Sigma x>0,\forall x\in\RR^{n}\backslash\{0\}\right\}.
\end{equation}
Then $\Sym{n}$ is a $\frac{1}{2}n(n+1)$--dimensional manifold that is isomorphic to the statistical manifold~$\normal_n$.
In other words, the mapping
\begin{equation}\label{eq:gaussians2}
  \Sym{n}\ni \Sigma \;\longmapsto\;
  p(\cdot,\Sigma)\,\vol
  \in\Dens(\RR^{n})
\end{equation}
is injective.
$\Sym{n}$ is therefore isomorphic, as a manifold, to $\normal_n$.
This is the first key ingredient towards matrix decompositions.
Let us now explain the second.

Denote by $\Diff(\RR^{n})$ the set of diffeomorphisms of $\RR^{n}$.
Then $\Diff(\RR^{n})$ is an infinite-dimensional Lie group (with respect to a certain topology) that is acting on $\Dens(\RR^{n})$ from the right by \emph{pullback}
\begin{equation}\label{eq:pullback_action}
  \Diff(\RR^{n})\times \Dens(\RR^{n}) \ni (\varphi,p\,\vol) \longmapsto \varphi^*(p\,\vol) = p\circ\varphi \det(D \varphi)\ud x \in \Dens(\RR^{n}).
\end{equation}
The corresponding left action is given by \emph{pushforward}
\begin{equation}
  \Diff(\RR^{n})\times \Dens(\RR^{n}) \ni (\varphi,p\,\vol) \longmapsto \varphi_*(p\,\vol) = (\varphi^{-1})^*(p\,\vol) \in \Dens(\RR^{n}).
\end{equation}
A subgroup of $\Diff(\RR^{n})$ is given by the linear transformations
\begin{equation}
  \{\varphi\in\Diff(\RR^n)\mid \varphi(x) = Ax,\; A\in\GL(n)\}\simeq \GL(n).
\end{equation}
The submanifold $\normal_n$ is invariant under the action of linear transformations (see \autoref{sub:omt_linear} for details).
Since $\normal_n\simeq \Sym{n}$ we can therefore replace the pair of infinite-dimensional manifolds $\Diff(\RR^{n})$ and $\Dens(\RR^{n})$ by the pair of finite-dimensional manifolds $\GL(n)$ and $\Sym{n}$.
This is the second key ingredient towards matrix decompositions.
Indeed, as we shall see all the aforementioned decompositions are obtained through an interplay between the action of $\GL(n)$ on $\Sym{n}$ and compatible Riemannian structures.

Throughout the paper we use the notion of \emph{Riemannian submersions} and \emph{descending metrics}.
A brief, self-contained presentation of these concepts is given in~\cite[\S\!~4]{Mo2015}.
For more details, we refer to Lang~\cite[Ch.\!~XIV]{La1999} and Petersen~\cite[\S\!~3.5]{Pe2006}.

\subsection{Contributions}\label{sec:contributions}
This paper is a combination of a survey and original research.
In this section we clarify what is old and what is new.

A first, general contribution is an explicit connection between matrix decompositions and infinite-dimensional geometry of groups of diffeomorphisms (as studied in \cite{KrMi1997b,KhWe2009} and references therein).
The more specific, original research contributions are
\begin{itemize}
  \item a vertical gradient flow for the polar decomposition of diffeomorphisms (\autoref{subsub:vertical_gradient_flow_smooth});
  \item a lifted entropy gradient flow for the polar decomposition of diffeomorphisms (\autoref{subsub:lifted_gradient_flow_smooth}), and a geometric proof strategy for OMT in the smooth category (\autoref{subsub:limit});
  \item a vertical gradient flow for the polar decomposition of matrices (\autoref{subsub:vert_gradient_flow_linear_omt});
  \item a new, geometric proof for the polar decomposition of matrices (\autoref{sub:polar_decomposition_matrices}, \autoref{thm:polar_decomposition_matrices}), based on convergence to a unique minimum of a lifted entropy gradient flow (\autoref{subsub:lifted_gradient_omt_linear}, \autoref{thm:limit_lifted_entropy_flow_omt_finite});
  \item a new, geometric proof for the $QR$ decomposition of matrices (\autoref{sec:qr}, \autoref{thm:QR}), based on convergence to a unique minimum of a lifted horizontal entropy gradient flow (\autoref{subsub:lifted_gradientflow_QR}, \autoref{thm:limit_lifted_entropy_flow_QR});
  \item a description of the Cholesky decomposition through Fisher--Rao geometry (\autoref{sec:cholesky});
  \item a new interpretation of Brockett's diagonalizing flow \cite{Br1991b} as a pullback entropy gradient flow (\autoref{sub:isospectral}, \autoref{cor:brockett_as_pullback_entropy}), and of the double bracket flow as an entropy gradient flow (\autoref{sub:isospectral}, \autoref{cor:brockett_double_bracket_as_entropy_gradient_flow});
  \item a new horizontal gradient flow to factorize characteristic polynomials and thereby obtain the spectral decomposition (\autoref{subsub:hor_flow_factorize_polynomials});
  \item a description of the singular value decomposition (SVD) through the developed Fisher--Rao geometry (\autoref{sec:svd}).
\end{itemize}

As a survey, the paper covers
\begin{itemize}
  \item the Riemannian structure of OMT and the Wasserstein distance (\autoref{sub:geometry_omt_smooth});
  \item the polar decomposition of diffeomorphisms described geometrically (\autoref{sub:polar_decomposition_smooth});
  \item the Fokker--Planck and heat equations as entropy gradient flows (\autoref{subsub:entropy_gradient_flow_smooth});
  \item the geometric description of optimal transport in the linear category (\autoref{sub:omt_linear});
  \item the polar decomposition of matrices described geometrically (\autoref{sub:omt_linear});
  \item the Fisher--Rao metric on the space of multivariate Gaussian distributions and the associated geodesics (\autoref{sec:fisher_rao});
  \item the extension of Fisher--Rao to the Riemannian metric on $\GL(n)$ associated with the $QR$ (or Iwasawa) decomposition (\autoref{sec:qr});
  \item the homogeneous space structure of the space of positive definite symmetric matrices (\autoref{sec:eigen});
  \item the Fisher--Rao metric on the space of diagonal positive definite matrices and the associated geodesics (\autoref{subsub:descending_metric_diag});
  \item the geometric description of isospectral flows (\autoref{sub:isospectral}).
\end{itemize}

Some readers might be interested in specific parts, but not the whole paper.
Therefore, we have tried to keep each section as independent as possible.
For example, sections \autoref{sub:geometry_omt_smooth}--\autoref{sub:polar_decomposition_smooth}, \autoref{sub:omt_linear}--\autoref{sub:polar_decomposition_matrices}, \autoref{sub:principal_bundle_qr}--\autoref{sec:qr}, and \autoref{sec:eigen} are almost independent in themselves.
Yet, care has been taken to point out both similarities and differences between the decompositions in Wasserstein geometry and in Fisher--Rao geometry.

\vspace*{-4pt}
\section{Wasserstein geometry and polar decompositions}\label{sec:wasserstein}

Recall the (real) polar decomposition of matrices: if $A\in\GL(n)$ there are unique matrices $P\in \Sym{n}$ and $Q\in \OO(n)$ such that $A=PQ$.
Brenier~\cite{Br1991} showed that this decomposition provides a finite-dimensional ``toy example'' of optimal mass transport.
In this section we shall look at optimal mass transport and polar decompositions from the point-of-view of Riemannian geometry.
Our presentation is essentially a combination of results found in~\cite{Br1991,Ot2001} and in~\cite[App.\!~5]{KhWe2009}, but with some new aspects, especially related to vertical and lifted gradient flows (as listed in \autoref{sec:contributions}).

\reviewertwo{
  At the beginning of Section 2 it is pointed out that the following presentation is based on previous results. I would strongly recommend to give further references throughout the Section itself. Possibly it may be helpful to refer to [3] \texttt{(J. Lott, given by reviewer)}.
  The citation is \citet{Lo2008}.
  \\[2ex]
  Citations are now given throughout the text.
  In particular, the paper by Lott suggested by the reviewer is cited.
}

\reviewertwo{
  A general reference for the concept of constructing metrics via descent would be appropriate. Probably the author can choose such a reference much better than I could. \\[2ex]
  DONE! (Added in the introduction.)
}

\reviewerone{
  On page 3 I would cite [13] Appendix 5; it’s a bit much to cite a whole book without specifying where the reader should go. \\[2ex]
  DONE!
}

\subsection{Optimal transport in the smooth category}\label{sub:geometry_omt_smooth}
Let us recall the classical $L^2$ optimal mass transport (OMT) problem of Monge~\cite{Mo1781}.
Usually, this problem is presented and analyzed in the very general setting of probability measures and measurable maps.
Here we present it in the smooth category, focusing on geometry rather than analysis.
\begin{problem}[Smooth OMT]\label{prob:smooth_omt}
 {\it Given $\mu_0,\mu_1\in\Dens(\RR^n)$, find $\varphi\in\Diff(\RR^n)$ that minimizes
  \begin{equation}\label{eq:omt_functional}
    \costfunc(\varphi) = \int_{\RR^n} \norm{x-\varphi(x)}^2\,\mu_0
  \end{equation}\\
  under the constraint
  \begin{equation}\label{eq:omt_constraint}
    \varphi_*\mu_0 = \mu_1.
  \end{equation}}
\end{problem}

Before diving into the geometric description of this problem, we list some properties.
\begin{enumerate}
  \item If we write the densities as $\mu_0 = \rho_0\ud x$ and $\mu_1 = \rho_1\ud x$, then the constraint~\eqref{eq:omt_constraint} reads
  \begin{equation}
    \det(D\varphi^{-1})\rho_0\circ\varphi^{-1} = \rho_1,
  \end{equation}
  where $\det(D\varphi^{-1})$ denotes the Jacobian determinant of $\varphi^{-1}$.
  \item The problem is symmetric in $\mu_0$ and $\mu_1$ under the inversion $\varphi\mapsto\varphi^{-1}$.
  Indeed, if $\varphi$ fulfills the constraint $\varphi_*\mu_0=\mu_1$, then
  \begin{equation}\label{eq:func_omt}
    \begin{split}
    \costfunc(\varphi) &= \int_{\RR^n} \norm{x-\varphi(x)}^2\,\mu_0 \\ &= \int_{\RR^n} \varphi_*\left(\norm{x-\varphi(x)}^2\,\mu_0\right)
    = \int_{\RR^n} \norm{x-\varphi^{-1}(x)}^2\,\mu_1.
    \end{split}
  \end{equation}
  Thus, if $\varphi$ is a solution to~\autoref{prob:smooth_omt}, then $\varphi^{-1}$ is a solution to the reverse problem ($\mu_0$ exchanged for $\mu_1$ and vice versa).
\end{enumerate}

To describe the geometry of~\autoref{prob:smooth_omt}, we think of $\Diff(\RR^n)$ and $\Dens(\RR^n)$ as infinite-dimensional manifolds.
The tangent space $T_\varphi\Diff(\RR^n)$ is identified with the space of smooth maps $C^\infty(\RR^n,\RR^n)$.
A Riemannian metric on $\Diff(\RR^n)$ is given by
\begin{equation}\label{eq:L2noninvariant}
  \mathcal{G}_\varphi(\dot\varphi,\dot\varphi) = \int_{\RR^{n}} \norm{\dot\varphi(x)}^{2}\mu_0 .
\end{equation}
\reviewerone{
  The factor of 1/2 in front of the metric in equation (3) seems odd to me. It does not affect the
geodesics, obviously, but it does keep showing up as a factor of 2 in the gradient flows that the author works with. Of course the 1/2 is standard when talking about the kinetic energy, but the Riemannian metric would ordinarily not have it. \\[2ex]
DONE! Now the factor 1/2 has been removed. As the reviewer hinted, many of the equations now look simpler.
}%
The geodesic equation associated with this Riemannian metric is very simple: since~$\mathcal{G}$ is independent of the base point $\varphi$, and since $C^\infty(\RR^n,\RR^n)$-variations of elements in $\Diff(\RR^n)$ remain in $\Diff(\RR^n)$, it is given by
\begin{equation}\label{eq:geodesic_eq_omt}
  \ddot\varphi = 0.
\end{equation}
In particular, the geometry is flat.

Given two diffeomorphisms $\varphi_0$ and $\varphi_1$, it follows from~\eqref{eq:geodesic_eq_omt} that there is a unique geodesic curve $[0,1]\mapsto\gamma(t)$ between them, given by
\begin{equation}
  \gamma(t) = (1-t)\varphi_0 + t \varphi_1.
\end{equation}
There is no guarantee that the path remains in $\Diff(\RR^n)$, although it always remains in $C^\infty(\RR^n,\RR^n)$, but we disregard this for now.
Instead, let us compute the energy (squared length) of~$\gamma$.
By definition, it is given by
\begin{equation}
  \begin{split}
  d^2(\varphi_0,\varphi_1) &\coloneqq \int_{0}^{1}\mathcal{G}_{\gamma(t)}(\dot\gamma(t),\dot\gamma(t))\,\ud t  = \int_{0}^{1} \int_{\RR^{n}}\norm{\varphi_1(x)-\varphi_0(x)}^{2}\,\mu_0 \,\ud t \\
  &=  \int_{\RR^{n}}\norm{\varphi_1(x)-\varphi_0(x)}^{2}\,\mu_0.
  \end{split}
\end{equation}
Therefore, if $\id$ denotes the identity mapping on $\RR^{n}$, the functional $\costfunc$ in~\eqref{eq:omt_functional} can be written
\begin{equation}
  \costfunc(\varphi) = d^2(\id,\varphi).
\end{equation}
Hence, OMT becomes a problem of Riemannian geometry: find the shortest geodesic curve from the identity to the constraint set $\mathcal C(\mu_0,\mu_1)\coloneqq\{\varphi\mid\varphi_*\mu_0=\mu_1 \}$.
This observation is already compelling, but it gets even more interesting.

The pushforward yields a left action of $\Diff(\RR^{n})$ on $\Dens(\RR^{n})$, with action map
\begin{equation}
  (\varphi,\mu) \mapsto \varphi_*\mu.
\end{equation}
The isotropy group of $\mu_0$ with respect to this action is given by the subgroup
\begin{equation}
  \Diff_{\mu_0}(\RR^{n})  = \{\varphi\in\Diff(\RR^{n})\mid \varphi_*\mu_0=\mu_0 \}.
\end{equation}
Notice that the constraint set $\mathcal C(\mu_0,\mu_1)$ is closed under the right action of $\Diff_{\mu_0}(\RR^{n})$.
That is, if $\zeta\in\mathcal C(\mu_0,\mu_1)$ and $\psi\in\Diff_{\mu_0}(\RR^{n})$, then $\zeta\circ\psi \in \mathcal C(\mu_0,\mu_1)$.
It is a short calculation to show that the reverse is also true: if $\zeta,\eta\in\mathcal C(\mu_0,\mu_1)$, then there exists a $\psi\in\Diff_{\mu_0}(\RR^{n})$ such that $\eta = \zeta\circ\psi$.
From the so called ``Moser-trick'' \cite{Mo1965} extended to non-compact manifolds~\cite{GrSh1979} we get that $\mathcal C(\mu_0,\mu_1)$ is always non-empty.
Thus, by fixing $\zeta\in\mathcal C(\mu_0,\mu_1)$, the constraint set $\mathcal C(\mu_0,\mu_1)$ is isomorphic to $\Diff_{\mu_0}(\RR^{n})$ by the mapping
\begin{equation}\label{eq:fiber_isomorphism_diffeos}
  \mathcal C(\mu_0,\mu_1)\ni \eta \mapsto \zeta^{-1}\circ\eta\in\Diff_{\mu_{0}}(\RR^{n}),
\end{equation}
so $\Diff_{\mu_0}(\RR^{n})$ parameterizes $\mathcal C(\mu_0,\mu_1)$.
If we define a projection mapping 
\begin{equation}\label{eq:projection_smooth_omt}
  \pi\colon\Diff(\RR^{n})\to\Dens(\RR^{n}), \quad \pi(\varphi) = \varphi_*\mu_0,
\end{equation}
\reviewerone{
  At the top of page 5 the use of variables makes things slightly confusing: in the first equation it is important that $\varphi$ is a fixed but arbitrary element of $C(\mu_0, \mu_1)$, while in the immediate next equation $\varphi$ is an arbitrary element of $\Diff(\RR^n)$. This could be a bit clearer just by specifying where these elements live when they are used. \\[2ex]
  DONE! (Now using $\zeta$ instead of $\varphi$ for the fixed element in the constraint set)
}%
then the constraint set $\mathcal C(\mu_0,\mu_1)$ is given by $\pi^{-1}(\mu_1)$.
\rednotes{
We call $\pi^{-1}(\mu_1)$ the \emph{fiber} of $\mu_1$.
Since $\pi^{-1}(\mu)$ is non-empty for every $\mu\in\Dens(\RR^{n})$, the left action of $\Diff(\RR^{n})$ on $\Dens(\RR^{n})$ is transitive.
By construction, the space of fibers is isomorphic to $\Dens(\RR^{n})$.
}
The sets $\pi^{-1}(\mu)$ are then the \emph{fibers} of a principal $\Diff_{\mu_0}(\RR^{n})$-bundle over $\Dens(\RR^{n})$:
\begin{equation}\label{eq:principal_bundle_omt}
  \begin{tikzcd}
    \Diff(\RR^{n}) \arrow[hookleftarrow]{r}{} \arrow{d}{\pi} & \Diff_{\mu_0}(\RR^{n}) \\
    \Dens(\RR^{n}) &
  \end{tikzcd}
\end{equation}
In other words, the space of left co-sets
\begin{equation}
  \Diff(\RR^{n})/\Diff_{\mu_0}(\RR^{n}) = \{ [\varphi]\coloneqq\varphi\circ\Diff_{\mu_0}(\RR^{n}) \mid \varphi\in\Diff(\RR^{n}) \}
\end{equation}
is isomorphic to $\Dens(\RR^{n})$ by the mapping~\eqref{eq:projection_smooth_omt}.

Let us now compute the derivative of the principal bundle projection~$\pi$ at $\varphi\in\Diff(\RR^{n})$.
To this end, let $\gamma(t)$ be a curve in $\Diff(\RR^{n})$ with $\gamma(0) = \varphi$ and let $u(t) = \dot\gamma(t)\circ\gamma(t)^{-1}$.
By the Lie derivative theorem (cf.~\cite[\S\,4.3]{MaRa1999}) it follows that
\begin{equation}
  \frac{\ud}{\ud t}\gamma(t)^*\mu = \gamma(t)^*\LieD_{u(t)}\mu,
\end{equation}
for any $\mu\in\Dens(\RR^{n})$.
From the product rule and $\gamma(t)_*\mu_0 \coloneqq (\gamma(t)^{-1})^*\mu_0$ we then get
\begin{equation}
  0 = \frac{\ud}{\ud t}\left( \gamma(t)^*\gamma(t)_*\mu_0 \right) = \gamma(t)^{*} \frac{\ud}{\ud t}\gamma(t)_{*}\mu_0 + \gamma(t)^{*}\LieD_{u(t)}\gamma(t)_{*}\mu_0.
\end{equation}
Applying the pullback of $\gamma(t)$ and rearranging the terms now yields
\begin{equation}
  \frac{\ud}{\ud t}\gamma(t)_{*}\mu_0 = - \LieD_{u(t)}\gamma(t)_{*}\mu_0.
\end{equation}
Since $\gamma(0)_{*}\mu_0 = \pi(\varphi)$ we get that the derivative of $\pi$ applied to $U\in T_{\varphi}\Diff(\RR^{n})$ is given by
\begin{equation}\label{eq:Dpi_inf_dim}
  D\pi(\varphi)\cdot U = - \LieD_{u}\varphi_{*}\mu_0,
  \qquad u \coloneqq U\circ\varphi^{-1}.
\end{equation}
Now, define the function $\rho\in C^{\infty}(\RR^{n})$ by
\begin{equation}
  \varphi_{*}\mu_0 = \rho \vol .
\end{equation}
Since $\mu_0 = \rho_0\vol$ and $\varphi_{*}(\rho_0\vol) = (\rho_0\circ\varphi^{-1})\varphi_{*}\vol$, it follows that
\begin{equation}
  \rho=\det(D\varphi^{-1})\rho_0\circ\varphi^{-1}.
\end{equation}
The derivative $D\pi$ can now be written
\begin{equation}\label{eq:deriv_pi}
  D\pi(\varphi)\cdot U = -\LieD_{u}(\rho\vol) = - (\nabla \rho\cdot u + \rho \divv{u})\vol = - (\divv{\rho u})\vol .
\end{equation}
Thus, we see that a vector $u\circ\varphi \in T_{\varphi}\Diff(\RR^{n})$ is in the kernel of $D\pi(\varphi)$ if and only if $\divv{\rho u} = 0$.
This leads us to the
the \emph{vertical distribution} 
\begin{equation}\label{eq:vert_dist_omt_infinite}
  \Ver_{\varphi} = \{ u\circ\varphi\in C^{\infty}(\RR^{n},\RR^{n})\mid \divv{\rho u}=0,\;\rho=\abs{D\varphi^{-1}}\rho_0\circ\varphi^{-1} \}.
\end{equation}
Geometrically, $\Ver_{\varphi}$ is the tangent space of the fiber going through $\varphi$.
Notice that the vertical distribution is defined without reference to a Riemannian metric: it is given solely by the principal bundle structure~\eqref{eq:principal_bundle_omt}.
Also notice that $u\circ\varphi \in \Ver_{\varphi}$ implies $u \in T_\id \Diff_{\rho\vol}(\RR^{n})$, i.e., $u$ is an element of the Lie algebra of $\Diff_{\rho\vol}(\RR^{n})$.

Let us now return to Riemannian geometry and the metric~\eqref{eq:L2noninvariant}.
The point is that the Riemannian metric $\mathcal G$ on $\Diff(\RR^{n})$ is compatible with the principal bundle~\eqref{eq:principal_bundle_omt}.
That is,~$\mathcal G$ is invariant under the action from the right of $\Diff_{\mu_0}(\RR^{n})$ on $\Diff(\RR^{n})$
\begin{equation}\label{eq:invariance_Wasserstein}
  \mathcal G_{\varphi}(\dot\varphi,\dot\varphi) = \mathcal G_{\varphi\circ\psi}(\dot\varphi\circ\psi,\dot\varphi\circ\psi), \quad\forall\, \psi\in \Diff_{\mu_0}(\RR^{n}).
\end{equation}
This means that $\mathcal G$ induces a Riemannian metric $\bar{\mathcal G}$ on $\Diff(\RR^{n})/\Diff_{\mu_0}(\RR^{n}) \simeq\Dens(\RR^{n})$.
Let us now compute what it is.

We first need
the orthogonal complement of the vertical distribution with respect to the Riemannian metric~\eqref{eq:L2noninvariant}.
This is the \emph{horizontal distribution}, given by
\begin{equation}\label{eq:horizontal_omt}
  \Hor_{\varphi} = \{ \nabla f\circ\varphi \mid f\in C^{\infty}(\RR^{n}) \}.
\end{equation}
\reviewertwo{
  The decomposition into vertical and horizontal bundle is well known. In the article it is only shown that (8) is a subset of the horizontal bundle, not that it is exhaustive. Maybe refer to Brenier’s original comment about the Helmholtz decomposition. \\[2ex]
  DONE! (added reference to the classical Helmholtz decomposition of vector fields)
}%
Indeed, if $u\circ\varphi\in \Ver_{\varphi}$, then
\begin{equation}
  \mathcal{G}_{\varphi}(u\circ\varphi,\nabla f\circ\varphi) = \int_{\RR^{n}} \nabla f\cdot u \,\rho\vol = \int_{\RR^{n}}-f\underbrace{\divv{\rho u}}_{0}\, \vol = 0.
\end{equation}
That every horizontal vector is of the form in \eqref{eq:horizontal_omt} follows from the classical Helmholtz decomposition of vector fields on $\RR^{n}$.

The vertical distribution depends on the reference volume form $\mu_0$.
A curious property of the Wasserstein geometry is that the horizontal distribution \emph{does not} depend on~$\mu_0$.

\reviewerone{
  Also on page 5, the formula for the vertical distribution is not completely obvious, so the author should either derive it or cite it. Here I thought the whole discussion on submersions felt rather rushed compared to the discussion on the metric and its geodesics and the group actions. For example the author mentions “from general results on Riemannian submersions” without a citation, and these are not completely standard results even in Riemannian geometry. \\[2ex]
  DONE! The full metric is now derived and there is a longer discussion about the distributions.
}

Since $\Hor_\varphi$ is transversal to the kernel of $D\pi(\varphi)$, it follows that
\begin{equation}
  D\pi(\varphi)\colon \Hor_{\varphi}\to T_{\pi(\varphi)}\Dens(\RR^{n})
\end{equation}
is an isomorphism.
The Riemannian metric $\bar{\mathcal G}$ is then defined as
\begin{equation}
  \bar{\mathcal G}_{\mu}(\dot\mu,\dot\mu) = \mathcal G_{\varphi}(D\pi(\varphi)^{-1}\cdot\dot\mu,D\pi(\varphi)^{-1}\cdot\dot\mu), \qquad \varphi\in \pi^{-1}(\mu).
\end{equation}
Due to the invariance~\eqref{eq:invariance_Wasserstein} the definition is independent of the choice of $\varphi\in \pi^{-1}(\mu)$.

To compute $D\pi(\varphi)^{-1}\cdot\dot\mu$ it follows from~\eqref{eq:horizontal_omt} that we need to find $\nabla f$ that fulfills
\begin{equation}
  D\pi(\varphi)\cdot \nabla f \circ \varphi = \dot \mu
\end{equation}
From~\eqref{eq:deriv_pi} this equation is given by
\begin{equation}
  -\divv{\rho \nabla f} = \dot\rho,
\end{equation}
where $\dot\mu = \dot\rho\vol$.
The operator $\Delta_{\rho} \coloneqq \divv \rho \nabla$ is an elliptic differential operator, invertible up to addition of constants.
Thus, $\nabla f = -\nabla \Delta_{\rho}^{-1}\dot\rho$ and $\bar{\mathcal G}$ is given by
\begin{eqnarray*}
  \bar{\mathcal G}_{\mu}(\dot\mu,\dot\mu)
    &= \mathcal G_{\varphi}(\nabla \Delta_{\rho}^{-1}\dot\rho\circ\varphi,\nabla \Delta_{\rho}^{-1}\dot\rho\circ\varphi)
\end{eqnarray*}
\begin{equation}\label{eq:wasserstein_metric_smooth_omt}
  \begin{split}
    &= \int_{\RR^{n}}\norm{\nabla \Delta_{\rho}^{-1}\dot\rho\circ\varphi}^2 \vol \\
    &= \int_{\RR^{n}}\norm{\nabla \Delta_{\rho}^{-1}\dot\rho}^2 \mu \\
    &= \int_{\RR^{n}}\pair{\nabla \Delta_{\rho}^{-1}\dot\rho,\rho\nabla \Delta_{\rho}^{-1}\dot\rho} \vol \\
    &= \int_{\RR^{n}}(-\Delta_{\rho}^{-1}\dot\rho)\dot\rho\, \vol.
  \end{split}
\end{equation}
Notice how the dependence on $\varphi$ is removed (by the change of variables in the third equality), as expected from the geometric considerations.
\reviewertwo{
  eq. (7): was the symbol $\rho$ introduced before? \\[2ex]
   FIXED!
}%
\reviewerone{
  In the middle of page 5, the 3-line computation that begins with “Consequently,” seems to come out of nowhere and does not make sense, since the author has not said anything about horizontal or vertical vectors yet; this looks like the result of a last-minute rearrangement of the paper. \\[2ex]
  FIXED!
}%
\reviewertwo{
  eq. (7) may be a bit much for someone not already knowing this result. A reference to a more detailed examination of the tangent space at a density may help (e.g. the relevant part of the aforementioned article by Otto). This will also simplify understanding the last equation before Remark 2.1, $f = \Delta^{-1}\dot{\bar\gamma}(0)$. \\[2ex]
  DONE! (A full derivation is now given)
}%

Using the language of geometry, the projection $\pi$ is a \emph{Riemannian submersion} between $(\Diff(\RR^{n}),\mathcal G)$ and $(\Dens(\RR^{n}),\bar{\mathcal G})$.
(See~\cite[\S~3.5]{Pe2006} for details on Riemannian submersions.)
\reviewertwo{
  “From general results on Riemannian submersions it follows...”: this should come with a reference. \\[2ex]
  DONE!
}%
Now, a general result on Riemannian submersions by Hermann~\cite{He1960} states that a geodesic curve $\gamma(t)$ such that $\dot\gamma(0)\in\Hor_{\gamma(0)}$ also fulfills $\dot\gamma(t)\in \Hor_{\gamma(t)}$ for any~$t$.
That is, $\gamma(t)$ remains tangential to the horizontal distribution.
Naturally, such geodesics are called \emph{horizontal}.
Geodesics on $(\Diff(\RR^{n}),\mathcal G)$ and $(\Dens(\RR^{n},\bar{\mathcal G})$ are related as follows.
\begin{enumerate}
  \item If $\gamma(t)$ is a horizontal geodesic on $\Diff(\RR^{n})$ then $\bar\gamma(t) \coloneqq \pi(\gamma(t))$ is a geodesic on $\Dens(\RR^{n})$.

  \item Conversely, if $\bar\gamma(t)$ is a geodesic on $\Dens(\RR^{n})$, then for every $\varphi_0\in\pi^{-1}(\bar\gamma(0))$ there exists a unique horizontal geodesic $\gamma(t)$ on $\Diff(\RR^{n})$ such that $\bar\gamma(t) = \pi(\gamma(t))$ and $\gamma(0)=\varphi_0$.

\end{enumerate}

Since the shortest path between the identity $\id$ and the fiber $\pi^{-1}(\mu_1)$ must be horizontal (otherwise we can make it shorter), it follows that \autoref{prob:smooth_omt} reduces to finding a geodesic $\bar\gamma(t)$ on $\Dens(\RR^{n})$ with $\bar\gamma(0)=\mu_0$ and $\bar\gamma(1)=\mu_1$.
The solution $\varphi$ is then obtained as the endpoint of the corresponding horizontal geodesic $\gamma(t)$ with $\gamma(0)=\id$.



Let us now explicitly explore horizontal solution geodesics of \autoref{prob:smooth_omt}.
From~\eqref{eq:geodesic_eq_omt} and the fact that $\Hor_{\id} = \{ \nabla f\mid f\in C^{\infty}(\RR^{n} \}$ it follows that such a curve is of the form
\begin{equation}\label{eq:hor_geodesics_omt_smooth}
  \gamma(t)(x) = x + t \nabla f(x) = \nabla (\underbrace{\frac{1}{2}\norm{x}^{2} + t f(x)}_{\phi(t,x)}).
\end{equation}
Since $\gamma(1) \in \pi^{-1}(\mu_1)$ we get a condition on $\phi(x)\coloneqq \phi(1,x)$ that reads
\begin{equation}
  (\nabla\phi)_*\mu_0 = \mu_1 \iff (\nabla\phi)^*\mu_1 = \mu_0. 
\end{equation}
Expressed in the variables $\rho_0$ and $\rho_1$, this is the Monge--Ampère equation for $\phi$
\begin{equation}\label{eq:monge_ampere}
  \det(\nabla^2\phi(x)) = \frac{\rho_0(x)}{\rho_1(\nabla\phi(x))}.
\end{equation}
Thus we recover the long-established result that \autoref{prob:smooth_omt} is equivalent to the Monge--Ampère equation.

Let us now summarize this section.
The flat Riemannian metric~\eqref{eq:L2noninvariant} on $\Diff(\RR^{n})$ induces a Riemannian metric on the space of probability densities $\Dens(\RR^{n})$.
\autoref{prob:smooth_omt} then becomes a standard geodesic problem: on the Riemannian manifold $(\Dens(\RR^{n}),\bar{\mathcal{G}})$ find the shortest curve $\bar\gamma(t)$ such that $\bar\gamma(0)=\mu_0$ and $\bar\gamma(1)=\mu_1$.
Once such a curve is found, the solution $\varphi$ to \autoref{prob:smooth_omt} is the endpoint of the corresponding horizontal geodesic $\gamma(t)$, given by~\eqref{eq:hor_geodesics_omt_smooth} with $f = -\Delta_{\rho_0}^{-1}\dot{\bar\gamma}(0)$.


\rednotes{



\begin{remark}
  It will be a formal presentation, although it can be made rigorous, for example by working instead of $\RR^n$ on a compact manifold, either using the Banach manifold category (as in Ebin and Marsden~\cite{EbMa1970}) or the tame Fréchet manifold category (as in Hamilton~\cite{Ha1982}).
\end{remark}
}

\begin{remark}
  The setup in this section can be made completely geometric, in that $\RR^{n}$ can be exchanged for any Riemannian manifold $M$.
  The key is that the geodesic equation~\eqref{eq:geodesic_eq_omt} instead becomes the point-wise geodesic equation of $M$, and the explicit formula \eqref{eq:hor_geodesics_omt_smooth} becomes
  \begin{equation}
    \gamma(t)(x) = \exp_x(t\nabla f(x)),
  \end{equation}
  where $\exp_x$ denotes the Riemannian exponential on $M$.
  For details we refer to McCann~\cite{Mc2001}, who developed optimal transport on Riemannian manifolds.
  The extension of Otto's geometric framework to Riemannian manifolds is discussed by Lott~\cite{Lo2008} and by Vilani~\cite{Vi2009}.
\end{remark}

\reviewerone{
  In Remark 2.1 the author mentions that the setup makes sense on a manifold, but certainly equations (9) and (10) do not make sense since gradients cannot be identified with vector fields in that setting. It would make sense to cite McCann’s work here.
  \\[2ex]
  DONE! Now explaining a bit more, and added citations.
}

\subsection{Polar decomposition of diffeomorphisms}\label{sub:polar_decomposition_smooth}

We now show how the geometry of OMT gives rise to the polar decomposition of maps, obtained by Brenier~\cite{Br1991}.
We shall also discuss different dynamical formulations, aiming to recover the polar decomposition as limits of gradient flows.

Let us first state the result.
\begin{theorem}[Polar decomposition of diffeomorphisms]\label{thm:polar_decomposition_smooth}
  Let $\varphi\in\Diff(\RR^{n})$ and $\mu_0\in\Dens(\RR^{n})$.
  Then there exists a strictly convex function $\phi\in C^{\infty}(\RR^{n})$, unique up to addition of a constant, and a unique diffeomorphism $\psi\in\Diff_{\mu_0}(\RR^{n})$ such that
  \begin{equation}\label{eq:polar_decomposition_diffeos}
    \varphi = \nabla\phi \circ \psi .
  \end{equation}
  The diffeomorphism $\nabla\phi$ is the unique solution of \autoref{prob:smooth_omt} with $\mu_1=\pi(\varphi)$.
\end{theorem}

To prove this result we need two lemmas. 

The subset $K_{\diamondsuit} \subset \Diff(\RR^{n})$ of diffeomorphisms connected to the identity by horizontal geodesics is called the \emph{polar cone}.
Thus, $\varphi\in K_{\diamondsuit}$ if and only if there is a horizontal geodesic $\gamma(t)$ on $\Diff(\RR^{n})$ such that $\gamma(0)=\id$ and $\gamma(1)=\varphi$.
\begin{lemma}\label{lem:polar_cone_isomorphism}
  Up to addition of constants, the mapping $\phi \mapsto \nabla\phi$ is an isomorphism between the space of smooth strictly convex functions on $\RR^{n}$ and the polar cone $K_{\diamondsuit}$.
  The polar cone itself is a convex subset of $(\Diff(\RR^{n}),\mathcal G)$.
\end{lemma}
\reviewertwo{
  In Lemma 2.3 strict convexity of $\phi$ seems to be required. It is shown that any diffeomorphism can be induced by a potential $\phi$. For the sake of completeness maybe add that any $\phi$ induces a diffeomorphism. \\[2ex]
  DONE!
}

\begin{proof}
  From~\eqref{eq:hor_geodesics_omt_smooth} it follows that elements in $K_{\diamondsuit}$ are of the form $\nabla\phi$ for some $\phi\in C^{\infty}(\RR^{n})$.
  Since $\nabla\phi$ is a diffeomorphism,
  \begin{equation}
    \det(\nabla^2\phi(x)) \neq 0, \quad \forall\;x\in\RR^{n}.
  \end{equation}
  Consequently, the symmetric matrix $\nabla^2\phi(x)$ has only non-zero eigenvalues.
  Let
  \begin{equation}
    \phi(t,x) \coloneqq \frac{1-t}{2}\norm{x}^{2} + t\,\phi(x).
  \end{equation}
  By definition of $K_{\diamondsuit}$, the path $\gamma(t)=\nabla\phi(t,\cdot)$ is a horizontal geodesic on $\Diff(\RR^{n})$.
  In particular, $t\mapsto\nabla^2\phi(t,x)$ is a continuous path of non-degenerate symmetric matrices.
  Since the eigenvalues of $\nabla^2\phi(t,x)$ are positive for $t=0$, they remain positive for any~$t \in [0,1]$.
  Thus, $\phi(1,\cdot)=\phi$ is a strictly convex function.

  Now, if $\phi$ is a given strictly convex function, then $\gamma(t) = (1-t)\id + t\nabla\phi$ is a horizontal geodesic curve in $\Diff(\RR^{n})$ originating from the identity.
  By the definition of $K_{\diamondsuit}$, it thereby follows that $\nabla\phi=\gamma(1)\in K_{\diamondsuit}$.

  Convexity of $K_{\diamondsuit}$ follows since a convex combination of two strictly convex functions is again strictly convex.
\end{proof}


The second lemma is the following non-trivial result.
\reviewerone{
  Lemma 2.4 is not proved, and the statement “the following non-trivial result, stated without a proof” immediately makes the reader pause. It seems necessary to make the author’s approach to decompositions work; glossing over it seems to negate the idea that one can prove matrix decompositions using only Riemannian geometry.
  \\[2ex]
  DONE! We now give a discussion of how a geometric proof could look like (to carry out this in detail for the infinite-dimensional case is left as a future topic.)
}%
\reviewertwo{
  Lemma 2.4 is the crucial non-trivial ingredient for Theorem 2.2. If it is stated without proof, a reference and a more detailed discussion should be added.
  \\[2ex]
  DONE! (see remark above)
}%
\begin{lemma}\label{lem:shortest_geodesic}
  The polar cone $K_{\diamondsuit}$ is a \emph{section} of the principal bundle~\eqref{eq:principal_bundle_omt}.
  That is, the mapping
  \begin{equation}
    \pi\big|_{K_{\diamondsuit}}\colon K_{\diamondsuit} \to \Dens(\RR^{n})
  \end{equation}
  is an isomorphism.
\end{lemma}

  This lemma follows from the work of Caffarelli~\cite{Ca1992b} on regularity of solutions of the Monge--Ampére equation (see also \cite[Ch.\!~12]{Vi2009} for a wider discussion about regularity).
  Caffarelli's proof, however, is not based on the Riemannian geometry considered here, but rather on PDE analysis techniques.
  In the linear, finite-dimensional category of optimal transport in \autoref{sub:omt_linear} below, the analog of \autoref{lem:shortest_geodesic} is proved geometrically by showing existence and uniqueness of a lifted gradient flow.
  We conjecture that the same proof technique can be used also in the smooth, infinite-dimensional category of optimal transport.
  A careful investigation of this, however, is outside the scope here and left for future work (see \autoref{subsub:limit} for a brief justification of the conjecture).


We are now ready to prove \autoref{thm:polar_decomposition_smooth}.
\reviewerone{
  In Theorem 2.2 and its proof the author switches back and forth between $\mu$ and $\mu_1$, which is slightly confusing. Also there is a “corresponding unique element $\nabla\phi \in K$” such that $\pi(\nabla\phi) = \mu$; it looks strange to not have the defining condition that ensures uniqueness mentioned here.
  \\[2ex]
  DONE! $\mu_1$ is now used consistently, and the defining condition for uniqueness is mentioned.
}

\begin{proof}[Proof of \autoref{thm:polar_decomposition_smooth}]
  Let $\mu_1=\pi(\varphi)$.
  From \autoref{lem:shortest_geodesic} we get that $\pi|_{K_{\diamondsuit}}\to \Dens(\RR^{n})$ is a bijection, so there is a unique $\nabla\phi\in K_{\diamondsuit}$ such that $\pi(\nabla\phi) = \mu_1$.
  From \autoref{lem:polar_cone_isomorphism} it follows that $\phi$ is strictly convex.
  \rednotes{
  A shortest geodesic $\bar\gamma(t)$ between $\vol$ and $\mu_1$.
  Let $\gamma(t)$ be the corresponding horizontal geodesic on $\Diff(\RR^{n})$ with $\gamma(0)=\id$.
  Then $\gamma(t)$ is the shortest curve between $\id$ and the fiber $\pi^{-1}(\mu_1)$, so $\gamma(1)$ solves \autoref{prob:smooth_omt} with $\mu_0=\vol$.
  From the geometry described in \autoref{sub:geometry_omt_smooth} it follows that $\gamma$ is of the form
  \begin{equation}
    \gamma(t)(x) = \nabla\big(\frac{1}{2}\norm{x}^{2} + t f(x)\big).
  \end{equation}
  In particular, $\gamma(1)(x) = \nabla\phi$ with $\phi = \frac{1}{2}\norm{x}^{2} + f(x)$.
  Since $\gamma(t) \in \Diff(\RR^{n})$ it follows that $D\gamma(t)$ is symmetric and has non-zero eigenvalues for every~$t$, and since $D\gamma(0)$ has only positive eigenvalues and $\gamma(t)$ is continuous in $t$, $D\gamma(1)=\nabla^2\phi$ also has positive eigenvalues, so it is positive definite, i.e., convex.
  }
  To obtain the actual decomposition, we notice that by construction $\varphi$ and $\nabla\phi$ belong to the same fiber $\pi^{-1}(\mu_1)$, so $\psi\coloneqq (\nabla\phi)^{-1}\circ\varphi$ is an element of $\Diff_{\mu_0}(\RR^{n})$.
  Hence, we arrive at the decomposition
    $\varphi = \nabla\phi\circ\psi$.
  From the geometry described in \autoref{sub:geometry_omt_smooth} it follows that a solution to \autoref{prob:smooth_omt} must be the endpoint of a horizontal geodesic from the identity, i.e., an element in $K_{\diamondsuit}$.
  The last assertion then follows since $\nabla\phi$ is unique.
\end{proof}

The geometric insights of OMT suggest the study of several \emph{Riemannian gradient flows}.
We consider three different types.
\begin{description}
  \item[Vertical gradient flow] This is a gradient flow restricted to the fibers of the principal bundle \eqref{eq:principal_bundle_omt}.
  The flow is constructed so that the element $\nabla\phi$ in \autoref{thm:polar_decomposition_smooth} is an equilibrium.
  \item[Entropy type gradient flow] This is a gradient flow on $\Dens(\RR^{n})$ with respect to the Riemannian metric \eqref{eq:wasserstein_metric_smooth_omt}.
  We shall use relative entropy as functional, but other choices are also possible (for example the family of porous medium potentials discussed by Otto~\cite{Ot2001}).
  \item[Lifted gradient flow] This is the lifting of an entropy type gradient flow to the polar cone.
  As mentioned, a strategy for a geometric proof of \autoref{lem:shortest_geodesic} is to establish existence and uniqueness of a limit.
  An interesting computational approach for \autoref{prob:smooth_omt} is to discretize the lifted gradient flow numerically. 
\end{description}
An illustration of the different types of gradient flows is given in \autoref{fig:polar_omt_smooth}.
Let us continue with a more detailed description of each type of flow.

\begin{figure}
\centering

\pgfmathsetmacro{\hoffset}{1}
\pgfmathsetmacro{\voffset}{0.35}
\pgfmathsetmacro{\gradscale}{0.45}
\pgfmathsetmacro{\bulletra}{0.05}
\begin{tikzpicture}[scale=1,>=latex]
  \coordinate (id) at (0,0);
  \coordinate (hor) at (4,0);
  \coordinate (p) at (4,2);
  \coordinate (psi) at (0,2);

  \draw[gray] (-\hoffset,0) node[left,black] {polar cone}  -- (4+\hoffset,0);
  \draw[gray] (4,-\voffset) -- (4,2+\voffset) node[rotate=-90,left,black] {\vphantom{pd}fiber};
  \draw[gray] (0,-\voffset) -- (0,2+\voffset) node[rotate=-90,left,black] {\vphantom{pd}fiber};

  \node[below right] at (id) {$\id$};
  \node[below right] at (hor) {$\nabla\phi$};
  \node[right] at (p) {$\varphi=\nabla\phi\circ\psi$};
  \node[right] at (psi) {$\psi$};


  \draw[black,dotted] (4,2-\gradscale*2) -- (4-\gradscale*4,2-\gradscale*2); 
  \draw[thick,black,->] (4,2) -- (4,2-\gradscale*2); 
  \draw[thick,gray,->] (4,2) -- (4-\gradscale*4,2-\gradscale*2) node[left,black] {$-\nabla_{\mathcal G} J$};

  \draw[thick,black,->] (0,0) -- (1,0) node[above,black] {$X$};

  \draw[fill] (id) circle [radius=\bulletra];
  \draw[fill] (hor) circle [radius=\bulletra];
  \draw[fill] (p) circle [radius=\bulletra];
  \draw[fill] (psi) circle [radius=\bulletra];

  \draw[gray] (-\hoffset,-2) node[left,black] {\vphantom{pd}densities}  -- (4+\hoffset,-2);
  \draw[gray,dashed] (0,-\voffset) -- (0,-2);
  \draw[gray,dashed] (4,-\voffset) -- (4,-2);
  \draw[fill] (0,-2) circle [radius=\bulletra] node[below] {$\mu_0$};
  \draw[fill] (4,-2) circle [radius=\bulletra] node[below] {$\mu_1$};
  \draw[thick,black,->] (0,-2) -- (1,-2) node[above,black] {$\bar X$};

  \draw[black,->] (2,-0.5) -- (2,-1.5) node[midway,right] {$\pi$};

\end{tikzpicture}
\caption{
Illustration of the geometry of the polar decomposition of diffeomorphisms.
The element $\nabla\phi$ in the factorization $\varphi=\nabla\phi\circ\psi$ is obtained at the intersection of the polar cone and the fiber of $\mu_1=\pi(\varphi)$.
To compute $\nabla\phi$, one may start at $\varphi$ and follow a gradient flow constrained to the fiber of $\mu_1$ (vertical gradient flow, see \autoref{subsub:vertical_gradient_flow_smooth}), or one may take a gradient flow of a functional on the space of densities that approaches $\mu_1$ (entropy gradient flow, see \autoref{subsub:entropy_gradient_flow_smooth}) and lift it to a corresponding gradient flow on the polar cone (lifted gradient flow, see \autoref{subsub:lifted_gradient_flow_smooth}).
}
\label{fig:polar_omt_smooth}
\end{figure}
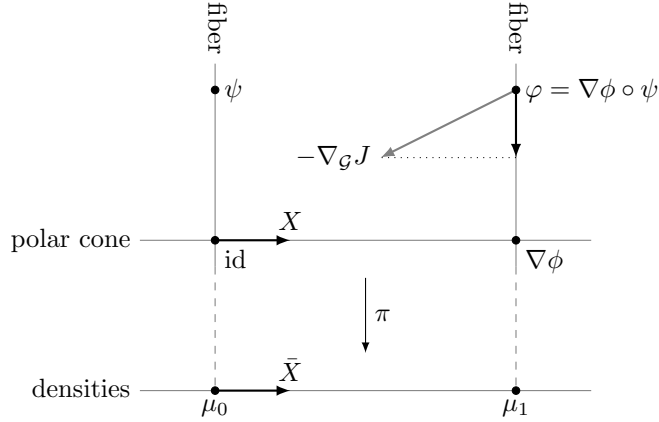



\reviewertwo{
  The gradient flows presented in Sections 2.2.1 and 2.2.2 are very interesting ideas. However, in the given presentation it would probably take the reader a long time to work out all the explicit formulas (I did not check the details of all presented steps, although the rough outline seems ok.).
  \\[2ex]
  DONE! Many more details are now given for all types of flows considered.
}

\reviewertwo{
  The vertical gradient flow 2.2.1 (or something very similar) seems to be discussed in [2, 1]. So it would help to add a few more intermediate steps in the presentation and discuss the relation with [2, 1] \texttt{papers by Haker, Tannenbaum, et. al.}.
  \\[2ex]
  DONE! A careful comparison with \cite*{AnHaTa2003} is now given in a remark.
}

\subsubsection{Vertical Gradient Flow}\label{subsub:vertical_gradient_flow_smooth}

Recall that the solution to \autoref{prob:smooth_omt} is obtained as the point on the fiber of $\mu_1$ that is closest to the identity.
Because we have the closed form expression $J$, given by~\eqref{eq:omt_functional}, for the distance from any diffeomorphism to the identity, it is natural to consider the constrained gradient flow
\begin{equation}\label{eq:gradient_flow_diffeos}
  \dot\eta = -\mathrm{Pr}_{\Ver}\nabla_{\mathcal G}J(\eta) ,\quad \eta(0) = \varphi
\end{equation}
where $\nabla_{\mathcal G}$ denotes the gradient with respect to $\mathcal G$ and $\mathrm{Pr}_{\Ver}$ denotes orthogonal projection onto the vertical distribution.
This gives us a gradient flow on the constraint manifold $\mathcal C(\mu_0,\pi(\varphi))=\pi^{-1}(\pi(\varphi))$ for which $\nabla\phi$ in the polar decomposition~\eqref{eq:polar_decomposition_diffeos} is an equilibrium.
Since $\mathcal{G}_{\eta}(\nabla_{\mathcal{G}}J(\eta),\dot\eta) = \mathcal{G}_{\eta}(\eta-\id,\dot\eta)$ it follows that \eqref{eq:gradient_flow_diffeos} becomes
\begin{equation}\label{eq:gradient_flow_diffeos2}
  \dot\eta + 2(\eta - \id) = - \nabla p \circ\eta, \qquad \eta_*\mu_0 = \pi(\varphi),
\end{equation}
where the smooth function $p$ is the Lagrangian multiplier enforcing $\eta$ to remain on the constraint manifold $\mathcal C(\mu_0,\pi(\varphi))$.
In particular, the constraint ensures that $\eta$ is always a diffeomorphism.

We mention that the term $\eta - \id$ in \eqref{eq:gradient_flow_diffeos2} should be interpreted as a tangent vector in $T_\eta\Diff(\RR^{n})$; the Riemannian notation would be $\log_{\eta}(\id) = \eta-\id$, where $$\log_\eta\colon\Diff(\RR^{n})\to T_\eta\Diff(\RR^{n})$$ is the inverse of the Riemannian exponential $$\exp_{\eta}\colon T_{\eta}\Diff(\RR^{n})\to \Diff(\RR^{n}).$$

\reviewerone{
  {\color{gray} At the top of page 13 I expected to see some mention of how one computes S} (and similarly how one would compute p in the corresponding situation at the end of page 8). {\color{gray} Also this would be an excellent place to write what the equation looks like in a low-dimensional concrete example.} \\[2ex]
  DONE! An equation for $p$ is now given.
}%

Let us now turn to the Lagrange multiplier $p$ in \eqref{eq:gradient_flow_diffeos2}.
Since $\dot\eta \in \Ver_\eta$, it follows from \eqref{eq:vert_dist_omt_infinite} that
\begin{equation}
  \divv{\rho_1 u} = 0 ,
\end{equation}
where $u=\dot\eta\circ\eta^{-1}$ and $\mu_1 = \rho_1 \vol$.
Composing \eqref{eq:gradient_flow_diffeos2} by $\eta^{-1}$ from the right, and applying $\divv{\rho_1\,\cdot\,}$, we then obtain an equation for $p$, namely
\begin{equation}\label{eq:lag_vert_flow_omt}
  \divv{\rho_1 \nabla p} = \divv{\rho_1 (\id - \eta^{-1})}. 
\end{equation}

We may write equation \eqref{eq:gradient_flow_diffeos2} in a ``fluid formulation'', using the right reduced variable $u=\dot\eta\circ\eta^{-1}$.
Indeed, composing \eqref{eq:gradient_flow_diffeos2} from the right by $\eta^{-1}$ leads to
\begin{equation}\label{eq:gradient_flow_diffeos3}
  \begin{split}
    \dot \eta &= u \circ\eta \\
    u &= \eta^{-1} - \id - \nabla p \\
    0 &= \divv{\rho_1 u } ,
  \end{split}
\end{equation}
where $p$, as before, is the Lagrange multiplier given by the solution of \eqref{eq:lag_vert_flow_omt}.


\begin{remark}
  The idea of computing the optimal transport map by a flow along the fiber has been considered before, by Angenent, Haker, and Tannenbaum~\cite{AnHaTa2003}.
  Their flow, however, is not a gradient flow with respect to the metric \eqref{eq:L2noninvariant}.
  Instead, it goes as follows.

  Parameterize $\eta$ as $\eta(t) = \eta_0\circ\psi(t)^{-1}$, where $\psi(t)\in \Diff_{\mu_0}(\RR^{n})$ (this is always possible because of the principal bundle structure \eqref{eq:principal_bundle_omt}).
  Next, $\psi$ itself is given as the flow of a time-dependent vector field $v(t)$ with $\divv{\rho_0 v(t)}=0$.
  Thus,
  \begin{equation}\label{eq:v_tannenbaum_right}
    \frac{\ud}{\ud t}\psi(t) = v(t)\circ\psi(t),\quad \psi(0) = \id.
  \end{equation}

  \rednotes{
  Alternatively, we can use that
  \begin{equation}
    (\frac{\ud}{\ud t}\psi(t)^{-1})\circ \psi + D\psi^{-1}\circ\psi\cdot\dot\psi = 0
    \iff
    (\frac{\ud}{\ud t}\psi(t)^{-1}) = - D\psi^{-1}\cdot\dot\psi \circ\psi^{-1}
  \end{equation}
  which together with \eqref{eq:v_tannenbaum_right} give
  \begin{equation}
    \frac{\ud}{\ud t}\psi(t)^{-1} = - D\psi(t)^{-1}\cdot v(t) , \quad \psi(0) = \id.
  \end{equation}
  }

  So far we have just carried out a change of variables $(\eta,\dot\eta) \leftrightarrow (\psi,v)$; to give the flow studied in \cite{AnHaTa2003} we need to specify what $v(t)$ is.
  The equation for $v(t)$ in \cite{AnHaTa2003} is
  \begin{equation}\label{eq:tannenbaum_flow}
    \begin{split}
      v(t) &= \eta(t) - \id - \nabla q, \\
      0 &= \divv{\rho_0 v}
    \end{split}
  \end{equation}
  where $q$ is the Lagrange multiplier corresponding to the Helmholtz projection.

  Now let us compare \eqref{eq:tannenbaum_flow} with \eqref{eq:gradient_flow_diffeos2}.
  To this extent, we need to see how the flow \eqref{eq:gradient_flow_diffeos2} looks like in the variables $(\psi,v)$.
  Differentiating $\eta(t)\circ\psi(t) = \eta_0$ we get
  \begin{align}
    \dot\eta(t)\circ\psi(t) + D\eta(t)\circ\psi(t)\cdot \dot\psi(t) &= 0
    \\ &\Updownarrow \\
    \dot\eta(t) + D\eta(t)\cdot (\dot\psi(t)\circ\psi(t)^{-1}) &= 0
    \\ &\Updownarrow \\
    u(t) + D\eta(t)\cdot v(t)\circ\eta(t)^{-1} &= 0
  \end{align}
  Thus, $u$ and $v$ are related by minus conjugation by $\eta(t)$.
  Using the fluid formulation \eqref{eq:gradient_flow_diffeos3} we then get
  \begin{equation}
    v(t) = (D\eta(t))^{-1}\cdot \Big(\eta(t) - \id + \nabla p \circ\eta(t)\Big).
  \end{equation}
  Notice that this choice of $v(t)$ is different from \eqref{eq:tannenbaum_flow}.
\end{remark}

\subsubsection{Entropy Gradient Flow} \label{subsub:entropy_gradient_flow_smooth}
In this section we consider gradient flows on $\Dens(\RR^{n})$ with respect to the Wasserstein Riemannian metric \eqref{eq:wasserstein_metric_smooth_omt}.
Such flows are studied by Jordan, Kinderlehrer, and Otto~\cite{JoKiOt1998} for the entropy functional (giving the Fokker--Planck equation) and later by Otto~\cite{Ot2001} for a more general class of functionals (giving porous medium equations).
Here, we focus on entropy as in \cite{JoKiOt1998}.

Let $\mu_1=\pi(\varphi)$ and take as potential function $H(\mu)$ the \emph{entropy of $\mu$ relative to $\mu_1$}, given by
\begin{equation}\label{eq:relative_entropy}
  H(\mu) = -\int_{\RR^{n}} \frac{\mu}{\mu_1}\log\left( \frac{\mu}{\mu_1}\right) \mu_1 .
\end{equation}
It is also called the \emph{Kullback--Leibler divergence}, especially in information theoretic contexts.
Differentiation with respect to time yields
\begin{equation}
  \begin{split}
    \frac{\ud}{\ud t}H(\mu) &= -\int_{\RR^{n}} \dot\mu\log\left( \frac{\mu}{\mu_1}\right) -\underbrace{\int_{\RR^{n}} \dot\mu}_{0} \\
    &= - \int_{\RR^{n}}\left(\Delta_{\rho}^{-1}\Delta_{\rho}\log\left( \frac{\mu}{\mu_1}\right)\right) \dot\mu \\
    &= \bar{\mathcal G}_{\mu}(\Delta_\rho \log\left( \frac{\mu}{\mu_1}\right)\vol, \dot\mu).
  \end{split}
\end{equation}
Since $\Delta_{\rho}\log(\mu/\mu_1) = \divv{\rho\nabla \log(\rho/\rho_1)} = \divv{\rho_1\nabla(\rho/\rho_1)}$, and since
\begin{equation}
  \int_{\RR^{n}} \divv{\rho_1\nabla(\rho/\rho_1)} \,\vol = 0
\end{equation}
from the divergence theorem, the Riemannian gradient flow
\begin{equation}
  \dot \mu = \nabla_{\bar{\mathcal G}} H(\mu)
\end{equation}
is given by
\begin{equation}\label{eq:dens_gradient_flow}
  \dot \mu = \divv{ \rho_1 \nabla \varrho} \;\vol
\end{equation}
where $\mu_1 = \rho_1\,\vol$ and $\mu = \varrho\mu_1$.
The flow strives toward the maximum of the relative entropy $H(\mu)$, which occurs at $\mu = \mu_1$.
Notice that if $\rho_1=1$ it becomes the standard heat flow.

The geometric insights of the flow \eqref{eq:dens_gradient_flow} can be used for analysis, in particular  the question of convergence towards a limit.
Indeed, by proving negativity of the Hessian of the relative entropy functional $H(\mu)$ with respect to the Wasserstein metric, Otto~\cite{Ot2001} was able to give exponential rates of convergence.
Of course, for the linear flow \eqref{eq:dens_gradient_flow} this can be achieved by standard PDE techniques, but Otto's geometric analysis also works for non-linear porous medium flows.

\todo[inline]{Say something about the geodesic approach for another choice of $\bar X$, using Fisher--Rao great circles.}

\subsubsection{Lifted gradient flow}\label{subsub:lifted_gradient_flow_smooth}
Here we are interested in constructing a gradient flow of diffeomorphisms, evolving on the polar cone $K_{\diamondsuit}$ such that its limit is the solution to \autoref{prob:smooth_omt}.
To do so, we consider lifting of the entropy gradient flow in \autoref{subsub:entropy_gradient_flow_smooth} with respect to the principal bundle \eqref{eq:principal_bundle_omt}.
First, define the lifted functional
\begin{equation}
  F\colon \Diff(\RR^{n}) \to \RR, \quad F(\varphi) \coloneqq H(\pi(\varphi)) = -\int_{\RR^{n}}\frac{\varphi_{*}\mu_0}{\mu_1} \log\left(\frac{\varphi_{*}\mu_0}{\mu_1} \right) \mu_1.
\end{equation}
By construction, $F$ is constant on the fibers, so its gradient $\nabla_{\mathcal G} F$ with respect to \eqref{eq:L2noninvariant} is orthogonal to the fibers: $\nabla_{\mathcal G}F(\varphi) \in \Hor_{\varphi}$.
Thus, the unconstrained gradient flow
\begin{equation}\label{eq:unconstrained_lifted_omt_smooth}
  \dot\varphi = \nabla_{\mathcal G}F(\varphi)
\end{equation}
traces an integral curve of the horizontal distribution.
Furthermore, since the projection $\pi$ is a Riemannian submersion, it follows that
\begin{equation}
  D\pi(\varphi)\cdot\nabla_{\mathcal G}F(\varphi) = \nabla_{\bar{\mathcal G}}H(\pi(\varphi)),
\end{equation}
so if $\varphi(t)$ is an integral curve of \eqref{eq:unconstrained_lifted_omt_smooth}, then $\mu(t) = \pi(\varphi(t))$ is an integral curve of the entropy gradient flow \eqref{eq:dens_gradient_flow}.
Since \eqref{eq:dens_gradient_flow} has $\mu_1$ as a limit, it follows that $\varphi(t)$ approaches the fiber $\pi^{-1}(\mu_1)$ as $t\to\infty$.
At first sight, it therefore looks promising to use the flow \eqref{eq:unconstrained_lifted_omt_smooth} with initial data $\varphi(0) = \id$ as a way to compute the solution to \autoref{prob:smooth_omt} (recall from \autoref{sub:geometry_omt_smooth} that the solution to \autoref{prob:smooth_omt} is a horizontal geodesic from $\id$ to $\pi^{-1}(\mu_1)$).
However, things are not quite that simple, because the horizontal distribution is \emph{not integrable}, so two different horizontal paths starting at $\id$ and ending at $\pi^{-1}(\mu_1)$ typically end up at \emph{different points} of the fiber $\pi^{-1}(\mu_1)$.

As a remedy we shall instead consider the lifted gradient flow \emph{constrained to the polar cone} $K_{\diamondsuit}$.
Notice that, in general, $T_{\varphi}K_{\diamondsuit} \neq \Hor_{\varphi}$, although $T_{\id}K_{\diamondsuit} = \Hor_{\id}$.
(We also know that $T_{\varphi}K_{\diamondsuit} \cap \Hor_{\varphi}$ is at least 1-dimensional, since $K_{\diamondsuit}$ consists of endpoints of horizontal geodesics.)
Consequently, $K_{\diamondsuit}$ is not invariant under the unconstrained gradient flow \eqref{eq:unconstrained_lifted_omt_smooth}: we need to consider the projection
\begin{equation}\label{eq:constrained_lifted_omt_smooth}
  \dot\varphi = \Pi_\varphi\nabla_{\mathcal G}F(\varphi), \quad \varphi(0) = \id,
\end{equation}
where $\Pi_{\varphi}\colon T_{\varphi}\Diff(\RR^{n})\to T_\varphi K_{\diamondsuit}$ denotes the orthogonal projection.
The flow \eqref{eq:constrained_lifted_omt_smooth} is then the Riemannian gradient flow of $F$ restricted to $K_{\diamondsuit}$ with respect to the Riemannian metric \eqref{eq:L2noninvariant} restricted to $K_{\diamondsuit}$.

Let us now work out \eqref{eq:constrained_lifted_omt_smooth} explicitly.
First, recall from \autoref{lem:polar_cone_isomorphism} that elements in the polar cone are of the form $\nabla \phi$ for a strictly convex, smooth function $\phi$.
Since
\begin{equation}
  \nabla\phi_{*}\mu_0 = \frac{\rho_0\circ (\nabla\phi)^{-1}\,\vol}{\det(\nabla^{2}\phi\circ (\nabla\phi)^{-1})},
\end{equation}
the functional $F$ restricted to $K_{\diamondsuit}$ is given by
\begin{equation}
  F|_{K_{\diamondsuit}}(\nabla\phi) = - \int_{\RR^{n}} \frac{\rho_0\circ (\nabla\phi)^{-1}}{\det(\nabla^{2}\phi\circ (\nabla\phi)^{-1})} \log\left(
    \frac{\rho_0\circ (\nabla\phi)^{-1}}{\rho_1\det(\nabla^{2}\phi\circ (\nabla\phi)^{-1})}
  \right) \, \vol ,
\end{equation}
where $\nabla^{2}\phi$ denotes the Hessian of $\phi$.
The change of variables induced by $\nabla\phi$ then gives
\begin{align}
  F|_{K_{\diamondsuit}}(\nabla\phi) &= - \int_{\RR^{n}} \frac{\rho_0}{\det(\nabla^{2}\phi)} \log\left(
    \frac{\rho_0}{(\rho_1\circ\nabla\phi)\det(\nabla^{2}\phi)}
  \right) \det(\nabla^{2}\phi)\, \vol
  \\
  &=
  - \int_{\RR^{n}} \log\left(
    \frac{\rho_0}{(\rho_1\circ\nabla\phi)\det(\nabla^{2}\phi)}
  \right) \mu_0 .
\end{align}
Now, take a curve $\nabla\phi = \nabla\phi(t)$ in $K_{\diamondsuit}$.
Then
\begin{align}
  \frac{\ud}{\ud t}F|_{K_{\diamondsuit}}(\nabla\phi)
  &= \int_{\RR^{n}} \frac{1}{\det(\nabla^{2}\phi)} \frac{\ud}{\ud t} \det(\nabla^{2}\phi)\, \mu_0
  + \int_{\RR^{n}}\frac{\ud}{\ud t}\log(\rho_1\circ\nabla\phi)\mu_0 \\ \label{eq:dFdt_calc}
  &= \int_{\RR^{n}} \tr\left((\nabla^{2}\phi)^{-1}\nabla^{2}\dot\phi\right)\, \mu_0
  + \int_{\RR^{n}}\frac{\nabla\rho_1\circ\nabla\phi \cdot \nabla\dot\phi}{\rho_1\circ\nabla\phi}\mu_0 \\
  &= \int_{\RR^{n}} \tr\left((\nabla^{2}\phi)^{-\top}\nabla^{2}\dot\phi\right)\, \mu_0
  + \int_{\RR^{n}}\frac{\nabla\rho_1\circ\nabla\phi \cdot \nabla\dot\phi}{\rho_1\circ\nabla\phi}\mu_0 \\
  &= -\int_{\RR^{n}} (\divv{\rho_0(\nabla^{2}\phi)^{-1}}) \cdot \nabla\dot\phi \, \vol
  + \int_{\RR^{n}}\frac{\nabla\rho_1\circ\nabla\phi \cdot \nabla\dot\phi}{\rho_1\circ\nabla\phi}\mu_0 \\
  &= -\int_{\RR^{n}} (\divv{(\nabla^{2}\phi)^{-1}}) \cdot \nabla\dot\phi \, \mu_0
  - \int_{\RR^{n}} ((\nabla^{2}\phi)^{-1}\nabla\rho_0) \cdot \nabla\dot\phi \, \vol \\
  &\phantom{=}\qquad\qquad + \int_{\RR^{n}}\frac{\nabla\rho_1\circ\nabla\phi \cdot \nabla\dot\phi}{\rho_1\circ\nabla\phi}\mu_0
\end{align}
where $\divv{}$ on matrices denotes the divergence operator applied rowwise.
From the definition \eqref{eq:L2noninvariant} of $\mathcal G$ it follows next that
\begin{align}
  \frac{\ud}{\ud t}F|_{K_{\diamondsuit}}(\nabla\phi) &=
  - \mathcal{G}_{\nabla\phi}(\divv{(\nabla^{2}\phi)^{-1}}, \nabla\dot\phi)
  - \mathcal{G}_{\nabla\phi}\left( \frac{(\nabla^{2}\phi)^{-1}\nabla\rho_0}{\rho_0},\nabla\dot\phi \right) \\
  &\qquad\qquad + \mathcal{G}_{\nabla\phi}\left( \frac{\nabla\rho_1\circ\nabla\phi}{\rho_1\circ\nabla\phi},\nabla\dot\phi \right).
\end{align}
Thus, the gradient is
\begin{equation}\label{eq:gradF_weak_omt_smooth}
  \nabla_{\mathcal G} F(\nabla\phi) =
    - \divv{(\nabla^{2}\phi)^{-1}}
    - \frac{(\nabla^{2}\phi)^{-1}\nabla\rho_0}{\rho_0}
    + \frac{\nabla\rho_1\circ\nabla\phi}{\rho_1\circ\nabla\phi} .
\end{equation}

Using \autoref{lem:polar_cone_isomorphism} we can represent the polar cone by strictly convex functions $\phi$, defined up to addition by constants.
The Riemannian metric $\mathcal G$ on $K_{\diamondsuit}$ then induces the Riemannian metric on the space of strictly convex functions
\begin{equation}
  \hat{\mathcal G}_{\phi}(\dot\phi,\dot\phi) \coloneqq \mathcal G_{\nabla\phi}(\nabla\dot\phi,\nabla\dot\phi)
   = \int_{\RR^{n}}\nabla\dot\phi \cdot\nabla\dot\phi\; \mu_0 = - \int_{\RR^{n}} (\divv{ \rho_0 \nabla\dot\phi}) \dot\phi \,\vol.
\end{equation}
Likewise, $F$ induces the functional
\begin{equation}\label{eq:Fhat_omt_smooth}
  \hat F(\phi) \coloneqq F(\nabla\phi).
\end{equation}

Now, the constrained lifted gradient flow \eqref{eq:constrained_lifted_omt_smooth}, written in the variable $\phi$, is given by
\begin{equation}
  \dot\phi = \nabla_{\hat{\mathcal G}}\hat F(\phi).
\end{equation}
From \eqref{eq:gradF_weak_omt_smooth} we then obtain the explicit formulation of the flow as
\begin{equation}\label{eq:explicit_lifted_gradient_flow_OMT_smooth}
  \nabla\cdot\rho_0\nabla\dot\phi = -\divv{\rho_0 \divv{(\nabla^{2}\phi)^{-1}} }
  - \divv{ (\nabla^{2}\phi)^{-1}\nabla\rho_0}
  + \divv{\rho_0 \frac{\nabla\rho_1\circ\nabla\phi}{\rho_1\circ\nabla\phi}} .
\end{equation}
In the simple case when $\rho_0 \equiv 1$ we get
\begin{equation*}
  \Delta\dot\phi = -\divv{\divv{(\nabla^{2}\phi)^{-1}} }
  + \Delta{\log(\rho_1\circ\nabla\phi)} .
\end{equation*}

As already mentioned, an approach for a geometric proof of \autoref{lem:shortest_geodesic} is to show that the flow \eqref{eq:explicit_lifted_gradient_flow_OMT_smooth} has a unique limit in the set of strictly convex functions.
We conjecture this to be true, at least when $\rho_1$ is log-concave, based on calculations showing negativeness of the Hessian of $\hat F$ (see \autoref{subsub:limit} for a brief justification).

\iftoggle{final}{}{{\color{blue}
A second dynamical approach to OMT is to consider flows on $\Dens(\RR^{n})$ and then use the correspondence to the polar cone $K_{\diamondsuit}$ given by \autoref{lem:shortest_geodesic}.
Suppose we have a path $\bar{\gamma}(t)$ on $\Dens(\RR^{n})$ such that $\bar{\gamma}(0)=\mu_0$ and $\bar{\gamma}(s)=\mu_1$ for some $s\in(0,\infty]$, and suppose that $\bar{\gamma}(t)$ can be embedded as an integral curve of a vector field $\bar X$ on $\Dens(\RR^{n})$, i.e., $\bar{\gamma}(t)$ is an integral curve of
\begin{equation}
  \dot\mu = \bar X(\mu).
\end{equation}
By linearizing the isomorphism between $K_{\diamondsuit}$ and $\Dens(\RR^{n})$ we obtain an isomorphism $D\pi\colon T K_{\diamondsuit}\to T\Dens(\RR^{n})$, which, by pullback, can be used to construct a vector field $X = D\pi^* \bar X$ on $K_{\diamondsuit}$.
The element $\nabla\phi$ in the polar decomposition~\eqref{eq:polar_decomposition_diffeos} is then given by $\eta(s)$, where $\eta(t)$ is the solution to the flow equation
\begin{equation}
  \dot\eta = X(\eta), \quad \eta(0) = \id.
\end{equation}
For each $\eta\in K_{\diamondsuit}$, the vector $X(\eta) = \nabla\theta$ is obtained by solving a second order linear PDE in~$\theta$ given by linearization of the Monge--Ampère equation~\eqref{eq:monge_ampere} and abstractly written
\begin{equation}
  D\pi(\eta)\cdot\nabla\theta = \bar X(\pi(\eta)).
\end{equation}

Let us now consider a specific choice of vector field $\bar X$ that fulfills the requirement.
It is the gradient flow on the (smooth) Wasserstein space $(\Dens(\RR^{n}),\bar{\mathcal G})$ suggested by~Otto~\cite{Ot2001}.
}}

\subsection{Optimal transport in the linear category}\label{sub:omt_linear}

As we have seen earlier, the multivariate Gaussian distributions $\normal_n$ constitute a submanifold of $\Dens(\RR^{n})$.
In this section we consider the geometry of optimal transport restricted to the linear category, such as studied by Takatsu~\cite{Ta2011}.

\begin{problem}[Linear OMT]\label{prob:linear_omt}
{\it  Given $\mu_0,\mu_1\in\normal_n$, find $\varphi(x) = Ax$ with $A\in\GL(n)$ that minimizes
  \begin{equation}\label{eq:lot_functional}
    \costfunc(\varphi) = \int_{\RR^n} \norm{x-\varphi(x)}^2\,\mu_0
  \end{equation}
  under the constraint
  \begin{equation}\label{eq:lot_constraint}
    \varphi_*\mu_0 = \mu_1.
  \end{equation}}
\end{problem}

If $\varphi(x)=Ax$ then
\begin{eqnarray*}
  \costfunc(\varphi) &= \int_{\RR^{n}} \norm{(I-A)x}^2 \,\mu_0 = \int_{\RR^{n}} x^{\top}\underbrace{(I-A)^\top(I-A)}_Bx \,\mu_0 \\
  &= \int_{\RR^{n}} \sum_{ij} B_{ij}x^j x^i \; \mu_0 =  \sum_{ij} B_{ij} \underbrace{\int_{\RR^{n}} x^j x^i \;\mu_0}_{\cov_{\mu_0}(x^i,x^j)}
\end{eqnarray*}
\begin{align}\label{eq:J_fin_dim}
  &= \tr(B\Sigma_0) = \tr\Big((I-A)\Sigma_0(I-A)^\top\Big),
\end{align}
where $(\Sigma_0)_{ij} = \cov_{\mu_0}(x^i,x^j)$.
Writing $\mu_0 = p(\cdot,\Sigma_0) \vol$, the left hand side of the constraint~\eqref{eq:lot_constraint} becomes
\begin{align}
  (\varphi_*\mu_0)(x) &= \varphi_*(p(\cdot,\Sigma_0)\vol)(x) = p(A^{-1}x,\Sigma_0)\det(A^{-1})\vol \\
  &= \left( \sqrt{\frac{1}{\det(A)^2\det(\Sigma_0)(2\pi)^n}}\exp(-\frac{1}{2}(A^{-1}x)^{\top}\Sigma_0^{-1} A^{-1}x) \right)\vol \\
  &= \left( \sqrt{\frac{1}{{\det(A \Sigma_0A^{\top})}{(2\pi)^n}}}\exp(-\frac{1}{2}x^{\top}(A \Sigma_0 A^{\top})^{-1}x) \right)\vol \\
  &= p(x,A \Sigma_0 A^{\top})\vol.
\end{align}
\reviewerone{
  In the middle of page 10, there is a $\det(W)$ which I think should be $\det(\Sigma_0)$.
  \\[2ex]
  DONE!
}
Since $\normal_n \simeq \Sym{n}$ we can now reformulate \autoref{prob:linear_omt} in terms of covariance matrices. 
\begin{problem}[Linear OMT, reformulated]\label{prob:linear_omt_2}
{\it  Given $\Sigma_0,\Sigma_1\in\Sym{n}$, find $P\in\GL(n)$ that minimizes
  \begin{equation}\label{eq:lot_functional_2}
    \costfunc(P) = \tr\Big((I-P)\Sigma_0(I-P)^\top\Big)
  \end{equation}
  under the constraint
  \begin{equation}\label{eq:lot_constraint_2}
    P \Sigma_0 P^\top = \Sigma_1.
  \end{equation}}
\end{problem}

The action of $\GL(n)$ on $\Sym{n}$ is transitive, so for each pair $\Sigma_0,\Sigma_1\in\Sym{n}$ there exists a $P\in\GL(n)$ fulfilling condition \eqref{eq:lot_constraint_2}.
This is the finite-dimensional analogue of the ``Moser trick'', used in \autoref{sub:geometry_omt_smooth}.
\reviewerone{
  At the end of page 10, the author mentions a finite-dimensional analogue of the Moser trick. While again this is a known linear algebra result, it would be interesting to see what the Moser trick looks like in this context.
  \\[2ex]
  DONE! The Moser trick in this context is a direct consequence of the homogeneous space structure of $\Sym{n}$, given in \eqref{eq:principal_bundle_finite_dim}.
}

Let us now proceed with the geometry of~\autoref{prob:linear_omt_2}.
Since
\begin{equation}\label{eq:GL_as_submanifold}
  \{\varphi \in \Diff(\RR^{n}) \mid \varphi(x)=Ax, \; A\in\GL(n)\}\simeq \GL(n)
\end{equation}
is a submanifold of $\Diff(\RR^n)$, it inherits the Riemannian metric~\eqref{eq:L2noninvariant}, so henceforth we think of $\GL(n)$ as a Riemannian submanifold.
In essence, the result is that the geometry of~\autoref{prob:linear_omt_2} duplicates that of the infinite-dimensional~\autoref{prob:smooth_omt}.
The key to see this is the following simple but important result.


\begin{lemma}\label{lem:tg}
  $\GL(n)$ is \emph{totally geodesic} in $\Diff(\RR^{n})$.
  That is, if $\gamma(t)$ is a geodesic curve in $\GL(n)$ then it is also a geodesic curve in $\Diff(\RR^n)$.
\end{lemma}

\begin{proof}
  Under the identification~\eqref{eq:GL_as_submanifold}, tangent vectors of the submanifold $\GL(n)$ consist of linear mappings $\RR^{n}\to\RR^{n}$.
  The result follows since the solution to~\eqref{eq:geodesic_eq_omt} with $\varphi(0)\in\GL(n)$ and $\dot\varphi(0)$ a linear mapping remains a linear mapping.
\end{proof}

Explicitly, the Riemannian metric~\eqref{eq:L2noninvariant} restricted to $\GL(n)$ is given by
\begin{equation}\label{eq:metric_linear_OMT}
  \mathcal{G}_A(\dot A,\dot A) = \tr(\Sigma_0 \dot A^{\top}\dot A),
\end{equation}
which can also be written
\begin{equation}\label{eq:metric_linear_OMT_2}
  \mathcal{G}_A(\dot A,\dot A) = \tr(A\Sigma_0A^{\top} (\dot A A^{-1})^{\top}(\dot A A^{-1})).
\end{equation}
The corresponding Riemannian squared distance between $A_0,A_1\in\GL(n)$ is given by
\begin{equation}
  d^{2}(A_0,A_1) = \tr(\Sigma_0 (A_0-A_1)^{\top}(A_0-A_1)),
\end{equation}
so, as in the infinite-dimensional case, we have
\begin{equation}
  J(A) = d^{2}(I,A).
\end{equation}
Notice that if $\Sigma_0 = I$, then $d^{2}(A_0,A_1) = \norm{A_1-A_0}^{2}_F$, where $\norm{\cdot}_{F}$ is the Frobenius norm.
\reviewerone{
  Equation (20) appears too early; it doesn’t really make sense until after $O(n,\Sigma_0)$ is defined in Section 2.3.1. \\[2ex]
  DONE!
}%
\reviewertwo{
after eq. (20): set $O(n, \Sigma_0)$ is only introduced afterwards, in definition of $O(n, \Sigma_0)$ writing the constraint as $A\Sigma_0 A^T = \Sigma_0$ seems more consistent
\\[2ex]
DONE!
}%

A direct consequence of~\autoref{lem:tg} is that the squared distance from the identity to $A$ with respect to $\mathcal{G}$ is given by~\eqref{eq:lot_functional_2}.
Therefore, all the geometric aspects of \autoref{prob:smooth_omt} are valid also for \autoref{prob:linear_omt_2}, but in a finite-dimensional setting.
In particular, solutions are given by horizontal geodesics.
For completeness, we shall now derive explicitly the analogous finite-dimensional geometric concepts.

\subsubsection{Principal bundle structure}
The principal bundle analogous to~\eqref{eq:principal_bundle_omt} is
\begin{equation}\label{eq:principal_bundle_finite_dim}
  \begin{tikzcd}
    \GL(n) \arrow[hookleftarrow]{r}{} \arrow{d}{\pi} & \OO(n,\Sigma_0) \\
    \Sym{n} &
  \end{tikzcd}
\end{equation}
where
\begin{equation}
  \OO(n,\Sigma_0) = \{ Q\in\GL(n)\mid Q\Sigma_0 Q^{\top} = \Sigma_0 \}
\end{equation}
is the symmetry Lie group corresponding to $\Diff_{\mu_0}(\RR^{n})$ in the infinite-dimensional case.

The projection $\pi\colon\GL(n)\to\Sym{n}$ is given by
\begin{equation}
  \pi(A) = A\Sigma_0 A^{\top}.
\end{equation}
Its derivative $D\pi(A)$ is computed as follows:
if $\Sigma = \pi(A)$ then
\begin{equation}
  \dot\Sigma = \dot A A^{-1} A\Sigma_0 A^{\top} + A\Sigma_0 A^{\top} (\dot A A^{-1})^{\top} = \dot A A^{-1} \Sigma + \Sigma (\dot A A^{-1})^{\top}.
\end{equation}
\reviewerone{
  Equation (23) should have a transpose on the last term; this doesn’t matter when talking about horizontal vectors only, but it does when the author uses the formula to get the vertical vectors.
  \\[2ex]
  DONE!
}%
Thus,
\begin{equation}\label{eq:pi_deriv_finite}
  D\pi(A)\cdot\dot A = V\Sigma + \Sigma V^{\top}, \qquad V=\dot A A^{-1},\; \Sigma = \pi(A).
\end{equation}
We encourage the reader to compare this formula with the infinite-dimensional case~\eqref{eq:deriv_pi}.

From \eqref{eq:pi_deriv_finite} we get that the vertical distribution is given by
\begin{equation}\label{eq:vert_dist_finite}
  \Ver_A = \{VA \in T_A\GL(n)\mid V\pi(A)+\pi(A)V^{\top} = 0 \iff V\in \mathfrak{o}(n,\pi(A))\},
\end{equation}
where $\mathfrak{o}(n,\Sigma)$ denotes the Lie algebra of $\OO(n,\Sigma)$.
Again, compare with the infinite-dimensional case~\eqref{eq:vert_dist_omt_infinite}.
\reviewertwo{
  eq. (24): $o(n,\pi(A))$ seems to be the Lie-algebra associated to $O(n,\Sigma_0)$, but notation is not defined. \\[2ex]
  FIXED!
}



\subsubsection{Descending metric}
The formula \eqref{eq:metric_linear_OMT_2} reveals that the metric $\mathcal G$
is $\OO(n,\Sigma_0)$ right-invariant, as expected.
It therefore descends to a metric $\bar{\mathcal{G}}$ on $\GL(n)/\OO(n,\Sigma_0) \break \simeq\Sym{n}$, corresponding to the restriction of the Wasserstein metric~\eqref{eq:wasserstein_metric_smooth_omt} to $\mathcal{N}_n$.

The horizontal distribution~\eqref{eq:horizontal_omt} restricted to~$\GL(n)$ is given by
\begin{equation}\label{eq:horizontal_linear_omt}
  \Hor_A = \{ \dot A \in T_A\GL(n) \mid \dot A A^{-1} \in \TSym{n} \}.
\end{equation}
Indeed, if $\varphi(x) = A x$ and $\dot\varphi(x)=\dot A x$ then
\begin{equation}
  \dot\varphi \in \Hor_{\varphi} \iff \dot A A^{-1}x = \nabla f(x).
\end{equation}
Since $x\mapsto \dot A A^{-1} x$ is a linear mapping, it follows that $f(x)$ must be a quadratic form; we thereby recover \eqref{eq:horizontal_linear_omt}.
Another way to arrive at the same result is to directly compute the orthogonal complement in $T_{A}\GL(n)$ of $\Ver_{A}$.
\reviewerone{
  It is plausible but not completely obvious why the restriction of the horizontal space of gradients should restrict to symmetric matrices in equation (22); this deserves a comment or short justification. Also it seemed strange to discuss horizontal vectors before vertical vectors.
  \\[2ex]
  DONE! Vertical distribution is now described above, and a computation showing the horizontal dist is given.
}%
The metric $\bar{\mathcal{G}}$ is now defined by
\begin{equation}
  \bar{\mathcal{G}}_{\pi(A)}(D\pi(A)\cdot \dot A, D\pi(A)\cdot \dot A) = \mathcal{G}_A(\dot A,\dot A), \quad \forall\; \dot A\in \Hor_A .
\end{equation}
To work it out explicitly, first notice that
\begin{equation}
  D\pi(A)\cdot \dot A = \dot A\Sigma_0 A^{\top} + A\Sigma_0 \dot A^{\top}.
\end{equation}

From \eqref{eq:metric_linear_OMT_2} and \eqref{eq:horizontal_linear_omt} it then follows that
\begin{equation}\label{eq:bar_metric_linear_OMT}
  \bar{\mathcal{G}}_{\Sigma}(\dot\Sigma,\dot\Sigma) = \tr(\Sigma S S),
\end{equation}
where $S\in\TSym{n}$ is the solution to the \emph{continuous Lyapunov equation} (a special case of a \emph{Sylvester equation}~\cite{Sy1884,BaSt1972}) given by
\begin{equation}
  \dot\Sigma = S\Sigma + \Sigma S.
\end{equation}


\subsubsection{Finite-dimensional Monge--Ampère equation}
From~\eqref{eq:horizontal_linear_omt} it follows that the horizontal geodesics from the identity are of the form
\begin{equation}\label{eq:horizontal_geodesics_matrix}
  \gamma(t) = \underbrace{I + t S}_{P(t)}, \quad S\in \TSym{n}.
\end{equation}
Consequently, the finite-dimensional analogue of the Monge--Ampère equation~\eqref{eq:monge_ampere} consists in finding $P \coloneqq P(1)\in\TSym{n}$ such that
\begin{equation}\label{eq:monge_ampere_finite_dim}
  \pi(P) = \Sigma_1 \iff P\Sigma_0P = \Sigma_1 .
\end{equation}
In particular, if $\Sigma_0=I$ the solution is the matrix square root of $\Sigma_1$.
Hence we see
that the solution to the Monge--Ampère equation~\eqref{eq:monge_ampere} with $\rho_0=1$ is, in a certain sense, a generalization of the matrix square root.

\subsection{Polar decomposition of matrices}\label{sub:polar_decomposition_matrices}

The finite-dimensional polar cone $K_{\lozenge}$ consists of those matrices that are connected to the identity matrix by a curve in $\GL(n)$ of the form~\eqref{eq:horizontal_geodesics_matrix}, i.e., by a horizontal geodesic.
The result corresponding to \autoref{lem:polar_cone_isomorphism} is the following.
\begin{lemma}\label{lem:finite_polar_cone_isomorphism}
  The polar cone $K_{\lozenge}\subset \GL(n)$ consists of all positive definite symmetric matrices.
  It is a convex submanifold of $\GL(n)$.
\end{lemma}

\begin{proof}
  The proof is almost identical to that of~\autoref{lem:polar_cone_isomorphism}:
  Let $P\in K_{\lozenge}$ and take $\gamma(t)$ to be the horizontal geodesic such that $\gamma(0)=I$ and $\gamma(1)=P$.
  Then, for each $t\in[0,1]$, $\gamma(t)$ is a symmetric matrix and an element of $\GL(n)$.
  Since $I$ has only positive eigenvalues, it follows that $\gamma(t)$ has only positive eigenvalues.
  Thus, $\gamma(t)$ is positive definite, so $K_{\lozenge}$ consists of positive definite symmetric matrices.

  Now, if $P$ is any positive definite symmetric matrix, then $\gamma(t) =  (1-t) I + t P$ is a horizontal geodesic originating from the identity.
  Thus, $P = \gamma(1) \in K_{\lozenge}$ per definition. 

  Convexity of $K_{\lozenge}$ follows since a convex combination of positive definite symmetric matrices is positive definite symmetric. %
  \reviewertwo{
    How about a three-line proof for Lemma 2.6?
    \\[2ex]
    DONE!
  }%
\end{proof}

Next follows the analogue of \autoref{lem:shortest_geodesic}.
\reviewerone{
  Lemma 2.7 is similarly not proved or even discussed, and only in Lemma 3.9 is some effort made to justify the statement. At the very least the author should explain why these lemmas are true even if a formal proof is too much trouble.
  \\[2ex]
  DONE! We now give a complete geometric proof.
}
\begin{lemma}\label{lem:shortest_geodesic_finite_dim}
  The polar cone $K_{\lozenge}$ is a section of the principal bundle~\eqref{eq:principal_bundle_finite_dim}.
  That is, the mapping
  \begin{equation}
    \pi\big|_{K_{\lozenge}}\colon K_{\lozenge} \to \Sym{n}
  \end{equation}
  is an isomorphism.
\end{lemma}

Whereas this result readily follows by linear algebraic techniques, we shall, as mentioned, give a new, geometric proof in \autoref{subsub:lifted_gradient_omt_linear}, based on a finite-dimensional analogue of the lifted Riemannian gradient flow in \autoref{subsub:lifted_gradient_flow_smooth}.


The decomposition now reads as follows.

\begin{theorem}[Polar decomposition of matrices]\label{thm:polar_decomposition_matrices}
  Let $A\in\GL(n)$ and $\Sigma_0\in\Sym{n}$.
  Then there exist unique matrices $P\in\Sym{n}$ and $Q\in\OO(n,\Sigma_0)$ such that
  \begin{equation}
    A = PQ.
  \end{equation}
  The matrix $P$ is the unique solution of \autoref{prob:linear_omt_2} with $\Sigma_1 = A\Sigma_0 A^{\top}$.
\end{theorem}
\reviewerone{
  In Theorem 2.8 the notation is a bit confusing; the A in Problem 3 is not the A here, but rather the P here. \\[2ex]
  DONE! Now using $P$ instead in Problem 3.
}

\begin{proof}
  Let $\Sigma_1 = \pi(A)$.
  From \autoref{lem:shortest_geodesic_finite_dim} we get a unique corresponding matrix $P\in K_{\lozenge} = \Sym{n}$ such that $\pi(P)=\Sigma_1$.
  By construction, $A$ and $P$ belong to the same fiber, so by the principal bundle structure~\eqref{eq:principal_bundle_finite_dim} it follows that $Q\coloneqq P^{-1}A \in \OO(n,\Sigma_0)$.
  That $P$ is the solution of \autoref{prob:linear_omt_2} follows from the geometry since it is the endpoint of a horizontal geodesic originating from the identity, which is the shortest curve between the identity matrix and the fiber $\pi^{-1}(\Sigma_1)$.
\end{proof}

\subsubsection{Vertical Gradient Flow}\label{subsub:vert_gradient_flow_linear_omt}

Let $\nabla_{\mathcal G}$ denote the gradient with respect to the metric~\eqref{eq:metric_linear_OMT} and let $\mathrm{Pr}_{\Ver}$ denote orthogonal projection onto the vertical distribution~\eqref{eq:vert_dist_finite}.
With $J$ as in \autoref{prob:linear_omt_2} we are then interested in the constrained gradient flow
\begin{equation}
  \dot B = - \mathrm{Pr}_{\Ver}\nabla_{\mathcal G} J(B), \quad B(0) = A.
\end{equation}
An equilibrium of this flow is obtained at the symmetric matrix $P$ in the polar decomposition $A=PQ$.

Since
\begin{equation}
  \frac{\ud}{\ud t}J(B) = 2\tr(\Sigma_0(B-I)^{\top}\dot B) = \mathcal G_{B}(2(B-I),\dot B),
\end{equation}
and since the horizontal distribution is given by \eqref{eq:horizontal_linear_omt}, it follows that the vertical gradient flow is
\begin{equation}\label{eq:vert_flow_omt_linear}
  \dot B + 2(B - I) = -S B, \quad B\Sigma_0 B^{\top} = \Sigma_1,
\end{equation}
where $S\in\TSym{n}$ is a Lagrange multiplier and $\Sigma_1 = A\Sigma_0 A^{\top}$.
\reviewerone{
  At the top of page 13 I expected to see some mention of how one computes S (and similarly how one would compute p in the corresponding situation at the end of page 8). Also this would be an excellent place to write what the equation looks like in a low-dimensional concrete example.
  \\[2ex]
  DONE! We now explain in great detail how $S$ is computed, and also give numerical examples.
 }%

\begin{remark}
  In order for the flow \eqref{eq:vert_flow_omt_linear} to be able to reach $P$ in \autoref{thm:polar_decomposition_matrices} it is necessary that the initial data~$A$ belong to the identity component of $\GL(n)$.
  A future topic (see \autoref{subsub:outlook_vert_limit}) is to find minimal conditions under which the flow converges to $P$.
\end{remark}

By construction, $\dot B \in \Ver_{B}$.
Multiplying \eqref{eq:vert_flow_omt_linear} from the right by $B^{-1}$ and using from \eqref{eq:vert_dist_finite} that $\dot B B^{-1} \in \mathfrak{o}(n,\Sigma_1)$, we obtain a continuous Lyapunov equation for the Lagrange multiplier $S$, namely
\begin{equation}
  S \Sigma_1 + \Sigma_1 S  = 4 \Sigma_1 - 2 (B^{-1} \Sigma_1  + \Sigma_1 B^{-\top}).
\end{equation}

It is possible to formulate the equations in the right reduced variable $\Omega\coloneqq \dot B B^{-1}$, without using Lagrange multipliers.
Indeed, multiplying \eqref{eq:vert_flow_omt_linear} from the right and subtracting the transpose of the whole equation, we get
\begin{align} 
  \underbrace{\dot B B^{-1}}_{\Omega}  + 2(I - B^{-1}) &= -S \\
  &\Updownarrow \\
  \Omega - \Omega^{\top}  + 2(B^{-\top} - B^{-1}) &= 0 .
\end{align}
Using the characterization \eqref{eq:vert_dist_finite} of $\Ver$ we then get
\begin{align}
  \Omega + \Sigma_1^{-1} \Omega \Sigma_1  + 2(B^{-\top} - B^{-1}) &= 0 \\
      &\Updownarrow \\
  \Sigma_1 \Omega + \Omega \Sigma_1  + 2\Sigma_1 (B^{-\top} - B^{-1}) &= 0 .
\end{align}
Thus, an alternative form for the gradient flow \eqref{eq:vert_flow_omt_linear} is
\begin{equation}\label{eq:vert_flow_omt_linear2}
  \left\{
  \begin{aligned}
    \dot B &= \Omega B \\
    \Sigma_1 \Omega + \Omega \Sigma_1  &= 2\Sigma_1 (B^{-1} - B^{-\top}).
  \end{aligned}
  \right.
\end{equation}
Notice that $B^{-\top}$ can be computed from $A^{-1}$ (the inverse of the initial data) and $\Sigma_0^{-1}$, since
\begin{equation}
  B^{-1} A \in \OO(n,\Sigma_0) \iff \Sigma_0 A^{\top} B^{-\top} = A^{-1} B \Sigma_0
  \iff B^{-\top} = A^{-\top}\Sigma_0^{-1}A^{-1} B \Sigma_0.
\end{equation}
Likewise,
\begin{equation}
  A^{-1} B \in \OO(n,\Sigma_0) \iff \Sigma_0 B^{\top} A^{-\top} = B^{-1} A\Sigma_0
  \iff B^{-1} = \Sigma_0 B^{\top} A^{-\top} \Sigma_0^{-1}A^{-1}.
\end{equation}

\begin{figure}
  \iftoggle{arxiv}{\includegraphics[scale=\iftoggle{aims}{1.0}{1.2}]{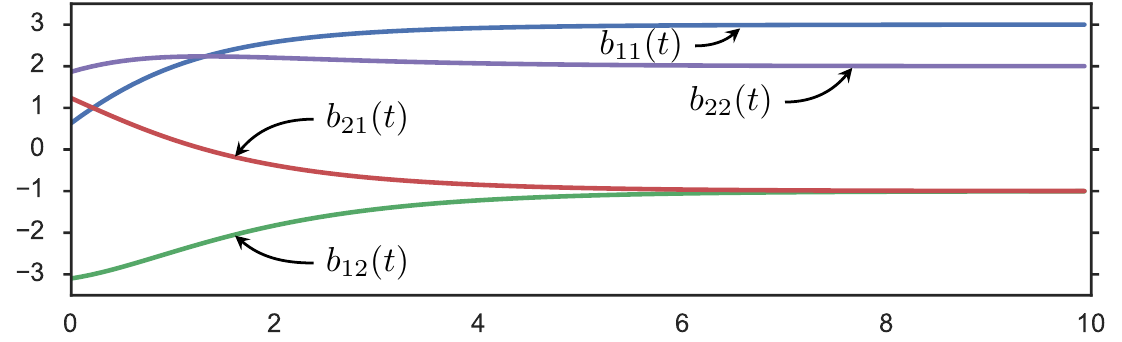}}{\input{fig_Bplot}}
  \caption{Evolution of the matrix elements of $B(t)$ for the vertical gradient flow in \protect\autoref{ex:vert_flow_omt_finite}.
  Notice that $B(0)=A$ and that $B(t)$ converges towards $P_{\infty}$ in \protect\eqref{eq:Pinf_and_Qinf_vert} as $t\to\infty$.
  }\label{fig:vert_flow_example_omt_finite}
\end{figure}

\begin{figure}
  \iftoggle{arxiv}{\includegraphics[scale=\iftoggle{aims}{1.0}{1.2}]{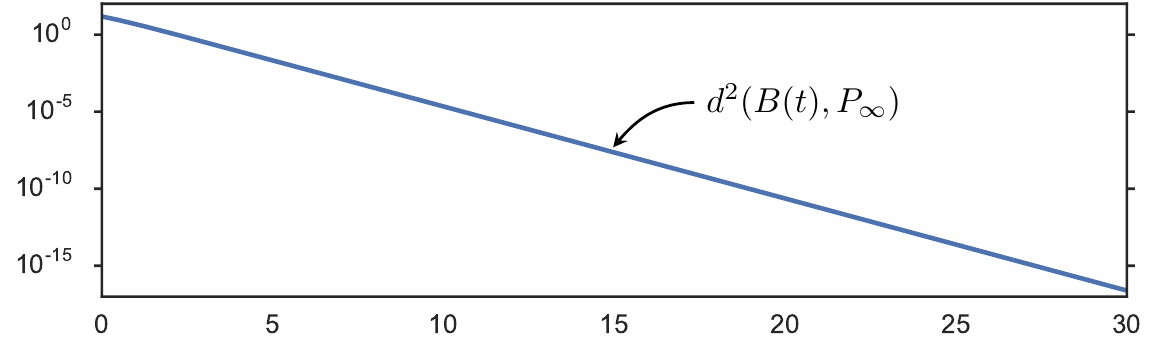}}{\input{fig_Berrorplot}}
  \caption{Convergence towards the limit $P_{\infty}$ of the vertical gradient flow in \protect\autoref{ex:vert_flow_omt_finite}.
  }\label{fig:vert_flow_example_omt_finite_errors}
\end{figure}

\begin{example}\label{ex:vert_flow_omt_finite}
  We give here an explicit example of a vertical gradient flow.
  Take
  \begin{equation} \label{eq:Pinf_and_Qinf_vert}
    P_{\infty} = \begin{pmatrix} 3 & -1 \\ -1 & 2 \end{pmatrix} \qquad\text{and}\qquad
    Q_{\infty} = \begin{pmatrix} \cos\theta & -\sin\theta \\ \sin\theta & \cos\theta \end{pmatrix}
  \end{equation}
  with $\theta = \pi/3$.
  Set
  \begin{equation}
    A \coloneqq P_{\infty}Q_{\infty} = \frac{1}{2}\begin{pmatrix}
      {3-\sqrt{3}} & -{3\sqrt{3}-1} \\
      {2\sqrt{3}-1} & {\sqrt{3}+2} \\
    \end{pmatrix}
  \end{equation}
  and take $\Sigma_0 = I$.
  The matrix $\Sigma_1$ is then given by
  \begin{equation}
    \Sigma_1 = A A^{\top} = \begin{pmatrix}
      10 & -5 \\
      -5 & 5
    \end{pmatrix}.
  \end{equation}

  We discretize the vertical gradient flow \eqref{eq:vert_flow_omt_linear2} in time by the \emph{Lie--Euler method} (cf.\ \cite{CeMaOw2014})
  \begin{equation}
    B_{k+1} = \exp(\Delta t\, \Omega_k)B_k, \quad B_0 = A,
  \end{equation}
  where $\Omega_k$ is computed from $B_k$ by solving the Sylvester equation in \eqref{eq:vert_flow_omt_linear2} using the Bartels--Stewart algorithm \cite{BaSt1972}.
  We use $\Delta t = 0.1$ as time-step.

  The evolution of the matrix elements
  \begin{equation}
    B(t) = \begin{pmatrix}
      b_{11}(t) & b_{12}(t) \\
      b_{21}(t) & b_{22}(t)
    \end{pmatrix}
  \end{equation}
  is shown in \autoref{fig:vert_flow_example_omt_finite};
  $B$ starts at $A$ and converges towards $P_{\infty}$.
  The convergence in squared Riemannian distance is shown in \autoref{fig:vert_flow_example_omt_finite_errors}; it appears to be exponential.
  To give a full explanation of the rapid convergence rate observed here is an interesting, future topic (see \autoref{sec:outlook} below).
\end{example}

\reviewertwo{
  Similarly, the gradient flows in Sections 2.4.1 and 2.4.2 could benefit from a few more intermediate steps. \\[2ex]
  DONE! A lot more details is now given for both types of flows.
  In particular, we also give numerical examples.
}

\subsubsection{Entropy Gradient Flow}

Here we consider the analogue of the entropy gradient flow~\eqref{eq:dens_gradient_flow}.
To this extent, the relative entropy functional~\eqref{eq:relative_entropy} restricted to $\mathcal{N}_{n}\simeq\Sym{n}$ is given by
\begin{equation}\label{eq:rel_entropy_fin_dim}
  H(\Sigma) = \frac{n}{2} -\frac{1}{2}\tr(\Sigma_1^{-1}\Sigma) +\frac{1}{2}\log\left( \det(\Sigma_1^{-1}\Sigma)\right).
\end{equation}
To see this, let $\mu=p(\cdot,\Sigma)$ and $\mu_1=p(\cdot,\Sigma_1)$ with $\Sigma,\Sigma_1\in\Sym{n}$.
We then have
\begin{equation}
  \frac{\mu}{\mu_1} = \sqrt{\frac{\det(\Sigma_1)}{\det(\Sigma)}}\exp\left(-\frac{1}{2} x^{\top}(\Sigma^{-1}-\Sigma_1^{-1})x\right),
\end{equation}
so
\begin{equation}
  \log\left(\frac{\mu}{\mu_1} \right) = \frac{1}{2}\log\left( \frac{\det(\Sigma_1)}{\det(\Sigma)} \right) - \frac{1}{2} x^{\top}\left( \Sigma^{-1}-\Sigma_1^{-1} \right) x .
\end{equation}
From~\eqref{eq:relative_entropy} we now get
\begin{align}
  H(\mu) &= -\frac{1}{2}\int_{\RR^{n}} \left( \log\left( \frac{\det(\Sigma_1)}{\det(\Sigma)} \right) - \ x^{\top}\left( \Sigma^{-1}-\Sigma_1^{-1} \right) x \right) \mu \\
  &= -\frac{1}{2}\log\left( \frac{\det(\Sigma_1)}{\det(\Sigma)} \right) \underbrace{\int_{\RR^{n}}\mu}_{1} + \frac{1}{2}\int_{\RR^{n}} x^{\top}\Sigma^{-1}x \mu - \frac{1}{2}\int_{\RR^{n}}x^{\top}\Sigma_{1}^{-1} x \mu \\
  &= \frac{1}{2}\log\left( \det(\Sigma_1^{-1}\Sigma) \right) + \frac{1}{2}\underbrace{\tr(\Sigma^{-1}\Sigma)}_{n} - \frac{1}{2}\tr(\Sigma_{1}^{-1}\Sigma),
\end{align}
where the last equality follows from the same calculation as in~\eqref{eq:J_fin_dim}.
This proves the formula~\eqref{eq:rel_entropy_fin_dim} for $H(\Sigma)$.

\reviewerone{
  Also on page 13, I thought the formula for $H(\Sigma)$ should have been derived since it’s not obvious. It seems to me that it comes from integrating the formulas for $\mu$ which are given by the Gaussian density in terms of $\Sigma$, so there is a cancellation here that I wanted to see. A similar integration should produce the formula in Lemma 3.4.
  \\[2ex]
  DONE! The formulae for $H(\Sigma)$, as well as that in the Lemma, are now carefully derived.
}
If $\Sigma=\Sigma(t)$ is a curve in $\Sym{n}$, then
\begin{align}
  \frac{\ud}{\ud t}H(\Sigma) &=
  \frac{1}{2}\frac{\ud}{\ud t}\log\left( \frac{\det(\Sigma)}{\det(\Sigma_1)} \right) - \frac{1}{2}\tr(\Sigma_1^{-1}\dot\Sigma)
  \\
  &=
  \frac{1}{2}\tr(\Sigma^{-1}\dot\Sigma) - \frac{1}{2}\tr(\Sigma_1^{-1}\dot\Sigma).
\end{align}
Taking $\dot\Sigma = S\Sigma + \Sigma S$ for $S\in\TSym{n}$, we get
\begin{equation}
  \begin{split}
    \frac{\ud}{\ud t}H(\Sigma)
    &= \frac{1}{2}\tr(\Sigma^{-1}(S\Sigma+\Sigma S)) - \frac{1}{2}\tr(\Sigma_1^{-1}(S\Sigma+\Sigma S))  \\
    &= \frac{1}{2}\big(\tr(\Sigma^{-1}S\Sigma)+\tr(S)\big) - \frac{1}{2}\big(\tr(\Sigma_1^{-1}S\Sigma)+\tr(\Sigma_1^{-1}\Sigma S)\big)  \\
    &= \frac{1}{2}\big(\tr(\Sigma\Sigma^{-1}S)+\tr(S)\big) - \frac{1}{2}\big(\tr(\Sigma\Sigma_1^{-1}S)+\tr(S \Sigma_1^{-1} \Sigma)\big)  \\
    &= \tr(\Sigma\Sigma^{-1}S) - \frac{1}{2}\big(\tr(\Sigma\Sigma_1^{-1}S)+\tr((S \Sigma_1^{-1} \Sigma)^\top)\big)  \\
    &= \tr(\Sigma\Sigma^{-1}S) - \tr(\Sigma\Sigma_1^{-1}S)
  \end{split}
\end{equation}
From~\eqref{eq:bar_metric_linear_OMT} it then follows that
\begin{equation}
  \frac{\ud}{\ud t}H(\Sigma) = \bar{\mathcal{G}}_{\Sigma}(2 I - \Sigma_1^{-1}\Sigma-\Sigma\Sigma_1^{-1},\dot\Sigma)
\end{equation}
The gradient flow
\begin{equation}
  \dot\Sigma = \nabla_{\bar{\mathcal G}}H(\Sigma)
\end{equation}
is therefore given by
\begin{equation}\label{eq:sym_gradient_flow}
  \dot\Sigma = 2 I - \Sigma_1^{-1}\Sigma-\Sigma\Sigma_1^{-1}.
\end{equation}
By construction, this is the restriction to Gaussian distributions of the infinite-dimensional gradient flow~\eqref{eq:dens_gradient_flow}.
As in~\eqref{eq:dens_gradient_flow}, the flow strives toward the maximum of the relative entropy $H(\Sigma)$, which occurs at $\Sigma = \Sigma_1$.

We shall now give a result on the convergence of \eqref{eq:sym_gradient_flow}.
The essential result is the following on convexity of minus the relative entropy functional.

\begin{lemma}\label{lem:pos_def_hessian_omt_finite}
  The Hessian of the relative entropy function \eqref{eq:rel_entropy_fin_dim} with respect to the Riemannian metric \eqref{eq:bar_metric_linear_OMT} fulfills the following inequality:
  there exists $\alpha> 0$ such that
  \begin{equation}
    -\hess(H)_\Sigma(\dot\Sigma,\dot\Sigma) \geq \alpha\, \bar{\mathcal{G}}_{\Sigma}(\dot\Sigma,\dot\Sigma), \quad \forall\; (\Sigma,\dot\Sigma)\in T\Sym{n}.
  \end{equation}
\end{lemma}

We postpone the proof of this result until the next section: to compute the Hessian of $H$ it is easier to first lift it to a function $F=H\circ\pi$ on $\GL(n)$, then compute the Hessian, and then restrict it to the horizontal distribution.

A consequence of \autoref{lem:pos_def_hessian_omt_finite} is the following result.

\begin{theorem}\label{thm:limit_entropy_flow_omt_finite}
  For any initial data $\Sigma(0)\in\Sym{n}$, the entropy gradient flow \eqref{eq:sym_gradient_flow} converges exponentially fast towards the minimum $\Sigma_1$ of the relative entropy $H(\Sigma)$.
\end{theorem}

\begin{proof}
  Given \autoref{lem:pos_def_hessian_omt_finite}, the result is a special instance of a general result on gradient flows on Riemannian manifolds.
  See, for example, \cite[\S\,3.5]{Ot2001} or \cite[Ch.~24: Remark~24.9]{Vi2009}.
\end{proof}

\todo[inline]{Possibly add a simple numerical example, $2\times 2$ or $3\times 3$ matrices.}

\subsubsection{Lifted Gradient Flow}\label{subsub:lifted_gradient_omt_linear} 


The objective here is to lift the relative entropy functional $H(\Sigma)$ in~\eqref{eq:rel_entropy_fin_dim} to $F=H\circ\pi$ and consider the gradient flow of $F$ restricted to the polar cone~$K_{\lozenge}$.
By showing that $-F$ is convex on $K_{\lozenge}$ we can thereby prove that the flow has a unique limit, which, as we shall see, implies that the mapping in \autoref{lem:shortest_geodesic_finite_dim} is an isomorphism.
In addition, the lifted gradient flow provides a dynamical method for computing~$P$ in the polar decomposition $A=PQ$, or, equivalently, the solution to \autoref{prob:linear_omt_2}.

The relative entropy lifted to $\GL(n)$ is given by
\begin{align}
  F\colon \GL(n) \to \RR, &\quad F(A)  = H(\pi(A)) \label{eq:lifted_rel_entropy_omt_finite} \\
  &= \frac{n}{2} - \frac{1}{2}\tr(\Sigma_1^{-1}A\Sigma_0 A^{\top}) + \frac{1}{2}\log\left(\frac{\det(A\Sigma_0 A^{\top})}{\det(\Sigma_1)} \right) \\
  &= \frac{n}{2} - \frac{1}{2}\tr(\Sigma_1^{-1}A\Sigma_0 A^{\top}) + \frac{1}{2}\log\left(\det(A)^2 \frac{\det(\Sigma_0)}{\det(\Sigma_1)} \right) \\
  &= \frac{n}{2} - \frac{1}{2}\tr(\Sigma_1^{-1}A\Sigma_0 A^{\top}) + \log\left(\det(A)\right) + \frac{1}{2}\log\left(\frac{\det(\Sigma_0)}{\det(\Sigma_1)} \right).
\end{align}
If $A = A(t)$ is a curve in $\GL(n)$, then
\begin{align}
  \frac{\ud}{\ud t} F(A) &= -\frac{1}{2}\tr(\Sigma_1^{-1}A\Sigma_0 \dot A^{\top}) -\frac{1}{2}\tr(\Sigma_1^{-1}\dot A\Sigma_0 A^{\top}) + \frac{1}{\det(A)}\frac{\ud}{\ud t}\det(A) \\
  &= -\frac{1}{2}\tr(\Sigma_0\dot A^{\top}\Sigma_1^{-1}A) - \frac{1}{2}\tr(\Sigma_0A^{\top}\Sigma_1^{-1}\dot A)) + \tr(A^{-1}\dot A) \label{eq:dFdt} \\
  &= -\frac{1}{2}\tr(\Sigma_0\dot A^{\top}\Sigma_1^{-1}A) - \frac{1}{2}\tr(\Sigma_0(\Sigma_1^{-1}A)^{\top}\dot A)) + \tr(\Sigma_0\Sigma_0^{-1}A^{-1}\dot A) \\
  &= -\frac{1}{2}\mathcal{G}_A(\dot A,\Sigma_1^{-1}A) - \frac{1}{2}\mathcal{G}_A(\Sigma_1^{-1}A,\dot A) + \mathcal{G}_A(A^{-\top}\Sigma_0^{-\top},\dot A) \\
  &= \mathcal{G}_A(A^{-\top}\Sigma_0^{-1}-\Sigma_1^{-1}A,\dot A).
\end{align}
Thus, the gradient of $F$ with respect to $\mathcal{G}$ is given by
\begin{equation}\label{eq:lifted_entropy_gradient}
  \nabla_{\mathcal G}F(A) = A^{-\top}\Sigma_0^{-1} - \Sigma_1^{-1}A
\end{equation}
and the corresponding gradient flow is
\begin{equation}\label{eq:lifted_gradient_flow_finite}
  \dot A = A^{-\top}\Sigma_0^{-1} - \Sigma_1^{-1}A
\end{equation}
Before we continue, let us make a few remarks.

\begin{itemize}
  \item $A\in\GL(n)$ is an equilibrium of $\nabla_{\mathcal G}F$ if and only if $A$ belongs to the fiber of $\Sigma_1$.
  That is, $\pi(A)=\Sigma_0$ if and only if $\nabla_{\mathcal G}F(A) = 0$.
  Indeed,
  \begin{equation}
    \pi(A) = \Sigma_1 \iff A\Sigma_0 A^{\top} = \Sigma_1 \iff A^{-\top}\Sigma_0^{-1}A^{-1} = \Sigma_1^{-1} \iff \nabla_{\mathcal G}F(A) = 0 .
  \end{equation}

  \item Because $F$ is lifted from a function on $\Sym{n}$ it is constant on the fibers: if $A,B\in\GL(n)$ and $\pi(A)=\pi(B)$ then $F(A)=F(B)$.
  Its gradient is therefore orthogonal to the fibers, i.e., $\nabla_{\mathcal G}F(A)$ is in the horizontal distribution $\Hor_{A}$.
  From the characterization \eqref{eq:horizontal_linear_omt} of $\Hor$ it means that
  \begin{equation}
    \nabla_{\mathcal G}F(A) A^{-1} \in \TSym{n}, \quad \forall\, A\in\GL(n).
  \end{equation}
  We encourage the reader to verify this from the expression~\eqref{eq:lifted_entropy_gradient}.
\end{itemize}

We aim to restrict the gradient flow \eqref{eq:lifted_gradient_flow_finite} to the polar cone $K_{\lozenge}$.
Recall from~\autoref{lem:finite_polar_cone_isomorphism} that the polar cone $K_{\lozenge}$ consist all positive definite symmetric matrices.
At the identity, it coincides with the horizontal distribution: $T_{e}K_{\lozenge} = \Hor_e = \TSym{n}$.
However, at points away from the identity this is not true;
for one thing, the horizontal distribution is not \emph{integrable}, i.e., it does not define a submanifold of~$\GL(n)$.
Thus, if $P\in K_{\lozenge}$ then $\nabla_{\mathcal G}F(P)$ is typically not an element of $T_PK_{\lozenge}$.
Consequently, the polar cone $K_{\lozenge}$ is \emph{not} invariant under the gradient flow~\eqref{eq:lifted_gradient_flow_finite}.
Indeed, from~\eqref{eq:lifted_entropy_gradient} we immediately see that for a generic positive definite symmetric matrix $P$, the vector $\nabla_{\mathcal G}F(P)$ fails to be a symmetric matrix.
Another way to understand this is to observe that the fibers generally do not cut the polar cone orthogonally (it cuts the horizontal distribution orthogonally by definition).

Now, the gradient of $F$ restricted to $K_{\lozenge}$ is given by
\begin{equation}\label{eq:grad_polar_cone_omt_finite}
  \Pi_P \nabla_{\mathcal{G}}F(P),
\end{equation}
where $\Pi_P$ is the orthogonal projection onto $T_PK_{\lozenge}$.
(The gradient on a submanifold is the projection of the gradient on the ambient Riemannian manifold.)
Thus, we need to know the orthogonal complement of the polar cone.

\begin{lemma}\label{lem:cone_complement_omt_finite}
  Let $P\in K_{\lozenge}$.
  Then the orthogonal complement of $T_{P}K_{\lozenge}$ inside $T_{P}\GL(n)$ with respect to the Riemannian metric~\eqref{eq:metric_linear_OMT} is given by
  \begin{equation}
    N_{P}K_{\lozenge} \coloneqq \{ V \in T_{P}\GL(n) \mid V \Sigma_0 + \Sigma_0V^{\top} = 0  \}.
  \end{equation}
\end{lemma}

\begin{proof}
  By \autoref{lem:finite_polar_cone_isomorphism}, $T_{P}K_{\lozenge} = \TSym{n}$.
  The result now follows from~\eqref{eq:metric_linear_OMT} since the orthogonal complement of $\TSym{n}$ with respect to the Frobenius inner product consists of skew-symmetric matrices.
\end{proof}

To compute \eqref{eq:grad_polar_cone_omt_finite} from~\eqref{eq:lifted_entropy_gradient} we therefore need to find $V\in N_{P}K_{\lozenge}$ such that
\begin{equation}
  P^{-1}\Sigma_0^{-1} - \Sigma_1^{-1}P + V \in \TSym{n}.
\end{equation}
Thus,
\begin{equation}
  P^{-1}\Sigma_0^{-1} - \Sigma_1^{-1}P + V = \Sigma_0^{-1}P^{-1} - P\Sigma^{-1} + V^{\top}
\end{equation}
and from the definition of $N_{P}K_{\lozenge}$ we get
\begin{align}
  P^{-1}\Sigma_0^{-1} - \Sigma_1^{-1}P + V &= \Sigma_0^{-1}P^{-1} - P\Sigma_1^{-1} - \Sigma_0^{-1}V\Sigma_0 \\
  \Sigma_0 V + V\Sigma_0 &= \Sigma_0 \left( \Sigma_0^{-1}P^{-1} - P^{-1}\Sigma_0^{-1}  + \Sigma_1^{-1}P -  P\Sigma_1^{-1} \right).
  \label{eq:sylverster_V}
\end{align}
Finally, we thereby obtain the gradient flow on $K_{\lozenge}$ as
\begin{equation}\label{eq:lifted_gradient_flow_polar_omt_finite}
  \dot P = P^{-1}\Sigma_0^{-1} - \Sigma_1^{-1}P + V,
\end{equation}
where $V$ is the solution to the Sylvester equation~\eqref{eq:sylverster_V}.
It is worth pointing out that if $\Sigma_0 = I$, then we obtain the simple equation
\begin{equation}\label{eq:lifted_gradient_flow_polar_omt_finite_simple}
  \dot P = P^{-1} - \frac{1}{2}\left( \Sigma_1^{-1}P + P\Sigma_{1}^{-1} \right).
\end{equation}

\begin{figure}
  \iftoggle{arxiv}{\includegraphics[scale=\iftoggle{aims}{1.0}{1.2}]{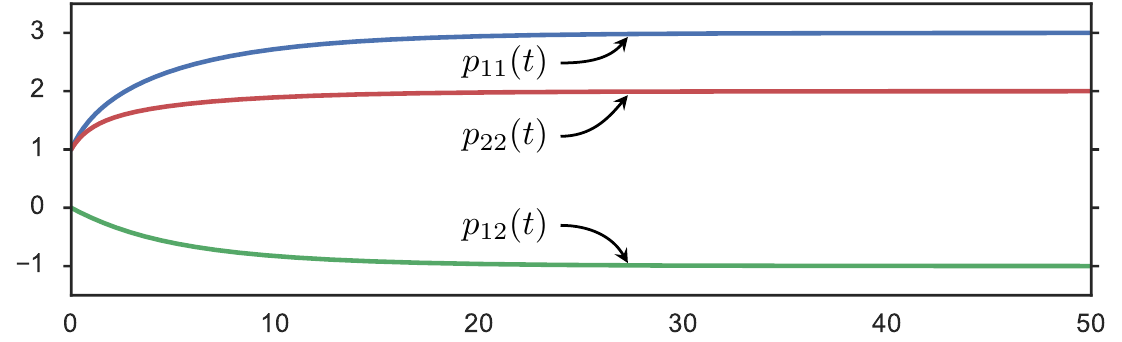}}{\input{fig_Pplot}}
  \caption{Evolution of the lifted gradient flow in \protect\autoref{ex:lifted_omt_finite}.
  Notice that $P(0)$ is the identity and that $P(t)$ converges towards $P_{\infty}$ in \protect\eqref{eq:Pinf_and_Qinf} as $t\to\infty$.
  }\label{fig:lifted_example_omt_finite}
\end{figure}

\begin{figure}
  \iftoggle{arxiv}{\includegraphics[scale=\iftoggle{aims}{1.0}{1.2}]{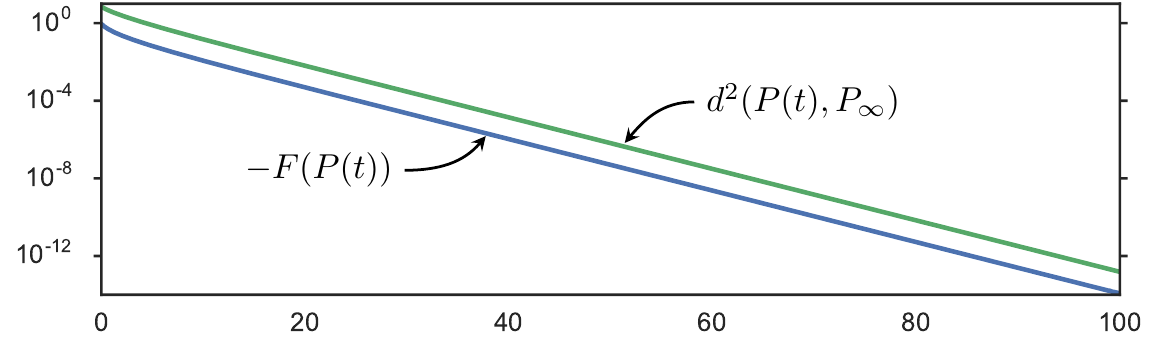}}{\input{fig_Perrorplot}}
  \caption{Convergence towards the limit $P_{\infty}$ of the lifted gradient flow in \protect\autoref{ex:lifted_omt_finite}.
  Notice that the convergence of both $-F(P(t))$ and $d^{2}(P(t),P_{\infty})$ as $t\to\infty$ is exponential, as fully explained by \protect\autoref{thm:limit_lifted_entropy_flow_omt_finite}.
  }\label{fig:lifted_example_omt_finite_errors}
\end{figure}

\begin{example}\label{ex:lifted_omt_finite}
  Let us give a simple example of how the polar decomposition can be numerically computed by solving the lifted gradient flow~\eqref{eq:lifted_gradient_flow_polar_omt_finite}.
  We use the same data as in \autoref{ex:vert_flow_omt_finite}.
  Thus,
  \begin{equation} \label{eq:Pinf_and_Qinf}
    P_{\infty} = \begin{pmatrix} 3 & -1 \\ -1 & 2 \end{pmatrix} \qquad\text{and}\qquad
    Q_{\infty} = \begin{pmatrix} \cos\theta & -\sin\theta \\ \sin\theta & \cos\theta \end{pmatrix}
  \end{equation}
  with $\theta = \pi/3$.
  Our objective is to compute the polar decomposition of
  \begin{equation}
    A \coloneqq P_{\infty}Q_{\infty},
  \end{equation}
  with $\Sigma_0 = I$.
  Since $P_{\infty}\in K_{\lozenge}$ and $Q\in \OO(n,I)$, we know, by construction, the polar components of $A$.
  We first set $\Sigma_1 = AA^{\top}$.
  The lifted gradient flow on $K_{\lozenge}$ is then given by~\eqref{eq:lifted_gradient_flow_polar_omt_finite_simple}, with initial data $P(0) = I$.

  We discretize the equation by the classical 4th order Runge--Kutta method \cite[\S\,322]{Bu2008}, with time-step~$\Delta t = 0.1$.
  The evolution of the elements
  \begin{equation}
    P(t) = \begin{pmatrix}
      p_{11}(t) & p_{12}(t) \\
      p_{12}(t) & p_{22}(t)
    \end{pmatrix}
  \end{equation}
  is shown in \autoref{fig:lifted_example_omt_finite};
  $P$ starts at the identity and converges towards $P_{\infty}$.
  The rate of convergence is shown in \autoref{fig:lifted_example_omt_finite_errors};
  both quantities $-F(P(t))$ and $d^{2}(P(t),P_{\infty})$ converge exponentially to zero as $t\to\infty$.
  We shall now give theoretical results that fully explain these numerical observations.
\end{example}

\begin{lemma}\label{lem:hessian_lifted_omt_finite}
  The Hessian of the lifted relative entropy functional \eqref{eq:lifted_rel_entropy_omt_finite} with respect to the Riemannian metric \eqref{eq:metric_linear_OMT} is given by
  \begin{equation}\label{eq:hessian_lifted_omt_finite}
    \hess(F)_{A}(\dot A,\dot A) = -\tr(\dot AA^{-1}\dot A A^{-1}) - \tr(\Sigma_0 \dot A^{\top}\Sigma_1^{-1}\dot A).
  \end{equation}
\end{lemma}

\begin{proof}
  Let $\gamma(t)$ be a geodesic curve with $\gamma(0)=A$ and $\dot\gamma(0) = \dot A$.
  From \eqref{eq:geodesic_eq_omt} and \autoref{lem:tg} we get that $\gamma(t) = A + t\dot A$.
  The Hessian of $F$ at $A$ is a bilinear form on $T_A\GL(n)$.
  Applied to $\dot A$ it is given by
  \begin{equation}
    \hess(F)_{A}(\dot A,\dot A) = \frac{\ud^{2}}{\ud t^{2}}\Big|_{t=0} F(\gamma(t)).
  \end{equation}
  From \eqref{eq:dFdt} we already have
  \begin{align}
    \frac{\ud}{\ud t} F(\gamma(t)) = -\tr(\Sigma_0 \gamma(t)^{\top}\Sigma_1^{-1}\dot{\gamma}(t)) + \tr(\dot{\gamma}(t) \gamma(t)^{-1}).
  \end{align}
  Differentiating once more and using that $\frac{\ud }{\ud t}\dot\gamma(t) = 0$, we get
  \begin{align}
    \frac{\ud^2}{\ud t^2}\Big|_{t=0} F(\gamma(t)) &= -\tr(\Sigma_0 \dot A^{\top}\Sigma_1^{-1}\dot A) + \tr\left(\dot A \frac{\ud}{\ud t}\Big|_{t=0}\gamma(t)^{-1}\right) \\
    &= -\tr(\Sigma_0 \dot A^{\top}\Sigma_1^{-1}\dot A) - \tr(\dot A A^{-1}\dot A A^{-1}),
  \end{align}
  which proves the result.
\end{proof}

With formula \eqref{eq:hessian_lifted_omt_finite} at hand, it is now straightforward to give a proof of \autoref{lem:pos_def_hessian_omt_finite}.

\begin{proof}[Proof of \protect\autoref{lem:pos_def_hessian_omt_finite}]
  Since the projection $\pi$ is a Riemannian submersion, it follows from general results in Riemannian geometry (see \cite[Ch.XIV\,\S 4]{La1999}) that the Hessian of the lifted relative entropy on $\GL(n)$ restricted to the horizontal distribution coincide with the Hessian of the relative entropy on $\Sym{n}$.
  That is,
  \begin{equation}\label{eq:hessian_correspondence}
    \hess(F)_A(\dot A,\dot A) = \hess(H)_{\pi(A)}(D\pi(A)\cdot \dot A, D\pi(A)\cdot \dot A), \quad
    \forall\; \dot A\in \Hor_A.
  \end{equation}
  Now, for $(\Sigma,\dot\Sigma)\in T\Sym{n}$, take $A\in \pi^{-1}(\Sigma)$ and $\dot A\in \Hor_A$ such that $D\pi(A)\cdot A = \dot\Sigma$.
  Then by \eqref{eq:hessian_correspondence} and \autoref{lem:hessian_lifted_omt_finite} we have
  \begin{align}\label{eq:hess_sigma_expr}
    \hess(H)_{\Sigma}(\dot\Sigma,\dot\Sigma) &= -\tr(\dot A A^{-1}\dot A A^{-1}) - \tr(\Sigma_0\dot A^{\top}\Sigma_1^{-1}\dot A) 
  \end{align}
  Since $\dot A$ is horizontal, it follows from the characterization~\eqref{eq:horizontal_linear_omt} that $\dot A A^{-1}$ is symmetric.
  Therefore
  \begin{equation}\label{eq:hess_first_est}
    \tr(\dot A A^{-1}\dot A A^{-1}) = \tr((\dot A A^{-1})^{\top}\dot A A^{-1}) = \norm{\dot A A^{-1}}_F^{2} \geq 0,
  \end{equation}
  where $\norm{\cdot}_F$ denotes the Frobienius norm.

  Next, since $\Sigma_1^{-1}$ is a symmetric, positive definite matrix, we can can use the Cholesky factorization to obtain $\Sigma_1^{-1} = L^{\top} L$, where $L$ is a lower triangular matrix with positive entries on the diagonal (a geometric description of the Cholesky factorization is given in~\autoref{sec:cholesky}).
  Then
  \begin{equation}\label{eq:hess_second_est}
    \tr(\Sigma_0\dot A^{\top}\Sigma_1^{-1}\dot A) = \tr(\Sigma_0 \dot A^{\top}L^{\top}L\dot A) =
    \mathcal G_{A}(L\dot A,L\dot A) \geq \alpha \mathcal G_{A}(\dot A,\dot A)
  \end{equation}
  where $\alpha$ is given by
  \begin{equation}\label{eq:alpha_inf}
    \alpha = \inf_{\dot A\in \Hor_A} \frac{\mathcal G_{A}(L\dot A,L\dot A)}{\mathcal G_{A}(\dot A,\dot A)}.
  \end{equation}
  That $\alpha >0$ follows since $L$ is non-degenerate.
  Combining \eqref{eq:hess_first_est} and \eqref{eq:hess_second_est} with \eqref{eq:hess_sigma_expr} gives the result.
\end{proof}

We shall now prove existence and uniqueness of a limit for the lifted gradient flow \eqref{eq:lifted_gradient_flow_polar_omt_finite}.
Again, the key is to give a positive bound on minus the Hessian of $F$ restricted to $K_{\lozenge}$.
However, we cannot directly use \autoref{lem:pos_def_hessian_omt_finite}, since the tangent spaces of $K_{\lozenge}$ are not horizontal (equation \eqref{eq:hessian_correspondence} cannot be used).
Nevertheless, we have the following result.

\begin{lemma}\label{lem:pos_def_hessian_lifted_omt_finite}
  Let $F|_{K_{\lozenge}}$ denote the lifted entropy functional \eqref{eq:lifted_rel_entropy_omt_finite} restricted to $K_{\lozenge}$.
  Then the Hessian of $F|_{K_{\lozenge}}$ with respect to the Riemannian metric \eqref{eq:metric_linear_OMT} restricted to $K_{\lozenge}$ fulfills the following inequality: there exists $\alpha>0$ such that
  \begin{equation}
    - \hess(F|_{K_{\lozenge}})_P(\dot P,\dot P) \geq \alpha \mathcal{G}_{P}(\dot P,\dot P), \quad
    \forall\, (P,\dot P)\in TK_{\lozenge} .
  \end{equation}
\end{lemma}

\begin{proof}
  From \autoref{lem:pos_def_hessian_omt_finite} we get
  \begin{equation}\label{eq:hess_FK_calc}
    \hess(F|_{K_{\lozenge}})_P(\dot P,\dot P) = -\tr(\dot P P^{-1}\dot P P^{-1}) - \tr(\Sigma_0\dot P\Sigma_1^{-1}\dot P).
  \end{equation}
  As in~\eqref{eq:hess_second_est}, the second term is estimated by
  \begin{equation}
    \tr(\Sigma_0\dot P\Sigma_1^{-1}\dot P) \geq \alpha \mathcal G_{P}(\dot P,\dot P),
  \end{equation}
  with
  \begin{equation}
    \alpha = \inf_{\dot P \in T_PK_{\lozenge}} \frac{\mathcal G_{P}(L\dot P,L\dot P)}{\mathcal G_{P}(\dot P,\dot P)} > 0,
  \end{equation}
  with the same $L$ as in~\eqref{eq:alpha_inf}, but infimum now over $T_PK_{\lozenge}$ instead of $\Hor_A$.

  For the first term of~\eqref{eq:hess_FK_calc} we cannot immediately say that it is positive, since $\dot P$ is not necessarily horizontal, so $\dot P P^{-1}$ is in general not a symmetric matrix.
  We can, however, use that $P$ \emph{itself} is symmetric positive definite.
  Indeed, let $P^{-1} = Z^\top Z$ be the Cholesky factorization of $P^{-1}$.
  Then
  \begin{equation}
    \tr(\dot P P^{-1} \dot P P^{-1}) = \tr(\dot P Z^{\top} Z \dot PZ^{\top} Z) = \tr(\underbrace{Z\dot P Z^{\top}}_S \underbrace{Z \dot PZ^{\top}}_S).
  \end{equation}
  Since $\dot P$ is a symmetric matrix, it follows that $S$ is symmetric.
  Therefore,
  \begin{equation}
    \tr(\dot P P^{-1} \dot P P^{-1}) = \tr(SS) = \tr(S^{\top}S) = \norm{S}_F^{2} \geq 0.
  \end{equation}
  This concludes the proof.
\end{proof}

Recall that the tangent vectors of $K_{\lozenge}$ are not necessarily horizontal.
Nevertheless, the tangent bundle $TK_{\lozenge}$ and the vertical distribution are transversal.

\begin{lemma}\label{lem:transversal_tangent_cone_linear_omt}
  Let $P\in K_{\lozenge}$.
  Then
  \begin{equation}
    T_{P}K_{\lozenge} \bigcap \Ver_P = \{ 0\}.
  \end{equation}
\end{lemma}

\begin{proof}
  Let $\dot P \in T_PK_{\lozenge}$ and assume $\dot P \in \Ver_P$.
  Then
  \begin{equation}
    \mathcal G_P(\dot P,X) = 0 \qquad\forall\, X\in \Hor_P .
  \end{equation}
  Since $X \in \Hor_P$ if and only if $X = SP$ for some $S\in\TSym{n}$, we get
  \begin{align}
    \tr(\Sigma_0 \dot P^{\top} S P) &= 0 \qquad\forall\, S\in \TSym{n} \\
     & \Updownarrow \\
    \tr(\Sigma_0 \dot P S P) &= 0 \qquad\forall\, S\in \TSym{n} .
  \end{align}
  Let $L^{\top}L$ be the Cholesky factorization of $\Sigma_0$.
  Taking $S=L^{\top}S' L$ for $S'\in\TSym{n}$, we can reformulate the condition as
  \begin{equation}
    \tr(L \dot P L^{\top} S' L P L^{\top}) = 0 \qquad\forall\, S'\in\TSym{n}.
  \end{equation}
  Taking $S' = L \dot P L^{\top}$ and using that $L$ is non-degenerate and $L P L^{\top}$ is symmetric positive definite, we get $\dot P = 0$.
  This concludes the proof.
\end{proof}


\autoref{lem:pos_def_hessian_lifted_omt_finite} and \autoref{lem:transversal_tangent_cone_linear_omt} implies the following result, explaining the observed convergence in \autoref{ex:lifted_omt_finite}.

\begin{theorem}\label{thm:limit_lifted_entropy_flow_omt_finite}
  The lifted entropy function~\eqref{eq:lifted_rel_entropy_omt_finite} restricted to $K_{\lozenge}$ admits a unique maximum $P_{\infty}\in K_{\lozenge}$.
  It fulfills
  \begin{equation}
    \pi(P_{\infty}) = P_{\infty}\Sigma_0P_{\infty} = \Sigma_1.
  \end{equation}
  Furthermore, for any initial data $P(0) \in K_{\lozenge}$, the lifted gradient flow~\eqref{eq:lifted_gradient_flow_polar_omt_finite} converges towards $P_{\infty}$ as $t\to \infty$, with estimates
  \begin{equation}
    F(P(t)) \geq \ee^{-2\alpha t} F(P(0)) \quad\text{and}\quad
    d^{2}(P(t),P_{\infty}) \leq \ee^{-2\alpha t} d^{2}(P(0),P_{\infty}),
  \end{equation}
  where $d(\cdot,\cdot)$ is the distance function of the metric~\eqref{eq:metric_linear_OMT} and $\alpha > 0$ is the constant in \autoref{lem:pos_def_hessian_lifted_omt_finite}.
\end{theorem}

\begin{proof}
  $F$ is strictly concave by \autoref{lem:pos_def_hessian_lifted_omt_finite}.
  Since, by \autoref{lem:polar_cone_isomorphism}, the set $K_{\lozenge}$ is convex, it follows that $F$ admits a unique maximum $P_{\infty} \in K_{\lozenge}$.

  From general results on gradient flows on Riemannian manifolds (see \cite[\S\,3.5]{Ot2001} for details), it also follows from \autoref{lem:pos_def_hessian_lifted_omt_finite} that
  \begin{equation}
    F(P_{\infty})-F(P(t)) \leq \ee^{-2\alpha t} \big( F(P_{\infty}) - F(P(0)) \big),
  \end{equation}
  and
  \begin{equation}
    d^{2}(P(t),P_{\infty}) \leq \ee^{-2\alpha t} d^{2}(P(0),P_{\infty}).
  \end{equation}
  We need to show that $\pi(P_{\infty}) = \Sigma_1$ and $F(P_{\infty}) = 0$.
  For this, we use that
  \begin{equation}
    \Pi_{P_{\infty}}\nabla_{\mathcal G} F(P_{\infty}) = 0,
  \end{equation}
  which, by \autoref{lem:transversal_tangent_cone_linear_omt}, implies
  \begin{align}
    \nabla_{\mathcal G} F(P_{\infty}) &= 0 \\
    &\Updownarrow \\
    \nabla_{\bar{\mathcal G}} H(\pi(P_{\infty})) &= 0.
  \end{align}
  Since $H$ is strictly concave with respect to $\bar{\mathcal G}$ (\autoref{lem:pos_def_hessian_omt_finite}), the last equality implies that $\pi(P_{\infty})$ must give the maximum value of the relative entropy.
  Thus, $F(P_{\infty}) = H(\pi(P_{\infty})) = 0$ and $\pi(P_{\infty}) = \Sigma_1$.
  This concludes the proof.
\end{proof}

We are now finally ready to give a geometric proof of \autoref{lem:shortest_geodesic_finite_dim}, based on the existence and uniqueness of the limit in \autoref{thm:limit_lifted_entropy_flow_omt_finite}

\begin{proof}[Geometric proof of \autoref{lem:shortest_geodesic_finite_dim}]

  First we show that $\pi|_{K_{\lozenge}}$ is surjective.
  Let $\Sigma_1 \in \Sym{n}$ and consider the lifted relative entropy gradient flow \eqref{eq:lifted_gradient_flow_polar_omt_finite}.
  By \autoref{thm:limit_lifted_entropy_flow_omt_finite} the limit of this flow gives an element $P_{\infty}\in K_{\lozenge}$ such that $\pi(P_{\infty}) = \Sigma_1$.
  Thus, $\pi|_{K_{\lozenge}}$ is surjective.

  Next, we show injectivity.
  Assume that $P' \in K_{\lozenge}$ fulfills $\pi(P') = \pi(P_{\infty}) = \Sigma_1$.
  Then, by \autoref{thm:limit_lifted_entropy_flow_omt_finite} the flow \eqref{eq:lifted_gradient_flow_polar_omt_finite} with $P(0) = P'$ converges towards $P_{\infty}$ as $t\to \infty$.
  Since $\Sigma_1$ is the maximum of $H$ we have for any $\dot\Sigma\in T_\Sigma\Sym{n}$ that
  \begin{equation}
    DH(\Sigma_1)\cdot \dot\Sigma = 0.
  \end{equation}
  Thus, for any $\dot P \in T_{P'}K_{\lozenge}$ we get
  \begin{equation}
    D F(P')\cdot \dot P = DH(\underbrace{\pi(P')}_{\Sigma_1})\cdot D\pi(P')\cdot \dot P =
    DH(\Sigma_0) \cdot (D\pi(P')\cdot \dot P )  = 0.
  \end{equation}
  This implies that $\Pi_{P'}\nabla_{\mathcal G} F(P') = 0$.
  Since the flow \eqref{eq:lifted_gradient_flow_polar_omt_finite} is
  \begin{equation}
    \dot P = \Pi_{P}\nabla_{\mathcal G} F(P)
  \end{equation}
  and $P(0) = P'$ it follows that the limit as $t\to\infty$ is given by $P'$.
  Thus $P' = P_{\infty}$ which proves that $\pi|_{K_{\lozenge}}$ is injective.
\end{proof}

\vspace*{-10pt}
\section{Fisher--Rao geometry and matrix decompositions}\label{sec:fisher_rao}

In this section we consider the same basic setting as in \autoref{sec:wasserstein} but with respect to a different Riemannian structure, namely the Fisher--Rao metric.
Whereas the Wasserstein metric is rooted in OMT, the Fisher--Rao metric originates from \emph{information geometry}---a branch of statistics that combines information theory and differential geometry.
Let us now give a brief introduction to both finite and infinite-dimensional information geometry.
For details, we refer to the monograph by Amari and Nagaoka~\cite{AmNa2000} (finite dimension) and the work by Khesin, Lenells, Misiolek, and Preston~\cite{KhLeMiPr2013} (infinite dimension).
Aspects of information geometry and multivariate Gaussian distributions, different from those presented here, are given by Barbaresco~\cite{Ba2013}.

\rednotes{

For details on the Lie group theory associated with the developments in this section, see \cite{Michor}.

}

Consider a probability distribution function $x\mapsto p(x,\theta)$ depending on parameters $\theta = (\theta^1,\ldots,\theta^k)$, for example, a multivariate Gaussian distribution depending on the covariance as discussed earlier. 
Fisher's information matrix~\cite{Fi1922} is given by
\begin{equation}\label{eq:fisher_information_matrix}
  \mathcal{I}_{ij}(\theta) = E\left[ \left(\frac{\pd}{\pd\theta^i} \ln p(\cdot,\theta)\right)\left(\frac{\pd}{\pd\theta^j} \ln p(\cdot,\theta)\right) \right].
\end{equation}
It measures the information about $\theta$ carried by a random variable with probability distribution $p(\cdot,\theta)\,\vol$.

Rao~\cite{Ra1945} interpreted $\mathcal{I}_{ij}(\theta)$ as a Riemannian metric on the ``manifold'' of probability distributions parameterized by~$\theta$.
This \emph{Fisher--Rao metric} is invariant under changes of coordinates $\theta\mapsto\tilde\theta$, which may appear curious at this stage.
However, the invariance becomes perfectly transparent when turning to the infinite-dimensional space $\Dens(\RR^{n})$ of all probability distributions on $\RR^{n}$, an approach pursued by Friedrich~\cite{Fr1991}.


\reviewertwo{
  Section 3 seems to focus on $\RR^n$ as base manifold, a few inconsistent references to a more general manifold $M$ are made in the beginning of the section. E.g. Def. 3.1 and the integral of the metric below.
  \\[2ex]
  DONE! These were typos that have now been fixed.
}

\begin{definition}\label{def:fisher_rao}
  The \emph{Fisher--Rao metric} is the Riemannian metric on $\Dens(\RR^{n})$ given by\footnote{Some authors use the factor $1/4$ in the definition of the Fisher--Rao metric.
  In our case, however, the formulas are easier without this factor.}
  \begin{equation}\label{eq:fisher_rao_metric_dens}
    \bar{\mathcal G}_{\mu}(\alpha,\beta) = \int_{\RR^{n}} \frac{\alpha}{\mu}\frac{\beta}{\mu}\,\mu .
  \end{equation}
\end{definition}



Notice something curious here: the Fisher--Rao metric is defined without using the Euclidean structure of~$\RR^{n}$.
(In contrast, the Wasserstein metric~\eqref{eq:wasserstein_metric_smooth_omt} uses the Euclidean structure through the gradient and divergence operators.)
\reviewertwo{
  “In contrast, the Wasserstein metric (7) uses the Euclidean structure through the gradient and the volume form dx.”: At least in the general (potentially non-smooth) setting, it is my understanding that the Wasserstein metric does not use a reference volume form. Choosing a reference measure simplifies solvability of the Monge-problem (because it must be solvable for two densities w.r.t. the same reference measure), but it does not affect the distance between measures and neither is there a preferred a-priori choice.
  \\[2ex]
  DONE! Indeed, if the two densities are represented as $\mu_0 = \rho_0\vol$ and $\mu_1 = \rho_1\vol$ the Wasserstein distance is of course independent of $\vol$.
  But, the point in the sentence above is to say that the Wasserstein distance depends on the Riemannian metric of the background manifold (in our case the Euclidean structure of $\RR^{n}$): that is the Wasserstein--Otto metric $\bar{\mathcal G}$ on $\Dens(\RR^{n})$ depends on the Euclidean structure.
  In contrast, the Fisher--Rao metric is independent of the Euclidean structure.
  The sentence above has been reformulated to better communicate this.
}%
As a consequence, it is invariant under arbitrary changes of coordinates, or, equivalently, under the pullback action of $\Diff(\RR^{n})$.
Explicitly, the invariance is seen as follows
\begin{equation}\label{eq:invariance_GF}
  \bar{\mathcal G}_{\mu}(\alpha,\beta) = \int_{\RR^{n}} \frac{\alpha}{\mu}\frac{\beta}{\mu}\mu = \int_{\RR^{n}} \varphi^* \left( \frac{\alpha}{\mu}\frac{\beta}{\mu}\mu \right) = \bar{\mathcal G}_{\varphi^*\mu}(\varphi^*\alpha,\varphi^*\beta).
\end{equation}

What is then the connection to Rao's original finite-dimensional metric given by the Fisher information matrix?
The answer is provided by the following result.

\begin{definition}\label{def:statistical_manifold}
  A \emph{statistical manifold} on $\RR^{n}$ is a submanifold of $\Dens(\RR^{n})$.
\end{definition}

\begin{proposition}\label{prop:link_finite_infinite_fisher_rao}
  Consider a finite-dimensional statistical manifold $\mathcal S\subset\Dens(\RR^{n})$, locally parameterized by $p(\cdot,\theta)\,\vol$.
  Then the Fisher--Rao metric~\eqref{eq:fisher_rao_metric_dens} restricted to $\mathcal S$, expressed in the local coordinates $\theta$, is given by Fisher's information matrix~\eqref{eq:fisher_information_matrix}.
\end{proposition}

\begin{proof}
  The expression $g_{ij}(\theta)$ for the Fisher--Rao metric expressed in local coordinates $\theta^1,\ldots,\theta^k$ is
  \begin{align}
    g_{ij}(\theta) &= \bar{\mathcal G}_{p(\cdot,\theta)\vol}\left(\frac{\pd p(\cdot,\theta)}{\pd \theta^i}\vol,\frac{\pd p(\cdot,\theta)}{\pd \theta^j}\vol\right)  \\
    &= \int_{\RR^{n}} \frac{\frac{\pd}{\pd{\theta^i}}p(x,\theta)}{p(x,\theta)}\frac{\frac{\pd}{\pd{\theta^j}}p(x,\theta)}{p(x,\theta)} p(x,\theta)\,\vol \\
    &= \int_{\RR^{n}} \left(\frac{\pd}{\pd\theta^i}\ln p(x,\theta)\right)\left(\frac{\pd}{\pd\theta^j}\ln p(x,\theta)\right) p(x,\theta)\,\vol \\
    &=  E\left[ \left(\frac{\pd}{\pd\theta^i} \ln p(x,\theta)\right)\left(\frac{\pd}{\pd\theta^j} \ln p(x,\theta)\right) \right] = \mathcal{I}_{ij}(\theta).
  \end{align}
  This concludes the proof.
\end{proof}




In this paper the primary example of a statistical manifold is, of course, the space of multivariate Gaussian distributions~$\mathcal{N}_{n}\simeq\Sym{n}$.
Let us now discuss this example in more detail.

When dealing with the Fisher--Rao metric on $\mathcal{N}_{n}$ it is convenient to use a different parameterization: instead of using the covariance matrix $\Sigma$ we use its inverse $W\coloneqq \Sigma^{-1}$.
The underlying reason is that the principle bundle structure associated with the Fisher--Rao geometry is based on a right action (pullback) instead of a left action (pushforward) as in the Wasserstein geometry.
Thus, the probability density function associated with $W\in\Sym{n}$ is given by
\begin{equation}\label{eq:gaussians3}
  p(x,W) \coloneqq \left( \sqrt{\frac{{\det(W)}}{{(2\pi)^n}}}\exp(-\frac{1}{2}x^{\top}W x) \right).
\end{equation}

\begin{lemma}\label{lem:sym_FR}
  The Fisher--Rao metric on $\mathcal{N}_n\simeq\Sym{n}$ is given by
  \begin{equation}\label{eq:FRsym}
    \bar{\mathcal G}_{W}(U,V) = \frac{1}{2}\tr(W^{-1}UW^{-1}V), \quad U,V\in \TSym{n}.
  \end{equation}
\end{lemma}

We shall prove this result in two ways; first by direct calculations, and then indirectly, by using the geometric invariance property, which gives the result up to multiplication by a scalar.

\begin{proof}[Direct proof of \protect\autoref{lem:sym_FR}]
  First, we rewrite \eqref{eq:gaussians3} as
  \begin{equation}
    p(x,W) = (2\pi)^{-n/2} \exp(\frac{1}{2}\log(\det(W))-\frac{1}{2}x^{\top} W x)
  \end{equation}
  Let $w_{ii'}$ denote the components of $W$.
  Then (using Einstein notation)
  \begin{align}
    \frac{\pd}{\pd w_{ii'}}p(x,W) &= p(x,W) \frac{1}{2}\frac{\pd}{\pd w_{ii'}} (\log(\det(W))-x^{\top}W x) \\
     &= p(x,W) \frac{1}{2} (\tr( W^{-1}\frac{\pd W}{\pd w_{ii'}}) - x^{\top} \frac{\pd W}{\pd w_{ii'}} x) \\
     &= p(x,W) \frac{1}{2} ( (W^{-1})_{jk} (\frac{\pd W}{\pd w_{ii'}})_{kj}  - (\frac{\pd W}{\pd w_{ii'}})_{kj} x_k x_j) \\
     &= p(x,W) \frac{1}{2} ( (W^{-1})_{i'i}  -  x_i x_{i'}).
  \end{align}
  From \autoref{prop:link_finite_infinite_fisher_rao} it then follows that the metric tensor is given by
  \begin{align}
    g_{ii'jj'}(W) &= \int_{\RR^{n}} \frac{1}{2^{2}}((W^{-1})_{i'i} - x_{i}x_{i'})((W^{-1})_{j'j} - x_{j}x_{j'})\, p(x,W)\,\vol \\
    \iftoggle{final}{}{
    &= \frac{1}{4}\left(
      (W^{-1})_{ii'}(W^{-1})_{jj'}
      -2 (W^{-1})_{ii'}(W^{-1})_{jj'}
      + \int_{\RR^{n}} x_{i}x_{i'}x_{j}x_{j'}\, p(x,W) \, \vol \right)
    \\
    }
    &= \frac{1}{4}\left(
      - (W^{-1})_{ii'}(W^{-1})_{jj'}
      + \int_{\RR^{n}} x_{i}x_{i'}x_{j}x_{j'}\, p(x,W) \, \vol \right)
    \\
    &= \frac{1}{4}\left(
      - (W^{-1})_{ii'}(W^{-1})_{jj'}
      + E[x_{i}x_{i'}x_{j}x_{j'}] \right).
  \end{align}
  Using Isserlis' theorem in statistics, the fourth order moments $E[x_{i}x_{i'}x_{j}x_{j'}]$ are given by
  \begin{equation}
    E[x_{i}x_{i'}x_{j}x_{j'}] = (W^{-1})_{ii'}(W^{-1})_{jj'} + (W^{-1})_{ij}(W^{-1})_{i'j'} + (W^{-1})_{ij'}(W^{-1})_{i'j}.
  \end{equation}
  We thereby get
  \begin{equation}
    g_{ii'jj'}(W) = \frac{1}{4}\left( (W^{-1})_{ij}(W^{-1})_{i'j'} + (W^{-1})_{ij'}(W^{-1})_{i'j} \right).
  \end{equation}
  Now,
\begin{align}
    \bar{\mathcal G}_{W}(U,V) &= g_{ii'jj'}(W) U_{ii'} V_{jj'}
    \\
    &= \frac{1}{4}\left( (W^{-1})_{ij}V_{jj'}(W^{-1})_{i'j'}U_{ii'} + (W^{-1})_{ij'}V_{jj'}(W^{-1})_{i'j}U_{ii'} \right) \\
    &= \frac{1}{4}\left( (W^{-1}V)_{ij'}(W^{-\top}U^{\top})_{j'i} + (W^{-1}V^{\top})_{ij}(W^{-\top}U^{\top})_{ji} \right)
  \end{align}
  \begin{eqnarray*}
    &= \frac{1}{4}\left( \tr(W^{-1}V W^{-\top}U^{\top}) + \tr(W^{-1}V^{\top}W^{-\top}U^{\top}) \right)\\
    &= \frac{1}{2} \tr(W^{-1}V W^{-1} U),
  \end{eqnarray*}
  where, in the last equality, we use that $W,U,V$ are symmetric matrices.
  This proves the result.
\end{proof}

\begin{proof}[Indirect proof of \protect\autoref{lem:sym_FR}]
  The action of $A\in\GL(n)$ lifted to act on $(W,U)\in T\Sym{n}$ is given by $(W,U)\cdot A = (A^\top W A,A^\top U A)$ (see \eqref{eq:right_GLn_action} below).
  We then have
  \begin{equation}
    \begin{split}
        \bar{\mathcal G}_{A^\top W A}(A^\top U A,A^\top V A) 
        &= \frac{1}{2}\tr((A^{\top}WA)^{-1}A^{\top}UA (A^{\top}WA)^{-1} A^{\top}VA) \\
        &= \frac{1}{2}\tr(A^{-1}W^{-1}A^{-\top}A^{\top}UA A^{-1}W^{-1}A^{-\top} A^{\top}VA) \\
        &= \frac{1}{2}\tr(A^{-1}W^{-1}UW^{-1}VA) \\
        & \text{(using cyclic property: $\tr(ABC) = \tr(BCA)$)} \\
        &= \frac{1}{2}\tr(W^{-1}UW^{-1}VA A^{-1}) \\
        &= \frac{1}{2}\tr(W^{-1}UW^{-1}V) = \bar{\mathcal G}_{W}(U,V)
    \end{split}
  \end{equation}
  \reviewertwo{
    The instantiation of the diffeomorphism action on densities in the linear category is only given in the proof of Lemma 3.4. Since it is used throughout this section it should have its own numbered equation.
    \\[2ex]
    DONE! The action is now given in the displayed equation \eqref{eq:right_GLn_action}.
  }%
  The metric $\bar{\mathcal G}$ is therefore invariant.
  From the classical uniqueness theorem by Cencov~\cite{Ch1982} it then follows that~$\bar{\mathcal G}$ is the Fisher--Rao metric up to multiplication by a positive scalar.
  \todo{This is not completely waterproof, since the group preserving $\mathcal N_n$ is larger than just $\GL(n)$. Fix this.}
  \iftoggle{final}{}{
  {\color{blue}
  To find the scalar, one can compare with \eqref{eq:fisher_rao_metric_dens} for $W=U=V=I$.
  Since
  \begin{align}
    \bar{\mathcal G}_{\mu}(\dot\mu,\dot\mu) &= \frac{1}{4}\int_{\RR^{n}} \left(\frac{\frac{\ud}{\ud t}p(x,W)}{p(x,W)}\right)^{2} p(x,W)\,\vol \\
    &= \frac{1}{16}\int_{\RR^{n}} \left(\tr(W^{-1}\dot W) - x^{\top}\dot W x\right)^{2} p(x,W)\,\vol \\
    &= \frac{1}{16} \tr(W^{-1}\dot W)^{2}\int_{\RR^{n}} p(x,W)\vol  \\ &\qquad\qquad
    -\frac{1}{8} \tr(W^{-1}\dot W) \int_{\RR^{n}} x^{\top}\dot W x \; p(x,W)\,\vol
    \\ &\qquad\qquad +\frac{1}{16}\int_{\RR^{n}} \left( x^{\top}\dot W x \right)^{2}p(x,W)\, \vol
    \\
    &= \frac{1}{16} \tr(W^{-1}\dot W)^{2} - \frac{1}{8} \tr(W^{-1}\dot W)^{2} + \frac{1}{16}\int_{\RR^{n}} \left( x^{\top}\dot W x \right)^{2}p(x,W)\, \vol
  \end{align}
  }}
\end{proof}

\reviewerone{
  Finally, the most interesting part of the paper for me personally was Lemma 3.4, which reproduces the natural metric on the space of Riemannian metrics, originally studied by Ebin in his thesis and later expanded upon by Freed/Groisser, Gil-Medrano/Michor, and Brian Clarke in his thesis. It is amusing that the metric on the space of densities is inherited by the metric on the space of metrics, but then bequeaths a metric on the space of (constant coefficient) metrics which essentially matches the original one. The authors mentioned above also give explicit formulas for the geodesics and curvature in this metric. In fact this also suggests that the descending metric from Section 2.3.2 could be used to generate a metric on the space of Riemannian metrics (or at least those which are pullbacks of a given metric under a diffeomorphism, a space relevant in Teichmuller theory as in
  Tromba’s book ``Teichmuller theory in Riemannian geometry''). I would be very curious to see what this geometry looks like, since it would seem to induce the well-known and physically important L2 geometry on the diffeomorphism group. This is much too big a project to include in the current paper, but at least Lemma 3.4 seems to deserve a remark along these lines, and I think it might be a nice project for the author to undertake in the future.
  \\[2ex]
  DONE! This is indeed an interesting topic to pursue! I am very grateful to the reviewer for sharing this idea. I have included it in the Outlook section, credited to the reviewer.
}

Of course, once a Riemannian metric is given, a succession of natural questions follows: the geodesic equation, a formula for geodesics, a formula for the distance, and a formula for the sectional curvature.
For $\Sym{n}$ equipped with the Fisher--Rao metric, these questions have been addressed in detail \cite{AtMi1981,BuRa1982,Bu1984,Sk1984,Ra1987,CaOl1990}.
Here, we give only a brief discussion.

First, a curve $W= W(t)$ with $t\in[0,1]$ is a geodesic if
\begin{equation}
  \frac{\ud}{\ud\epsilon}\Big|_{\epsilon = 0} \int_{0}^{1} \bar{\mathcal G}_{W+\epsilon\delta W} (\dot W +\epsilon\delta\dot W,\dot W+\epsilon\delta\dot W) \ud t = 0
\end{equation}
for all variations $\delta W = \delta W(t)$.
Using \autoref{lem:sym_FR} we get
\begin{align}
  \frac{\ud}{\ud\epsilon}\Big|_{\epsilon = 0}& \frac{1}{2}\int_{0}^{1}\tr((W+\epsilon\delta W)^{-1})(\dot W+\epsilon\delta\dot W) (W+\epsilon\delta W)^{-1})(\dot W+\epsilon\delta\dot W)  ) \ud t \\
  &= \int_{0}^{1}-\tr( W^{-1} \delta W W^{-1}\dot W W^{-1}\dot W ) + \tr( W^{-1}\delta\dot W W^{-1}\dot W ) \ud t \\
  &= \int_{0}^{1}-\tr( \delta W W^{-1}\dot W W^{-1}\dot W W^{-1}) - \tr( \delta W \frac{\ud}{\ud t}W^{-1}\dot W W^{-1}) \ud t \\
  &= \int_{0}^{1}-\tr( \delta W W^{-1}\dot W W^{-1}\dot W W^{-1}) - \tr( \delta W W^{-1}\ddot W W^{-1})  \\ & \qquad\qquad+ \tr( \delta W W^{-1}\dot W W^{-1}\dot W W^{-1}) + \tr( \delta W W^{-1}\dot W W^{-1}\dot W W^{-1}) \ud t  \\
  &= \int_{0}^{1} - \tr( \delta W W^{-1}\ddot W W^{-1})  + \tr( \delta W W^{-1}\dot W W^{-1}\dot W W^{-1}) \ud t .
\end{align}
Putting the last expression to zero and using the fundamental lemma of calculus of variations, we get, after multiplying from the left and right by $W$, the geodesic equation
\begin{equation}\label{eq:FR_gaussians_geo_eq}
  \ddot W - \dot W W^{-1} \dot W = 0.
\end{equation}

Next, we consider solutions to the geodesic equation \eqref{eq:FR_gaussians_geo_eq}.
We treat here only the case where the initial data is the identity $I$ (due to the invariance, each geodesic can be shifted, by the action of $\GL(n)$, to this case).
We claim that the solution is
\begin{equation}\label{eq:FR_gaussians_solution}
  W(t) = \exp(t \dot W_0)
\end{equation}
where $\dot W_0\in T_I\Sym{n} = \TSym{n}$ is the initial velocity and $\exp$ denotes the matrix exponential.
Let us now verify this.

\iftoggle{final}{We have}{Using $\ud/\ud t \exp(t A) = A\exp(t A) = \exp(t A) A$, we have}
\begin{equation}
  \dot W(t) = \dot W_0 \exp(t \dot W_0) \Rightarrow \ddot W(t) = \dot W_0 \dot W_0\exp(t \dot W_0) ,
\end{equation}
so
\begin{equation}
  \dot W(t) W(t)^{-1}\dot W(t) = \dot W_0 \exp(t \dot W_0) \exp(-t \dot W_0)  \dot W_0 \exp(t \dot W_0)
  = \dot W_0 \dot W_0 \exp(t\dot W_0).
\end{equation}
Therefore, $W(t)$ fulfill \eqref{eq:FR_gaussians_geo_eq} with $W(0) = I$ and $\dot W(0) = \dot W_0$.

\begin{remark}\label{rmk:commute_FR}
  Notice that the formula \eqref{eq:FR_gaussians_solution} for geodesics originating from the identity implies $[\dot W_0, W(t)] = 0$.
  Multiplying from the left by $W(t)^{-1}$ we get
  \begin{equation}
    W(t)^{-1}\underbrace{\dot W_0 W(t)}_{\dot W(t)} - \underbrace{\dot W_0 W(t)}_{\dot W(t)}W(t)^{-1} \iff [W(t)^{-1},\dot W(t)] = 0,
  \end{equation}
  so $W(t)^{-1}$ commutes with $\dot W(t)$ for all $t$.
  In \autoref{sec:eigen} below we shall give a geometric explanation of this observation.
\end{remark}

Finally, we now derive the geodesic distance function $\bar d(\cdot,\cdot)$.
If $W_0,W_1 \in \Sym{n}$, then $\bar d(W_0,W_1) = \bar d(I,W_1W_0^{-1})$ due to the invariance of the Fisher--Rao metric.
Therefore, it is enough to derive the distance from the identity to an element $W_1\in\Sym{n}$.
By definition, it is given by
\begin{equation}
  \bar d(I,W_1)^{2} = \int_0^{1} \bar{\mathcal G}_{\gamma(t)}(\dot\gamma(t),\dot\gamma(t))\ud t,
\end{equation}
where $\gamma(t)$ is the geodesic curve between $I$ and $W_1$.
From \eqref{eq:FR_gaussians_solution} it follows that
\begin{equation}\label{eq:FR_boundary_value_geodesics}
  \gamma(t) = \exp(t \log(W_1)),
\end{equation}
where $\log$ denotes the matrix logarithm.
We thereby get
\begin{align}
  \bar d(I,W_1)^{2} &= \int_0^{1} \frac{1}{2}\tr\Big(\exp(-t\log(W_1)) \exp(t\log(W_1)) \log(W_1) \\ &\qquad\qquad\quad\;  \exp(-t\log(W_1)) \exp(t\log(W_1)) \log(W_1) \Big)\ud t
  \\
  &=
  \int_0^{1} \frac{1}{2}\tr( \log(W_1) \log(W_1) )\ud t
  \\
  &=
  \frac{1}{2} \tr( \log(W_1) \log(W_1)).
\end{align}
The invariance then yields
\begin{equation}\label{eq:FR_gaussians_distance}
  \bar d(W_0,W_1)^{2} = \frac{1}{2} \tr( \log(W_1W_0^{-1}) \log(W_1W_0^{-1}) ).
\end{equation}
\iftoggle{final}{}{{\color{blue}
If $W_0$ and $W_1$ commute we get
\begin{align}
  \bar d(W_0,W_1)^{2}
  &= \frac{1}{2} \tr( \log(W_1 W_0^{-1}) (\log(W_1) - \log(W_0)) ) \\
  &= \frac{1}{2} \tr\Big( \big(\log(W_1) - \log(W_0)\big) \big(\log(W_1) - \log(W_0)\big) \Big)
\end{align}
}}

\iftoggle{final}{}{\newpage}
\subsection{Principal bundle structure}\label{sub:principal_bundle_qr}
So far we have that the statistical manifold of multivariate Gaussian distributions $\mathcal N_n$ is identified with $\Sym{n}$ by~\eqref{eq:gaussians3}, and the Fisher--Rao metric represented on $\Sym{n}$ is given by~\eqref{eq:FRsym}.
Let us now discuss the connection to matrix decompositions.

Recall from \autoref{sec:wasserstein} that a pivotal step to describe the geometry of the polar decomposition is to have (i) a principal bundle structure (to get fibers), and (ii) a Riemannian metric compatible with that bundle (to get horizontal geodesics).
The setting in this section follows the same pattern.
There are, however, some structural differences.

First, the principal bundle structure is based on pullback instead of pushforward:
\begin{equation}\label{eq:principal_bundle_fisher_infinite}
  \begin{tikzcd}
    \Diff_{\mu_0}(\RR^{n}) \arrow[hookrightarrow]{r}{} & \Diff(\RR^{n}) \arrow{d}{\pi} &  \\
    & \Dens(\RR^{n})
  \end{tikzcd}
\end{equation}
where
\begin{equation}
  \pi(\varphi) = \varphi^*\mu_0.
\end{equation}
If $\mu_0 = p(\cdot,W_0) \vol$ is a Gaussian distribution and $$\varphi(x) = Ax$$ with $A\in\GL(n)$, then
\begin{align}
  (\varphi^*\mu_0)(x) &= \varphi^*(p(\cdot,W_0)\vol)(x) = p(Ax,W_0)\det(A)\vol \\
  &= \left( \sqrt{\frac{\det(A)^2\det(W_0)}{(2\pi)^n}}\exp(-\frac{1}{2}(Ax)^{\top}W_0 Ax) \right)\vol \\
  &= \left( \sqrt{\frac{\det(A^{\top}W_0 A)}{(2\pi)^n}}\exp(-\frac{1}{2}x^{\top}A^{\top}W_0 Ax) \right)\vol \\
  &= p(x,A^{\top}W_0 A)\vol.
\end{align}
Thus, the action of $A\in\GL(n)$ on $W\in \Sym{n}$ is
\begin{equation}\label{eq:right_GLn_action}
  W \cdot A = A^{\top} W A
\end{equation}
and the isotropy group of $W$ is given by
\begin{equation}
  \OO(n,W^{-1}) = \{ Q\in\GL(n)\mid Q W^{-1} Q^{\top} = W^{-1} \}
  = \{ Q\in\GL(n)\mid Q^{\top} W Q = W \}.
\end{equation}
We now see that the finite-dimensional principal bundle corresponding to \eqref{eq:principal_bundle_fisher_infinite} is
\begin{equation}\label{eq:principal_bundle_fisher_finite}
  \begin{tikzcd}
    \OO(n,W_0^{-1}) \arrow[hookrightarrow]{r}{} & \GL(n) \arrow{d}{\pi} &  \\
    & \Sym{n}
  \end{tikzcd}
\end{equation}
\reviewertwo{
  eq. (32): it is pretty clear from the context what O(n,W0) is, but a formal definition should be given.
  \\[2ex]
  DONE! The definition is now given, and is consistent with that in Section 2.3.1.
}%
where
\begin{equation}\label{eq:symproj}
  \pi(A) = A^{\top}W_0 A.
\end{equation}

\iftoggle{final}{}{{\color{blue}
The derivative $D\pi(A)$ is computed as follows:
if $W = \pi(A)$ then
\begin{equation}
  \dot W = (A^{-1}\dot A)^{\top} A^{\top} W_0 A + A^{\top} W_0 A (A^{-1}\dot A),
%
\end{equation}
so
\begin{equation}\label{eq:pi_deriv_finite}
  D\pi(A)\cdot\dot A = V^{\top} W + W V, \qquad V=A^{-1}\dot A ,\; W = \pi(A).
\end{equation}
}}
The corresponding vertical distribution, i.e., the kernel of $D\pi$, is given by
\begin{equation}
  \Ver_A = \{ \dot A \in T_A\GL(n)\mid V^{\top} W_0 + W_0 V = 0, \; V=\dot A A^{-1} \}.
\end{equation}

\begin{remark}
  It is possible to describe the Fisher--Rao geometry using the pushforward bundle structure as in \autoref{sec:wasserstein}, but that reverses the order of the elements in the matrix decompositions discussed below.
  Also, the standard in the literature is to consider Fisher--Rao in a right-invariant setting, as in this section.
\end{remark}

The second structural difference is the following.
In Wasserstein geometry we start with a natural metric on $\Diff(\RR^{n})$ (or $\GL(n)$) and we show that it induces a metric on $\Dens(\RR^{n})$ (or $\Sym{n}$).
In Fisher--Rao geometry the situation is the reverse: we start with a metric on $\Dens(\RR^{n})$ (or $\Sym{n}$), so we need to construct a compatible metric on $\Diff(\RR^{n})$ (or $\GL(n)$).
Such a metric is heavily constrained: it must have left-invariant properties in order to descend with respect to the principal bundle~\eqref{eq:principal_bundle_fisher_infinite} (or~\eqref{eq:principal_bundle_fisher_finite}), but it must also be right-invariant in order to induce the Fisher--Rao metric on the base space.
That such metrics exist is not obvious a priori.
But they do exist.
Indeed, a two-parameter family is given in~\cite{Mo2015}, giving rise to \emph{optimal information transport} (OIT)---an information theoretic analogue of OMT where the polar decomposition is replaced by a different factorization of diffeomorphisms that solves the OIT problem.
However, in this paper we refrain from discussing OIT, and instead we focus on the finite-dimensional setting which gives us matrix decompositions.
Thus, our next objective is to derive a metric on $\GL(n)$ with the desired properties.

\subsection{\texorpdfstring{$QR$}{QR} (or Iwasawa) decomposition}\label{sec:qr}

%

\iftoggle{final}{}{aa}
\todo[inline]{Something about the history and usage of the QR factorization.
Also, say that it is called Iwasawa decomposition in Lie theory.}

Following \cite[\S\! 5.2]{Mo2015}, we shall construct a right-invariant Riemannian metric on $\GL(n)$ such that the projection~\eqref{eq:symproj} with $W_0=I$ becomes a Riemannian submersion with respect to the Fisher--Rao metric on $\Sym{n}$.
To this end, consider the two projection operators $\ell\colon \gl(n) \to \gl(n)$ and $\sigma\colon\gl(n)\to\gl(n)$ given by
\begin{equation}
  \ell(U)_{ij} = \left\{ \begin{matrix} 0 & \text{if $i \leq j$} \\ U_{ij} & \text{otherwise}\end{matrix} \right.
\end{equation}
and
\begin{equation}
  \sigma(U)_{ij} = U_{ij} + U_{ji}.
\end{equation}
In words, $\ell(U)$ selects the strictly lower diagonal entries.\footnote{For readers acquainted with \textsf{\footnotesize MATLAB} or SciPy, $\ell(U)$ corresponds to $\texttt{tril}(U,-1)$.}
We then define a right-invariant metric on $\GL(n)$ by
\begin{equation}\label{eq:GLmet}
  \mathcal G_{A}(U,V) = \frac{1}{2}\tr\left( \ell(UA^{-1})^{\top}\ell(VA^{-1}) + \sigma(UA^{-1})\sigma(VA^{-1}) \right).
\end{equation}
Throughout this section, the corresponding distance function is denoted $d(\cdot,\cdot)$.
Notice that $\GL(n)$ is not connected (it contains two connected components).
As is customary for non-connected Riemannian manifolds we define the distance between two elements of different components to be infinite.

Recall that the Lie algebra $\oo(n)$ of $\OO(n)\coloneqq \OO(n,I)$ consists of skew symmetric matrices.
The metric~\eqref{eq:GLmet} is constructed so that the orthogonal complement of $\oo(n)$ in $\gl(n)$ is given by the upper triangular matrices.
That is,
\begin{equation}
  \oo(n)^{\bot} = \{ U \in \gl(n)\mid U_{ij}=0\;\;\text{if}\;\;i > j\}. 
\end{equation}
\reviewertwo{
  I am more familiar with the notation $U^\bot$ for the orthogonal comple- ment, also affects next equation and the proof of Lemma 3.8.
  \\[2ex]
  DONE!
}%
\reviewertwo{
  should the constraint not be “if $i \geq j$”?
  \\[2ex]
  DONE! Typo: it should of course be $i > j$.
}%
The horizontal distribution is thereby given by
\begin{equation}
  \mathrm{Hor}_A = \{ U \in T_{A}\GL(n)\colon UA^{-1} \in \oo(n)^{\bot} \}.
\end{equation}
Here, the horizontal distribution fulfills an additional feature: since the space of upper triangular matrices $\oo(n)^{\bot}$ is a Lie algebra (closed under the commutator bracket), it follows that the horizontal distribution is integrable.
That is, it corresponds to the tangent space of foliated manifolds.
As we shall see below in the proof of \autoref{lem:polar_cone}, the manifold that cuts through the identify is given by the Lie group of upper triangular matrices with positive diagonal entries.


Let us now conclude the relation between the constructed metric $\mathcal G$ and the Fisher--Rao metric $\bar{\mathcal G}$.
\begin{lemma}\label{lem:GLnmetric}
  Let $U\in T_A\GL(n)$.
  Then
  \begin{equation}
    \mathcal G_{A}(U,U) = \bar{\mathcal G}_{\pi(A)}(D\pi(A)\cdot U,D\pi(A)\cdot U).
  \end{equation}
  if and only if $U\in \mathrm{Hor}_A$.
  That is, the projection $\pi\colon (\GL(n),\mathcal G) \to (\Sym{n},\bar{\mathcal G})$ given by~\eqref{eq:symproj} is a \emph{Riemannian submersion}.
\end{lemma}

\begin{proof}
  See~\cite[Prop.~5.8]{Mo2015}.
\end{proof}

A direct consequence is the following.

\begin{lemma}\label{cor:horizon_geodesics}
  Let $[0,1]\ni t\to \zeta(t)$ be a geodesic curve in $\Sym{n}$, and let $A\in\pi^{-1}(\zeta(0))$.
  Then there is a unique geodesic curve $[0,1]\ni t \to \gamma(t)$ in $\GL(n)$ fulfilling $\gamma(0)=A$, $\gamma'(t)\in \mathrm{Hor}$, and $\pi(\gamma(t))=\zeta(t)$.
\end{lemma}

\begin{proof}
  Follows since $\pi$ is a Riemannian submersion.
  See, e.g., \cite[Lem.~5.1]{Mo2015}.
\end{proof}

Recall the objective of this section: a geometric description of the $QR$ factorization of matrices.
For any $A\in\GL(n)$, we aim to construct, geometrically, a $Q\in\OO(n)$ and an upper triangular matrix $R$ such that $A=QR$.
The following central concept is analogous to the polar cone in \autoref{sec:wasserstein}.



\begin{definition}
  The \emph{upper triangular cone} in $\GL(n)$ is the subset
  \begin{equation}
    K_{\triangledown} = \{ R\in\GL(n)\mid d(I,R) \leq d(Q,R),\; \forall\; Q\in\OO(n) \}.
  \end{equation}
\end{definition}

\todo[inline]{Consider calling it \emph{slice} instead of \emph{cone}.}

In words, the upper triangular cone is the set of all elements in $\GL(n)$ whose closest point on the identity fiber $\pi^{-1}(I)=\OO(n)$ is the identity.
The connection to the $QR$ decomposition is established by the following result, which shows that $T_R K_{\triangledown} = \Hor_R$ and also motivates the name ``upper triangular cone''.

\begin{lemma}\label{lem:polar_cone}
  The upper triangular cone $K_{\triangledown}$ consists of upper triangular matrices with positive diagonal entries.
\end{lemma}

\begin{proof}
  The horizontal distribution $\Hor$ is right-invariant and given at the identity by the upper triangular matrices $\oo(n)^{\bot}$.
  But the space of upper diagonal matrices is a Lie algebra $\g$.
  Thus, the horizontal distribution is integrable and the corresponding Lie group $G$ connected to the identity is given exactly by the upper triangular matrices with positive diagonal entries.
  This proves that every element in $K_{\triangledown}$ is upper triangular with positive diagonal entries.
  That every element in $G$ is also an element in $K_{\triangledown}$ follows from \autoref{cor:horizon_geodesics} since any two elements in $\Sym{n}$ are connected by a unique Fisher--Rao geodesic.\todo[inline]{This should be more specific, for example by referring to the symmetric space structure.}
\end{proof}

As expected, the upper triangular cone $K_{\triangledown}$ gives us a subset of $\GL(n)$ that is in one-to-one relation with $\Sym{n}$.

\begin{lemma}\label{lem:section_FR_finite}
  $K_{\triangledown}\subset \GL(n)$ is a section of the bundle~\eqref{eq:principal_bundle_fisher_finite}.
  That is, the restriction
  \begin{equation}
    \pi\colon K_{\triangledown}\to \Sym{n}
  \end{equation}
  is an isomorphism.
  Furthermore, $\pi$ is an isometry between $(K_{\triangledown},\mathcal G)$ and $(\Sym{n},\bar{\mathcal G})$.
\end{lemma}

\begin{proof}
  Whereas $\pi$ being an isomorphism is clear from basic linear algebra, using for example the spectral decomposition, we shall give another, geometric proof, that is independent of results in linear algebra.

  First, surjectively follows from \autoref{cor:horizon_geodesics}, since any element in $\Sym{n}$ can be connected to~$I$ by a minimal geodesic, which is then lifted to a curve in $K_{\triangledown}$.

  Let us now prove injectivity.
  Let $R_1,R_2\in K_{\triangledown}$ and assume that $\pi(R_1) = \pi(R_2)$.
  Since $K_{\triangledown}$ is integrable, it follows that $A=R_1R_2^{-1}\in K_{\triangledown}$.
  Thus, there is a shortest horizontal geodesic $\gamma(t)$ connecting $I$ with $A$.
  Since $\pi$ is a Riemannian submersion, we get a corresponding shortest geodesic $\bar\gamma(t)$ on $\Sym{n}$ between $I$ and $\pi(A)$.
  But, since $R_1$ and $R_2$ belong to the same fiber, we have $\pi(A) = I$.
  Therefore, $\bar\gamma(t) = I$ for all~$t$, in particular, $\dot{\bar\gamma}(t) = 0$.
  Thus, $\dot\gamma(t)\in\Ver$ for all~$t$.
  But we also have that $\dot\gamma(t)\in\Hor$.
  Since $\Ver$ and $\Hor$ are orthogonal, this implies $\dot\gamma(t) = 0$.
  In turn, $\gamma(t) = I$ for all $t$, so $R_1 = R_2$.
  This proves injectivity.

  Finally, that $\pi$ is an isometry follows directly from the fact that the horisontal distribution is integrable and $T_R K_{\triangledown} = \Hor_R$, so any geodesic in $K_{\triangledown}$ descends to a geodesic on $\Sym{n}$ since $\pi$ is a Riemannian submersion.
  \rednotes{

  By definition of $K_{\triangledown}$ we therefore have minimal horizontal geodesics $\gamma_1(t)$ and $\gamma_2(t)$ such that $\gamma_1(0)=\gamma_2(0)=I$, $\gamma_1(1) = R_1$, and $\gamma_2(1) = R_2$.
  Since $\pi$ is a Riemannian submersion, we get two geodesics $\bar\gamma_1(t) \coloneqq\pi(\gamma_1(t))$ and $\bar\gamma_2(t)\coloneqq\pi(\gamma_2(t))$ on $\Sym{n}$.
  The start and end points of $\bar\gamma_1$ and $\bar\gamma_2$ are given by $I$ and $\pi(R_1)$ respectively.
  Now, geodesics on $\Sym{n}$ with respect to the Fisher--Rao metric $\bar{\mathcal G}$ are unique.
  This follows, for example, since $\bar{\mathcal G}$ is a \emph{Hessian metric}, see \autoref{sub:entropy_gradient_sym} below or \cite[Ch.~6]{Sh2007}.
  Thus, we have that $\bar\gamma_1(t) = \bar\gamma_2(t)$ for all $t\in[0,1]$.
  From \autoref{cor:horizon_geodesics} it then follows that $\gamma_1(1) = \gamma_2(1)$.
  This proves injectivity.
  }
\end{proof}

\begin{theorem}[QR decomposition]\label{thm:QR}
  Let $A\in\GL(n)$.
  Then there exists a unique $Q\in\OO(n)$ and a unique upper triangular matrix $R$ with positive diagonal entries
  such that $A=QR$.
\end{theorem}

\begin{proof}
  Let $W=\pi(A)$.
  From \autoref{lem:section_FR_finite} it follows that there is a unique $R\in K_{\triangledown}$ such that $\pi(R) = W$.
  By construction, $R$ and $A$ belong to the same fiber, so $Q\coloneqq AR^{-1}\in \OO(n)$.
  $Q$ is unique since $A$ is invertible and $R$ is unique.
  That~$R$ is upper triangular with positive diagonal entries follows from~\autoref{lem:polar_cone}.
\end{proof}

\subsubsection{Entropy Gradient Flow}\label{sub:entropy_gradient_sym}
Recall \autoref{sub:entropy_gradient_sym} where we considered the entropy gradient flow on $\Sym{n}$ with respect to the Wasserstein metric.
Here, we shall again derive the entropy gradient flow on $\Sym{n}$, but now with respect to the Fisher--Rao metric.

First, from \eqref{eq:rel_entropy_fin_dim} it follows that
the relative entropy functional in the variable $W$ is given by
\begin{equation}
  H(W) = \frac{n}{2} -\frac{1}{2}\tr(W_1W^{-1}) +\frac{1}{2}\log\left( \det(W_1 W^{-1})\right).
\end{equation}
Now,
\begin{equation}
  \begin{split}
    \frac{\ud}{\ud t}H(W) &= \frac{1}{2}\tr(W_1W^{-1}\dot W W^{-1}) -\frac{1}{2}\frac{\ud}{\ud t} \log(\det(W)) \\
    &= \frac{1}{2}\tr(W^{-1}W_1W^{-1}\dot W) - \frac{1}{2}\tr(W^{-1}\dot W) \\
    &= \frac{1}{2}\tr(W^{-1}W_1W^{-1}\dot W) - \frac{1}{2}\tr(W^{-1}W W^{-1}\dot W) \\
    &= \bar{\mathcal G}_W(W_1 - W,\dot W).
  \end{split}
\end{equation}
Thus, the gradient flow
\begin{equation}
  \dot W = \nabla_{\bar{\mathcal G}}H(W)
\end{equation}
\reviewerone{
  At the end of page 18 it started to seem strange that the gradient flow sometimes has a factor of 1/2 and sometimes not; this I think is a consequence of the author’s choice of scaled Riemannian metric.
  \\[2ex]
  DONE! The 1/2 is removed from all the gradient flows.
}
is given by
\begin{equation}\label{eq:entropy_gradient_flow_FR}
  \dot W =  W_1 -  W
\end{equation}
which, of course, has $W_1$ as a limit for any initial data.


The exponential convergence of \eqref{eq:entropy_gradient_flow_FR} towards $W_1$
is not a coincidence; there is geometry concealed here also.
Namely, the Fisher--Rao metric is a \emph{Hessian metric}, so it is given by the Hessian of a convex function.
The convex function is, in fact, minus the relative entropy functional (we encourage the reader to check this).
Thus,
\begin{equation}\label{eq:hessian_entropy_FR}
  \hess(H) = -\bar{\mathcal G}
\end{equation}
so, by the same standard technique as in \autoref{thm:limit_entropy_flow_omt_finite} we obtain a geometric proof for the exponential convergence of \eqref{eq:entropy_gradient_flow_FR} towards $W_1$ (which, of course, we already knew by basic linear ODE theory).
For more details on the Hessian structure of Fisher--Rao we refer to Shima~\cite[Ch.~6]{Sh2007}.
Gradient flows of Hessian metrics, in particular the non-smooth case, is discussed by Alvarez, Bolte, and Brahic~\cite{AlBoBr2004}.

\reviewerone{
  At the top of page 19 there is a very brief discussion of Hessian metrics, which is not a well-known topic. I think a little more time spent on this would be valuable, in particular summarizing why gradient flows of Hessian metrics are simple.
  \\[2ex]
  DONE! We now explain a little bit about Hessian structures.
  However, to fully understand the structure, one would have to discuss affine connections in detail, which is out of the scope of this paper.
  Therefore, the description is kept short.
}

\subsubsection{Lifted Gradient Flow} \label{subsub:lifted_gradientflow_QR}

In order to recover $R$ in the $QR$ decomposition by a horizontal gradient flow, we need to lift the flow~\eqref{eq:entropy_gradient_flow_FR} to the upper triangular polar cone~$K_{\triangledown}$.
Due to the result that the horizontal distribution is integrable, the situation is much simpler than in the corresponding lifted gradient flow in the Wasserstein geometry, treated in \autoref{subsub:lifted_gradient_omt_linear}.
Indeed, first define the lifted relative entropy functional on $\GL(n)$, given by
\begin{equation}\label{eq:lifted_rel_entropy_QR}
  F(A) \coloneqq H(\pi(A)) = H(A^{\top}A).
\end{equation}
Because $F$ is constant on the fibers (by construction), and because $T_R K_{\triangledown} = \Hor_R$, we automatically have that $\nabla_{\mathcal G} F(R) \in T_R K_{\triangledown}$ for any $R\in K_{\triangledown}$.
Thus, we never have to project onto the polar cone as for the Wasserstein geometry in \autoref{subsub:lifted_gradient_omt_linear}.
Furthermore, since $\pi$ is a Riemannian submersion, the gradient flow on $K_{\triangledown}$
\begin{equation}\label{eq:lifted_gradient_flow_qr}
  \dot R = \nabla_{\mathcal G} F(R)
\end{equation}
is in one-to-one relation with the entropy gradient flow \eqref{eq:entropy_gradient_flow_FR}.
In other words, $\gamma(t)$ is an integral curve of \eqref{eq:lifted_gradient_flow_qr} if and only if $\bar\gamma(t) \coloneqq \pi(\gamma(t))$ is an integral curve of \eqref{eq:entropy_gradient_flow_FR}.
We shall use this relation to derive the flow.


Since
\begin{equation}
  D\pi(R)\cdot \dot R = \dot R^{\top} R + R^{\top}\dot R
\end{equation}
we get from \eqref{eq:entropy_gradient_flow_FR} a flow for $R\in K_{\triangledown}$ as
\begin{equation}
  \dot R^{\top}R + R^{\top}\dot R =  W_1 -  R^{\top}R .
\end{equation}
From \cite{Ho1957} the equation for $\dot R$ can be rewritten as
\begin{equation}
  \dot R = \frac{1}{2} R^{-\top} (W_1 - R^{\top}R) + Z R
\end{equation}
where $Z$ is a skew-symmetric matrix.
We interpret $Z$ as a Lagrange multiplier to ensure that $\dot R$ is upper triangular.
Multiplying from the right by $R^{-1}$, we get
\begin{equation}
  \dot R R^{-1} = \frac{1}{2}\left( R^{-\top} W_1 R^{-1} -  I \right) + Z
\end{equation}
Since $\dot R R^{-1}$ is upper triangular,
it follows that $Z$ must be given by
\begin{equation}
  Z = \frac{1}{2}\left(\ell(R^{-\top} W_1 R^{-1} -  I)^{\top} - \ell(R^{-\top} W_1 R^{-1} -  I)\right).
\end{equation}
Using that $R^{-\top} W_1 R^{-1} -  I$ is a symmetric matrix, the lifted gradient flow \eqref{eq:lifted_gradient_flow_qr} then becomes
\begin{equation}\label{eq:lifted_gradientflow_QR_explicit}
  \dot R = (\mathsf{u}- \frac{1}{2}\mathsf{d})(R^{-\top} W_1 R^{-1} -  I) R,
\end{equation}
where $\mathsf{u}$ is the operator on matrices that selects the upper triangular entries\footnote{For readers acquainted with \textsf{\footnotesize MATLAB} or SciPy, $\mathsf{u}(B)$ corresponds to $\texttt{triu}(B)$.}
\begin{equation}
  \mathsf{u} \coloneqq \id - \ell,
\end{equation}
and $\mathsf{d}$ the operator that selects the diagonal entries\footnote{$\mathsf{d}(B)$ corresponds to $\texttt{diag}(\texttt{diag}(B))$ in \textsf{\footnotesize MATLAB} and SciPy.}.

\begin{theorem}\label{thm:limit_lifted_entropy_flow_QR}
  The lifted entropy function~\eqref{eq:lifted_rel_entropy_QR} restricted to $K_{\triangledown}$ admits a unique maximum $R_{\infty}\in K_{\triangledown}$.
  It fulfills
  \begin{equation}
    \pi(R_{\infty}) = R_{\infty}^{\top}R_{\infty} = W_1.
  \end{equation}
  Furthermore, for any initial data $R(0) \in K_{\triangledown}$, the lifted gradient flow~\eqref{eq:lifted_gradientflow_QR_explicit} converges towards $R_{\infty}$ as $t\to \infty$, with estimates
  \begin{equation}
    F(R(t)) \geq \ee^{-2 t} F(R(0)) \quad\text{and}\quad
    d^{2}(R(t),R_{\infty}) \leq \ee^{-2 t} d^{2}(R(0),R_{\infty}),
  \end{equation}
  where $d(\cdot,\cdot)$ is the distance function of the metric~\eqref{eq:GLmet}.
\end{theorem}

\begin{proof}
  From general results on gradient flows on Riemannian manifolds (see \cite[\S\,3.5]{Ot2001} for details), it follows from \eqref{eq:hessian_entropy_FR} that the entropy gradient flow \eqref{eq:entropy_gradient_flow_FR} converges towards the unique maximum $W_1$ with rates
  \begin{equation}
    H(W_1)-H(W(t)) \leq \ee^{-2 t} \big( H(W_1) - H(W(0)) \big),
  \end{equation}
  and
  \begin{equation}
    \bar d^{2}(W(t),W_1) \leq \ee^{-2 t} \bar d^{2}(W(0),W_1),
  \end{equation}
  where $\bar d$ denotes the Riemannian distance with respect to the Fisher--Rao metric, given by \eqref{eq:FR_gaussians_distance}.
  Since solution curves of the lifted gradient flow on $K_{\triangledown}$ project to solutions of the entropy gradient flow on $\Sym{n}$, the result in the theorem now follow from \autoref{lem:section_FR_finite}, as $\pi\colon K_{\triangledown}\to\Sym{n}$ is an isometry.
  This concludes the proof.
\end{proof}

\begin{example}\label{ex:lifted_QR}
  We give a simple example of how the $QR$ decomposition can be numerically computed by solving the lifted gradient flow~\eqref{eq:lifted_gradientflow_QR_explicit}.
  Let
  \begin{equation} \label{eq:Rinf_and_Qinf}
    R_{\infty} = \begin{pmatrix} 3 & -1 \\ 0 & 2 \end{pmatrix} \qquad\text{and}\qquad
    Q_{\infty} = \begin{pmatrix} \cos\theta & -\sin\theta \\ \sin\theta & \cos\theta \end{pmatrix}
  \end{equation}
  with $\theta = \pi/3$.
  Further, let
  \begin{equation}
    A \coloneqq Q_{\infty} R_{\infty}.
  \end{equation}
  Our objective is to compute the $QR$ factorization of $A$ (which, of course, we already know to be $Q_{\infty} R_{\infty}$).
  We first set $W_1 = \pi(A) = A^{\top}A$.
  The lifted gradient flow on $K_{\triangledown}$ is then given by~\eqref{eq:lifted_gradientflow_QR_explicit}.
  The initial data is $R(0) = I$.

  We discretize the equation by the classical 4th order Runge--Kutta method \cite[\S\,322]{Bu2008}, with time-step~$\Delta t = 0.1$.
  The evolution of
  \begin{equation}
    R(t) = \begin{pmatrix}
      r_{11}(t) & r_{12}(t) \\
      0 & r_{22}(t)
    \end{pmatrix}
  \end{equation}
  is shown in \autoref{fig:lifted_example_QR};
  $R$ starts at the identity and converges towards $R_{\infty}$.
  The rate of convergence is shown in \autoref{fig:lifted_example_QR_errors};
  both quantities $-F(R(t))$ and $d^{2}(R(t),R_{\infty})$ converge to zero with rate $\exp(-2t)$ as $t\to\infty$.
  These results are fully explained by \autoref{thm:limit_lifted_entropy_flow_QR}.
\end{example}

\begin{figure}
  \iftoggle{arxiv}{\includegraphics[scale=\iftoggle{aims}{1.0}{1.2}]{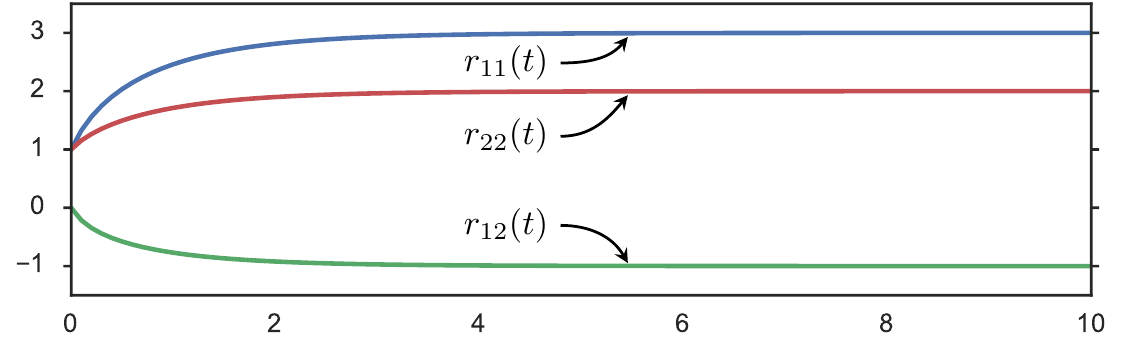}}{\input{fig_Rplot}}
  \caption{Evolution of the lifted gradient flow in \protect\autoref{ex:lifted_QR}.
  Notice that $R(0)$ is the identity and that $R(t)$ converges towards $R_{\infty}$ in \protect\eqref{eq:Rinf_and_Qinf} as $t\to\infty$.
  }\label{fig:lifted_example_QR}
\end{figure}

\begin{figure}
  \iftoggle{arxiv}{\includegraphics[scale=\iftoggle{aims}{1.0}{1.2}]{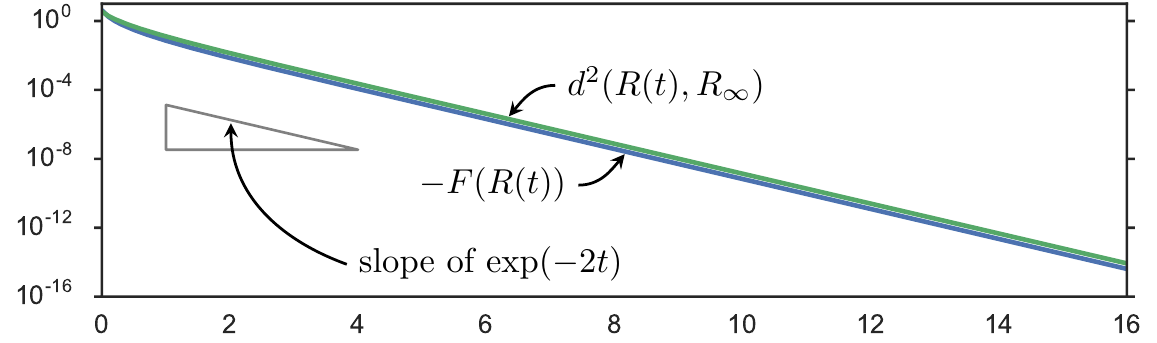}}{\input{fig_Rerrorplot2}}
  \caption{Convergence towards the limit $R_{\infty}$ of the lifted gradient flow in \protect\autoref{ex:lifted_QR}.
  The convergence of both $-F(R(t))$ and $d^{2}(R(t),R_{\infty})$ is exponential with  rate $\exp(-2t)$, as ensured by \protect\autoref{thm:limit_lifted_entropy_flow_QR}.
  }\label{fig:lifted_example_QR_errors}
\end{figure}



\subsection{Cholesky decomposition}\label{sec:cholesky}

The Cholesky decomposition is another classical matrix factorization.
We now show that it is a direct consequence of the geometry developed in the previous section.

\begin{theorem}[Cholesky decomposition]
  Let $W\in \Sym{n}$.
  Then there is a unique lower triangular matrix $L$ with positive entries on the diagonal such that $W=LL^{\top}$.
\end{theorem}

\begin{proof}
  By \autoref{lem:section_FR_finite} there exists a unique $R\in K_{\triangledown}$ such that $\pi(R) = W$.
  Thus,
  \begin{equation}\label{eq:Rfactor}
    R^{\top} R = W.
  \end{equation}
  Let $L = R^{\top}$.
  In terms of $L$, equation \eqref{eq:Rfactor} reads
  \begin{equation}\label{eq:cholesky}
    LL^{\top} = W.
  \end{equation}
  Since $R$ is upper triangular with positive entries on the diagonal, it follows that $L$ is lower triangular with positive entries on the diagonal.
  Thus, equation~\eqref{eq:cholesky} is the Cholesky decomposition of~$W$.
  $L$ is unique since $R$ is unique.
\end{proof}

  It follows from the geometry that one can use the Cholesky decomposition to obtain the $QR$ decomposition of $A\in\GL(n)$.
  Indeed:
  \begin{enumerate}
    \item set $W=A^{\top}A$;
    \item compute the Cholesky decomposition $W=LL^{\top}$;
    \item set $R=L^{\top}$ and compute $Q=AR^{-1}$.
  \end{enumerate}

Furthermore, notice that the lifted entropy gradient flow \eqref{eq:lifted_gradientflow_QR_explicit} gives the Cholesky factorization of $W_1$, by taking $L=R_{\infty}^{\top}$.

\subsection{Spectral decomposition}\label{sec:eigen}
In this section we show how the Fisher--Rao geometry of $\Sym{n}$ is related to the spectral decomposition.
Our objective is to give a geometric description of the classical result that every $W\in\Sym{n}$ can be factorized as $W=Q\Lambda Q^{\top}$ where $\Lambda$ is diagonal with positive entries and $Q\in\OO(n)$.
\reviewertwo{
  Section 3.3: While I was able to understand the general idea of this section, I was not able to properly follow it to the end with reasonable effort. Admittedly, this may be due to my limited experience in this domain. Maybe this section can be made more accessible.
  \\[2ex]
  DONE! The section is now extended and largely rewritten. Hopefully, it is now more accessible.
}%

From the principal bundle structure~\eqref{eq:principal_bundle_fisher_finite} it follows that $\Sym{n}$ is a \emph{homogeneous space} with respect to the left action of $\OO(n)$ on $\GL(n)$, i.e., $\Sym{n}\simeq \OO(n)\backslash \GL(n)$. 
Recall that the Fisher--Rao metric is obtained from the metric $\mathcal G$ on $\GL(n)$ given by \eqref{eq:GLmet}.
The remarkable property of $\mathcal G$ is that even after exhausting its left-invariant properties, to yield the Fisher--Rao metric on $\OO(n)\backslash \GL(n)$, there are still right-invariant properties available.
It is therefore natural to continue exhausting the invariance properties by a second quotient, but now from the right.
In other words, to study the double quotient space $\OO(n)\backslash \GL(n)/\OO(n) \simeq \Sym{n}/\OO(n)$.
Roughly speaking, the results in this section are that $\Sym{n}/\OO(n)$ can be identified with the space of eigenvalues and that the Fisher--Rao metric induces horizontal directions, given by diagonal matrices.
This gives rise to the spectral decomposition, much like the situation of the polar and $QR$ decompositions.
There is, however, one important difference: the quotient $\Sym{n}/\OO(n)$ is not a manifold.
This complicates things, but we shall avoid getting technical.
For full details we refer to Michor~\cite[Ch.~VI]{Mi2008}.


\subsubsection{Double Bundle Structure}\label{subsub:double_bundle_structure}
Consider the right action of $\OO(n)$ on $\Sym{n}$, given by the map
\begin{equation}\label{eq:so3_right_action}
  \Phi\colon \OO(n)\times \Sym{n} \ni (Q,W) \mapsto Q^\top W Q.
\end{equation}
We say that $W$ and $W'$ belong to the same \emph{orbit} if there exists a $Q\in\OO(n)$ such that $W' = Q^{\top}W Q$.
The set of orbits is therefore the quotient space
\begin{equation}
  \Sym{n}/\OO(n) = \{ \Phi(\OO(n),W)\colon W\in\Sym{n} \}.
\end{equation}
Contrary to the situation in \autoref{sec:qr}, the action is not free, so the orbits are in general not isomorphic to each other: different orbits may have different dimensions.
For example, the orbit of the identity matrix is zero dimensional since $Q^\top I Q = I$ for any $Q\in\OO(n)$.
Consequently, the action~\eqref{eq:so3_right_action} does not give rise to a principal bundle, as before.
The highest dimensional orbits are those for which the eigenvalues of $W$ are all different; this is the generic case.
Such orbits have the same dimension as $\OO(n)$, namely $n(n-1)/2$.

We now give a different, more intuitive way to understand the orbits. 
Elements in $\Sym{n}$ can be thought of as $n$-dimensional ellipsoids embedded in $\RR^{n}$ and the action of $\OO(n)$ corresponds to rotations and reflections.
Rotation about an axis of an ellipsoid with all different semi-axes ``changes'' it.
\reviewertwo{
  Several times infinitesimal rotations are mentioned. I don’t see why we need to consider the infinitesimal limit here.
  \\[2ex]
  DONE! The point of using infinitesimal was that some global rotations doesn't change the ellipsoid even if the semi-axes are all different (for example rotation by $\pi$).
  As the reviewer points out, to use infinitesimal here is confusing, so now it instead reads rotation about an axis.
}%
If, however, the ellipsoid has two semi-axes of the same length, then rotations in the plane spanned by these axes does not change it, so the action is singular in such directions.
Here is the key point: the orbits themselves represent ellipsoids without reference to orientation.
\reviewertwo{
  “the orbits themselves represent ellipsoids without reference to attitude”: from my interpretation of this sentence it seems to me that “attitude” should be replaced by e.g. “coordinate system” or “orientation”.
  \\[2ex]
  DONE! Reformulated as suggested.
}%
In other words,
\begin{center}
  space of orbits $=$ space of ellipsoids.
\end{center}
As we shall see later in this section, the Fisher--Rao metric induces a natural metric on the space of ellipsoids; a formal statement, since the quotient $\Sym{n}/\OO(n)$ has singular points, so it is not a manifold.

The next question is how to work with the space of orbits $\Sym{n}/\OO(n)$.
One possibility is to represent an element in $\Sym{n}/\OO(n)$ by $n$ strictly positive real numbers representing the lengths of the semi-axes of the corresponding ellipsoid.
This description, however, is redundant, since any permutation give the same ellipsoid.
Hence, the correct space is $(\RR^{+})^n/\mathfrak{S}_n$, where the symmetric group $\mathfrak{S}_n$ consists of all permutations of~$n$ elements.
This shows that $(\RR^{+})^n$ is an $n!$ covering of $\Sym{n}/\OO(n)$.
\reviewerone{
  In the second-last paragraph of page 20 it is mentioned that there is an n-covering. I was puzzled here since the argument suggests it should be an n!-covering. DONE!
}
\todo[inline]{Mention Voronoi cone and Weyl chamber (as Peter Michor mentioned).}

There is also another, more explicit way to work with $\Sym{n}/\OO(n)$, namely it is isomorphic to the set of monic polynomials with positive roots, given by
\begin{equation}
  \poly{n} = \{p\mid p(\lambda) = \lambda^n + \sum_{k=0}^{n-1} a_k \lambda^k, \; a_k\in\RR,\; p^{-1}(\{0\})\subset \RR^{+} \}.
\end{equation}
The projection $\varpi\colon \Sym{n}\to \poly{n}$ is given by
\begin{equation}
  \varpi(W) = \lambda\mapsto \det(\lambda I-W).
\end{equation}
Thus, $\varpi(W)$ is the characteristic polynomial of~$W$, and consequently the roots of $\varpi(W)$ are the eigenvalues of~$W$.

To summarize, and connect again to $\GL(n)$, the combination of the left and right action of $\OO(n)$ on $\GL(n)$ gives rise to the following `double' bundle:
\begin{equation}\label{eq:principal_bundle_fisher_finite_double}
  \begin{tikzcd}
    \OO(n) \arrow[hookrightarrow]{r}{} & \GL(n) \arrow{d}{\pi} \arrow[hookleftarrow]{r}{} &  \OO(n) \arrow[leftrightarrow]{d}{\id} \\
    & \Sym{n} \arrow[hookleftarrow]{r}{} \arrow{d}{\varpi} & \OO(n) \\
    & \poly{n} &
  \end{tikzcd}
\end{equation}

\subsubsection{Descending Metric and the Spectral Theorem}\label{subsub:descending_metric_diag}

We have already seen that the metric $\mathcal G$ on $\GL(n)$ descends to the Fisher--Rao metric on $\Sym{n}$.
Since the Fisher--Rao metric on $\Sym{n}$ is invariant with respect to the right action~\eqref{eq:so3_right_action} it formally descends to a ``metric'' on $\poly{n}$ (recall that $\poly{n}$ is not a manifold).
Let us now derive what it is.

To start, we need the vertical directions of the second projection $\varpi$, i.e., the kernel of~$D\varpi$.
From \eqref{eq:so3_right_action} we see that the infinitesimal action of $\xi\in\oo(n)$ on $W\in\Sym{n}$ is
\begin{equation}\label{eq:inf_action_on_sym}
  -\xi W + W\xi .
\end{equation}
Thus, the vertical directions at $W$ are given by
\begin{equation}
  \Ver_W = \{ -\xi W + W\xi \mid \xi\in\oo(n) \}.
\end{equation}

Let us now derive the horizontal directions, i.e., the directions orthogonal to $\Ver_W$.
We have that
\begin{equation}
  \begin{split}
    \bar{\mathcal G}_W(-\xi^\top W + W\xi,S) &= \frac{1}{2}\tr(W^{-1}(-\xi W + W\xi)W^{-1}S) \\
    &= \frac{1}{2}\tr(-W^{-1}\xi S+\xi W^{-1} S) \\
    &= \frac{1}{2}\tr(\xi (- S W^{-1}+ W^{-1} S)).
  \end{split}
\end{equation}
The condition on $S$ for this expression to vanish for any $\xi\in\oo(n)$ is that $-S W^{-1}+ W^{-1} S$ is a symmetric matrix.
In turn, this implies
\begin{equation}
   S W^{-1} - W^{-1} S = 0 \iff [S,W^{-1}] = 0.
\end{equation}
Thus, the horizontal directions at $W$ are given by
\begin{equation}
  \Hor_W = \{ S\in\TSym{n}\mid [S,W^{-1}]=0 \}.
\end{equation}
Consequently, $S\in \Hor_W$ if and only if $S$ and $W^{-1}$ are simultaneously diagonalizable.
Let us now discuss two special cases.

First, when $W = I$.
Since any matrix commutes with $I^{-1} = I$, we get that the horizontal directions at $I$ fills the entire tangent space: $\Hor_I = T_I\Sym{n}$.
Thus, any geodesic originating from the identity is horizontal.
Since an initially horizontal geodesic remains horizontal at all times \cite{He1960}, it follows that if $W(t)$ is a geodesic with $W(0) = I$, then
\begin{equation}
  [\dot W(t),W(t)^{-1}] = 0, \quad\forall\, t.
\end{equation}
Thus, as declared, we get a geometric explanation of the observation in \autoref{rmk:commute_FR}.

Second, when $W=\Lambda$ is a diagonal matrix with positive entries.
Since two diagonal matrices always commute, we immediately get that all diagonal matrices are contained in $\Hor_\Lambda$.
Let $\diag{n}\subset \Sym{n}$ be the submanifold of diagonal matrices with positive entries.
Notice that $\diag{n}$ itself is a Riemannian manifold, as it inherits, by restriction, the Fisher--Rao metric on $\Sym{n}$.
The tangent space $T_\Lambda\diag{n}$ consists of all diagonal matrices, so $T_\Lambda\diag{n} \subset \Hor_\Lambda$.
Thus, since initially horizontal geodesics remain horizontal, we get the following result.
\begin{lemma}\label{lem:diag_tg_FR}
  $\diag{n}$ is a totally geodesic submanifold of $\Sym{n}$.
  That is, if $\gamma(t)$ is a geodesic in $\diag{n}$, then $\gamma(t)$ is also a geodesic in $\Sym{n}$.
\end{lemma}
This result also follows directly from the geodesic equation \eqref{eq:FR_gaussians_geo_eq}.
We now give a characterization of $\Hor_\Lambda$ for a generic $\Lambda$.

\begin{lemma}\label{lem:hor_diag_FR_generic}
  Let $\Lambda\in\diag{n}$ be generic (all its entries are different).
  Then
  \begin{equation}
    \Hor_\Lambda = T_\Lambda\diag{n}.
  \end{equation}
  That is, $\Hor_\Lambda$ is the space of all diagonal matrices.
\end{lemma}

\begin{proof}
  If $\Lambda\in\diag{n}$ is generic, then from \autoref{subsub:double_bundle_structure} the orbit of $\Lambda$ has the maximum dimension $n(n-1)/2$.
  Thus, the dimension of $\Ver_\Lambda$ is also $n(n-1)/2$, so
  \begin{equation}
    \dim(\Hor_\Lambda) = \dim(T_\Lambda\Sym{n}) - \dim(\Ver_\Lambda) = \frac{n(n+1)}{2} - \frac{n(n-1)}{2} = n.
  \end{equation}
  Since the dimension of both $T_\Lambda\diag{n}$ and $\Hor_\Lambda$ is $n$ and since $T_\Lambda\diag{n}\subset \Hor_\Lambda$, it follows that $\Hor_\Lambda = T_\Lambda\diag{n}$.
\end{proof}


\reviewerone{
  I was confused by the discussion in the middle of page 21. W goes from diagonalizable to diagonal: “Since W is always diagonalizable, it follows that the space of diagonal matrices belong to HorW .” This is probably a minor point but I got a bit stuck on it.
  \\[2ex]
  DONE! Indeed, the formulation was very confusing. It has now been changed completely.
}

From~\autoref{lem:diag_tg_FR} we have that the submanifold $\diag{n}$ of diagonal matrices with positive entries is totally geodesic in~$\Sym{n}$.
From our discussion it is also clear that it is tangential to the (non-regular) horizontal bundle $\Hor$ and that $\Hor_\Lambda = T_\Lambda\diag{n}$ at generic points (\autoref{lem:hor_diag_FR_generic}).
It is therefore natural to think of $\diag{n}$ as the analogue of the upper triangular cone $K_{\triangledown}$ (or the polar cone in \autoref{sec:wasserstein}).

Let us now give the key result leading to the spectral decomposition.
It is the analog of \autoref{lem:shortest_geodesic_finite_dim} and \autoref{lem:section_FR_finite}, although the result is slightly weaker because of the singular action (we do not have injectivity).

\begin{lemma}\label{lem:surjective}
  The restricted mapping
  \begin{equation}
    \varpi\colon \diag{n}\to \poly{n}
  \end{equation}
  is surjective.
\end{lemma}

\begin{proof}
  Let $p\in \poly{n}$.
  Then $p$ has $n$ real positive roots $\lambda_1,\ldots,\lambda_n$.
  The corresponding diagonal matrix $\Lambda\in\diag{n}$ then fulfills $\varpi(\Lambda) = p$.
\end{proof}

\begin{theorem}[Spectral decomposition]\label{thm:spectral_decomposition}
  Let $W\in\Sym{n}$.
  Then there exists $Q\in\OO(n)$ and a diagonal matrix $\Lambda$ with positive entries such that
  \begin{equation}
    W = Q^\top\Lambda Q .
  \end{equation}
\end{theorem}

\begin{proof}
  Let $p=\varpi(W)$.
  From \autoref{lem:surjective} we get that there exists $\Lambda \in \diag{n}$ such that $\varpi(\Lambda) = p$.
  Since $\Lambda$ and $W$ belong to the same orbit, there exists a $Q\in\OO(n)$ such that $W = Q^\top\Lambda Q$.
\end{proof}

\begin{remark}
  The non-injectivity of $\varpi$ in \autoref{lem:surjective} is the reason that the spectral decomposition is not unique.
  This non-uniqueness is directly related to $(\RR^{+})^{n}\simeq \diag{n}$ being an $n!$ covering of $\Sym{n}/\OO(n)\simeq \poly{n}$.
\end{remark}

So far, we did not state what the Riemannian metric on $\poly{n}$ is.
Actually, strictly speaking we cannot, since $\poly{n}$ is not a manifold (we do not want to get into technical details of Riemannian orbifolds).
But, since $\diag{n}$ is a manifold and also an $n!$ covering of $\poly{n}$, we derive the geodesics on $\diag{n}$ instead.

From \autoref{lem:sym_FR} it is clear that the Fisher--Rao metric on $\diag{n}$, expressed in the diagonal entries $\lambda_1,\ldots,\lambda_n$, is given by
\begin{equation}\label{eq:met_diag}
  \hat{\mathcal G}_\Lambda\big( (a_1,\ldots,a_n),(b_1,\ldots,b_n) \big)
   = \frac{1}{2}\sum_{i=1}^{n} \frac{a_i b_i}{\lambda_i^{2}}.
\end{equation}
Since, by \autoref{lem:diag_tg_FR}, $\diag{n}$ is totally geodesic in $\Sym{n}$, it follows directly from \eqref{eq:FR_gaussians_geo_eq} that the geodesic equation is
\begin{equation}
  \ddot \lambda_i - \frac{\dot\lambda_i^{2}}{\lambda_i} = 0.
\end{equation}
Let us now derive its Hamiltonian form, without using the reference to \eqref{eq:FR_gaussians_geo_eq}.

The Hamiltonian corresponding to the Lagrangian $L(\Lambda,\dot\Lambda) = \hat{\mathcal{G}}_\Lambda(\dot\Lambda,\dot\Lambda)$ is given by
\begin{equation}
  H\big((\lambda_1,\ldots,\lambda_n),(m_1,\ldots,m_n) \big) = \frac{1}{2}\sum_{i=1}^{n} \lambda_i^{2} m_i^{2},
\end{equation}
where $m_i = \dot\lambda_i/\lambda_i^{2}$ are the momentum variables.
Since $H$ is separable in the index $i$ (the $i$th term only depends on $\lambda_i$ and $m_i$), it follows that the associated Hamiltonian system decouples into $n$ independent systems, each given by
\begin{equation}\label{eq:horizontal_diag}
  \begin{split}
    \dot \lambda &= m\lambda^2 \\
    \dot m &= -m^2\lambda.
  \end{split}
\end{equation}
It is straightforward to check that the quantity $\lambda m$ is a first integral.
From
\begin{equation}
  \lambda(t)m(t) = \frac{\dot\lambda(t)}{\lambda(t)} = -\frac{\dot m(t)}{m(t)} = \lambda(0)m(0)
\end{equation}
it then follows that the solution to \eqref{eq:horizontal_diag} is
\begin{equation}
  \lambda(t) = \lambda(0)\exp(t \lambda(0) m(0)),
  \quad
  m(t) = m(0)\exp(-t  \lambda(0) m(0)).
\end{equation}
In particular, for every $l_0,l_1 > 0$, there is a unique geodesic curve $\lambda(t)$ such that $\lambda(0) = l_0$ and $\lambda(1)=l_1$.
It is given by
\begin{equation} \label{eq:diag_boundary_value_geodesics}
  \lambda(t) = l_0 \exp\left(t \log\left(\frac{l_1}{l_0}\right)\right) = l_0 \left( \frac{l_1}{l_0} \right)^{t}.
\end{equation}
We encourage the reader to compare this formula with \eqref{eq:FR_boundary_value_geodesics}, which corresponds to $l_0=1$.
The phase diagram of~\eqref{eq:horizontal_diag} is illustrated in \autoref{fig:horizontal_diag}.

\reviewerone{
  Where is the geodesic equation (40) coming from? Is this a special case of another geodesic equation, or should it be immediately derivable from the Fisher-Rao metric (30)? And what is the Hamiltonian?
  \\[2ex]
  DONE! The geodesic equation is now described in much more detail.
}%

\begin{figure}
  \centering
  \iftoggle{arxiv}{\includegraphics[scale=\iftoggle{aims}{1.0}{1.2}]{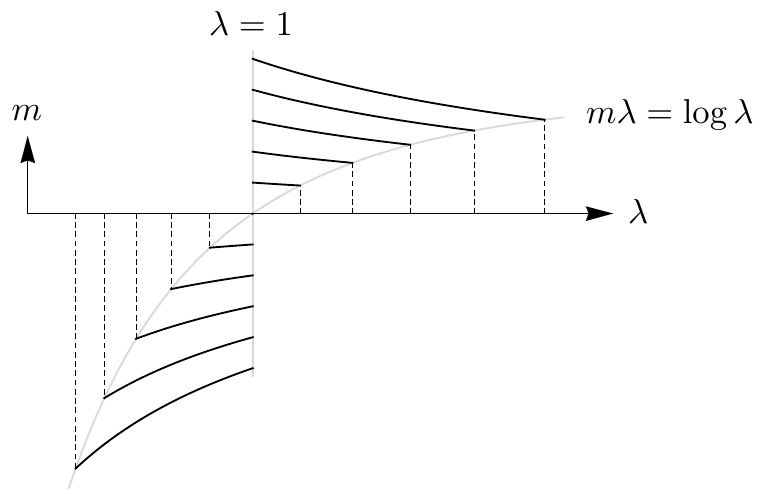}}{\input{fig_horizontal_diag}}
  \caption{
    Phase diagram of equation~\eqref{eq:horizontal_diag} for geodesics on~$\diag{n}$.
    For every $l>0$ there is a unique integral curve $\lambda(t)$ such that $\lambda(0)=1$ and $\lambda(1)=l$.
    In consequence, every $\Lambda\in\diag{n}$ is connected to the identity~$I$ by a unique horizontal geodesic.
    \iftoggle{final}{}{\protect{\color{blue}
    The existence of unique geodesics between any two points is a consequence of the geometry, since the Fisher--Rao metric on~$\diag{n}$ is Hessian.}}
  }\label{fig:horizontal_diag}
\end{figure}

It remains to compute the Riemannian distance $\hat d(\Lambda_0,\Lambda_1)$ between $\Lambda_0,\Lambda_1 \in \diag{n}$.
From \eqref{eq:met_diag} and \eqref{eq:diag_boundary_value_geodesics} we get
\begin{align}
  \hat d(\Lambda_0,\Lambda_1)^{2} &= \int_0^{1} \hat{\mathcal G}_{\Lambda(t)}(\dot\Lambda(t),\dot\Lambda(t))\ud t\\
  & = \int_0^{1}\frac{1}{2}\sum_{i=1}^{n} \left(\frac{\dot\lambda_i(t)}{\lambda_i(t)}\right)^{2}\ud t\\
  & =  \frac{1}{2}\sum_{i=1}^{n} \left(\log\left(\frac{(\Lambda_1)_{i}}{(\Lambda_0)_{i}}\right)\right)^{2} \\
  & =  \frac{1}{2}\sum_{i=1}^{n} \Big(\log((\Lambda_1)_{i})-\log((\Lambda_0)_{i}) \Big)^{2}.
\end{align}
A direct consequence is the following.

\begin{proposition}\label{prop:isometric_map_diag}
  The mapping
  \begin{equation}
    (\lambda_1,\ldots,\lambda_n) \mapsto \frac{1}{\sqrt{2}}\big(\log(\lambda_1),\ldots,\log(\lambda_n) \big)
  \end{equation}
  is an isometric isomorphism between $\diag{n}$, equipped with the Fisher--Rao metric \eqref{eq:met_diag}, and $\RR^{n}$, equipped with the standard Euclidean structure.
\end{proposition}

\iftoggle{final}{}{{\color{blue}

If $\Lambda,\Lambda' \in \diag{n}$, then the entropy of $\Lambda$ relative to $\Lambda'$ is given by
\begin{equation}
  H(\Lambda) = \frac{n}{2} + \frac{1}{2}\sum_{i=1}^{n}\left( \lambda'_i \lambda_i^{-1} + \log(\lambda_i' \lambda_i^{-1}\right)
\end{equation}

}}

\subsubsection{Vertical (or Isospectral) Flows}\label{sub:isospectral}

Let $A(t)$ be a curve such that $\varpi(A(t)) = \varpi(A(s))$ for all~$s$ and~$t$.
That is, $A(t)$ is a curve on a single orbit of $\Sym{n}/\OO(n)$.
Clearly, this implies that the eigenvalues of $A(t)$ are independent of~$t$.
Differential equations whose flows are such curves are called \emph{isospectral flows}.
From \eqref{eq:inf_action_on_sym} we see that any isospectral flow must be of the form
\begin{equation}
  \dot S = [\Phi(S),S],
\end{equation}
where $\Phi\colon \Sym{n}\to\oo(n)$ and $[\cdot,\cdot]$ is the matrix commutator.

The most studied case is the Toda flow~\cite{To1970}, which describes waves on non-linear lattices.
The Hamiltonian integrable structure of this flow and its connection to the KdV equation has been studied extensively, starting with the work of Flaschka~\cite{Fl1974}; see also \cite{Mo1975,DeLiNaTo1986,MoVe1991}.
Furthermore, Symes~\cite{Sy1982} showed that the Toda flow is a continuous version of the $QR$ algorithm for computing eigenvalues, so there is a strong connection to numerical linear algebra; see also \cite{DeNaTo1983,Wa1984,Ch1984}.
For a numerical treatment of general isospectral flows, see Calvo, Iserles, and Zanna~\cite{CaIsZa1997} and references therein.
The Toda flow is known to be a gradient flow \cite{BlBrRa1992}, but we shall not discuss it more here; we refer the survey paper by Tomei~\cite{To2013} for details and further references.
Instead, we make a connection to the work by Brockett~\cite{Br1991b}, who had the idea to construct isospectral gradient flows that diagonalize matrices.

Brockett's flow is constructed as follows.
The orbit of $M\in\Sym{n}$ is parametrized by $Q^{\top} M Q$ with $Q\in\OO(n)$.
Thus, a flow $\gamma(t)$ on the Lie group $\OO(n)$ induces a flow $\gamma(t)^{\top} M \gamma(t)$ on the orbit of $M$.
Now, $\OO(n)$ comes with a canonical Riemannian metric (corresponding to minus the Killing form).
At a base point $Q\in\OO(n)$ for vectors $U,V\in T_Q\OO(n)$, it is given by
\begin{equation}\label{eq:can_met_On}
  \mathcal{H}_{Q}(U,V) = \tr(Q^{\top}U Q^{\top} V).
\end{equation}
Brockett constructed a gradient flow on $\OO(n)$ such that if $Q^{\top}MQ$ is diagonal, then $Q$ is an equilibrium.
He came up with the functional
\begin{equation}
  E(Q) = \tr(NQ^{\top}MQ),
\end{equation}
where $N\in\diag{n}$, and showed that the gradient flow
\begin{equation}
  \dot Q = - \nabla_{\mathcal H} E(Q)
\end{equation}
is given by
\begin{equation}\label{eq:brockett_flow}
  \dot Q = - Q(N Q^{\top}M Q - Q^{\top}M Q N).
\end{equation}
Notice that if $Q^{\top}M Q$ is diagonal, then $Q$ is an equilibrium.
Thus, such an equilibrium corresponds to a diagonalization of $M$, i.e., it produces a spectral decomposition of $M$.
We now show how Brockett's flow \eqref{eq:brockett_flow} is related to the Fisher--Rao geometry presented in this paper.

The action of $\OO(n)$ on $W_1\in\Sym{n}$ is given by the map
\begin{equation}
  \Gamma(Q)\coloneqq Q^{\top}W_{1}Q.
\end{equation}
Its derivative is
\begin{equation}
  D\Gamma(Q)\cdot Q = \dot Q^{\top} W_1 Q + Q^{\top} W_1 \dot Q.
\end{equation}
Written in the left translated variable $\xi = Q^{\top} \dot Q$, and using that $\xi\in\oo(n)$ is skew-symmetric, we get
\begin{equation}
  D\Gamma(Q)\cdot Q = [\Gamma(Q),\xi].
\end{equation}

\begin{lemma}\label{lem:transpose_brocket_operator}
  The Riemannian transpose $D^{\top}\Gamma(Q)\colon T_{\Gamma(Q)}\Sym{n}\to T_{Q}\OO(n)$, defined by
  \begin{equation}
    \mathcal H_{Q}(D^{\top}\Gamma(Q)\cdot \dot W, \dot Q) = \bar{\mathcal G}_{\Gamma(Q)}(\dot W, D\Gamma(Q)\cdot Q),
  \end{equation}
  is given by
  \begin{equation}
    D^{\top}\Gamma(Q)\cdot \dot W = \frac{1}{2}Q[\Gamma(Q)^{-1},\dot W].
  \end{equation}
\end{lemma}

\begin{proof}
  From the definitions \eqref{eq:FRsym} and \eqref{eq:can_met_On} of $\bar{\mathcal G}$ and $\mathcal H$ we get
  \begin{align}
    \bar{\mathcal G}_{\Gamma(Q)}(\dot W,D\Gamma(Q)\cdot Q)
    &= \frac{1}{2}\tr(\Gamma(Q)^{-1}\dot W \Gamma(Q)^{-1}[\Gamma(Q),\xi]) \\
    &= \frac{1}{2}\tr(\Gamma(Q)^{-1}\dot W \Gamma(Q)^{-1}(\Gamma(Q)\xi - \xi\Gamma(Q)]) \\
    &= \frac{1}{2}\tr(\Gamma(Q)^{-1}\dot W \xi - \dot W \Gamma(Q)^{-1}\xi) \\
    &= \frac{1}{2}\tr( (\Gamma(Q)^{-1}\dot W - \dot W \Gamma(Q)^{-1})\xi) \\
    &= \frac{1}{2}\tr( [\Gamma(Q)^{-1},\dot W]\xi) \\
    &= \mathcal H_{I}\left(\frac{1}{2}[\Gamma(Q)^{-1},\dot W],\xi \right) \\
    &= \mathcal H_{Q}\left(\frac{1}{2}Q[\Gamma(Q)^{-1},\dot W],\dot Q \right),
  \end{align}
  where, in the last equality, we used the left invariance of $\mathcal H$.
  This concludes the proof.
\end{proof}

Consider now the entropy \eqref{eq:rel_entropy_fin_dim} relative to $N\in\diag{n}$, given by
\begin{equation}
  H(W) = \frac{n}{2} - \frac{1}{2}\tr(N W^{-1}) + \frac{1}{2}\log(\det(N W^{-1})).
\end{equation}
Through pullback by $\Gamma$, we then obtain a functional on $\OO(n)$, given by
\begin{equation}\label{eq:pullback_entropy_On}
  F(Q) \coloneqq H(\Gamma(Q)) = \frac{n}{2} - \frac{1}{2}\tr(N Q^{\top}W_1^{-1}Q) + \frac{1}{2}\log(\det(N Q^{\top}W_1^{-1} Q)).
\end{equation}

\begin{proposition}\label{prop:gradient_flow_pullback_entropy_On}
  The pullback entropy gradient flow
  \begin{equation}
    \dot Q = \nabla_{\mathcal H} F(Q)
  \end{equation}
  is given by
  \begin{equation}\label{eq:gradient_flow_pullback_entropy_On}
    \dot Q = -\frac{1}{2}Q[N,\Gamma(Q)^{-1}].
  \end{equation}
\end{proposition}

\begin{proof}
  First,
  \begin{align}
    \frac{\ud}{\ud t} F(Q) &= \frac{\ud}{\ud t} (H\circ \Gamma (Q)) = \bar{\mathcal G}_{\Gamma(Q)}(\nabla_{\bar{\mathcal G}} H(\Gamma(Q)), \frac{\ud}{\ud t}\Gamma(Q)) \\
    &= \bar{\mathcal G}_{\Gamma(Q)}(\nabla_{\bar{\mathcal G}} H(\Gamma(Q)),D\Gamma(Q)\cdot \dot Q) \\
    &= \mathcal H_{Q}(D^{\top}\Gamma(Q)\cdot \nabla_{\bar{\mathcal G}} H(\Gamma(Q)),\dot Q).
  \end{align}
  Thus, the gradient of $F$ with respect to $\mathcal H$ is given by
  \begin{equation}
    \nabla_{\mathcal H} F(Q) = D^{\top}\Gamma(Q)\cdot \nabla_{\bar{\mathcal G}} H(\Gamma(Q)).
  \end{equation}
  From \eqref{eq:entropy_gradient_flow_FR} we get
  \begin{equation}
    \nabla_{\bar{\mathcal G}} H(\Gamma(Q)) = N - \Gamma(Q).
  \end{equation}
  From \autoref{lem:transpose_brocket_operator} it then follows that
  \begin{align}
    D^{\top}\Gamma(Q)\cdot \nabla_{\bar{\mathcal G}} H(\Gamma(Q)) &=
    -\frac{1}{2}Q[N-\Gamma(Q),\Gamma(Q)^{-1}] = -\frac{1}{2}Q[N,\Gamma(Q)^{-1}],
  \end{align}
  where the last equality follows since $[\Gamma(Q),\Gamma(Q)^{-1}] = 0$.
  This proves the result.
\end{proof}

Comparing the pullback entropy gradient flow \eqref{eq:gradient_flow_pullback_entropy_On} with Brockett's flow \eqref{eq:brockett_flow}, and using that $\Gamma(Q)^{-1} = Q^{\top}W_1^{-1}Q$, we immediately get the following result, which shows how Brockett's flow is related to entropy and Fisher--Rao geometry.

\begin{corollary}\label{cor:brockett_as_pullback_entropy}
  Up to scaling by $1/2$, Brockett's flow \eqref{eq:brockett_flow} is the gradient flow \eqref{eq:gradient_flow_pullback_entropy_On} with $W_1 = M^{-1}$.
\end{corollary}

Implicitly, what this result shows is that Brockett's flow can be interpreted as the entropy gradient flow restricted to the Riemannian submanifold given by the orbit of $W_1$.
Thus, if $\Orb{W_1}$ denotes the orbit of $W_1$ and $\Pi_W\colon T_W\Sym{n}\to \Ver_W$ denotes orthogonal projection, then we have the following result.

\begin{proposition}
  The gradient flow of $H$ restricted to $\Orb{W_1}$
  \begin{equation}
    \dot W = \Pi_W\nabla_{\bar{\mathcal G}} H(W), \quad W(0) \in \Orb{W_1}
  \end{equation}
  is given by
  \begin{equation}\label{eq:entropy_gradient_flow_OrbW1}
    \dot W = \frac{1}{2}[W,[W^{-1},N]].
  \end{equation}
\end{proposition}

\begin{proof}
  If $W = \Gamma(Q)$ we have
  \begin{equation}
    \dot W = \frac{\ud}{\ud t}\Gamma(Q) = D\Gamma(Q)\cdot Q = [\Gamma(Q),\xi].
  \end{equation}
  It follows from \eqref{eq:gradient_flow_pullback_entropy_On} that $\xi = \frac{1}{2}[\Gamma(Q)^{-1},N] = \frac{1}{2}[W^{-1},N]$.
  This proves the \break result.
\end{proof}

The flow \eqref{eq:gradient_flow_pullback_entropy_On} also induces a flow on the orbit of $\Sigma_1 \coloneqq W_{1}^{-1}$, which recovers, up to scaling by $1/2$, the double bracket formulation of \eqref{eq:brockett_flow} given by Brocket~\cite[Eq.\ 2]{Br1991b}.

\begin{corollary}\label{cor:brockett_double_bracket_as_entropy_gradient_flow}
  Expressed in the variable $\Sigma = W^{-1}$, the entropy gradient flow \eqref{eq:entropy_gradient_flow_OrbW1} takes the double bracket form
  \begin{equation}\label{eq:double_bracket_flow}
    \dot \Sigma = \frac{1}{2}[\Sigma,[\Sigma,N]].
  \end{equation}
\end{corollary}

\begin{proof}
  Since $\Sigma = W^{-1}$ we get
  \begin{equation}
    \dot \Sigma = -W^{-1}\dot W W^{-1},
  \end{equation}
  and from \eqref{eq:entropy_gradient_flow_OrbW1}
  \begin{equation}
    \dot \Sigma = -\frac{1}{2}W^{-1}[W,[W^{-1},N]]W^{-1} = -\frac{1}{2}[[W^{-1},N],W^{-1}]
    = \frac{1}{2}[\Sigma,[\Sigma,N]].
  \end{equation}
  This proves the result.
\end{proof}

If $N\in\diag{n}$ is generic (all elements are different), then the double bracket flow \eqref{eq:double_bracket_flow} converges to a diagonal matrix in $\Orb{\Sigma_1}$, thus giving a spectral decomposition of $\Sigma_1$ \cite[Th.\ 2]{Br1991b}.\footnote{For numerical experiments confirming the convergence, see \cite[Fig.\ 1]{Br1991b}.}
Consequently, under the same condition, \eqref{eq:entropy_gradient_flow_OrbW1} converges to a diagonal matrix in $\Orb{W_1}$.




\subsubsection{Horizontal Gradient Flow to Factorize Characteristic Polynomials}\label{subsub:hor_flow_factorize_polynomials}

We have seen how to construct a vertical gradient flow on the orbit of $W_1$ that converges towards a spectral decomposition of $W_1$.
Here, we aim to construct a horizontal gradient flow, evolving on $\diag{n}$, such that the eigenvalues of $W_1$ are obtained in the limit.

The ideas goes as follows.
For $W_1\in\Sym{n}$, construct its characteristic polynomial $p_1 = \varpi(W_1)$.
Then define a functional $\hat F\colon\poly{n}\to \RR^{+}$ such that $\hat F(p)=0$ if and only if $p=p_1$.
Lift $\hat F$ to a functional on $\diag{n}$
\begin{equation}
  F(\Lambda) = \hat F(\varpi(\Lambda))
\end{equation}
and consider the gradient flow
\begin{equation}\label{eq:abstract_diag_flow}
  \dot \Lambda = - \nabla_{\bar{\mathcal G}} F(\Lambda).
\end{equation}
If $F$ is a strictly convex functional on $\diag{n}$ with respect to $\bar{\mathcal G}$, then this flow will converge to $\Lambda_{\infty}$ such that $\varpi(\Lambda_{\infty}) = p_1$.

There is, however, an obstruction with the suggested approach.
Namely, $F$ can never be strictly convex.
If it was, then $F$ would have a unique minimum, so there should be a unique sequence of eigenvalues of $W_1$.
But this cannot be, since any reshuffling of the eigenvalues gives a different minimum.

Let us take a closer look of what can happen.
If $\Lambda=\mathrm{diag}(\lambda_1,\ldots,\lambda_n)$ and, say, $\lambda_1=\lambda_2$, then the symmetry $\lambda_1\leftrightarrow \lambda_2$ implies that $(\nabla_{\bar{\mathcal G}}F(\Lambda))_{1} = (\nabla_{\bar{\mathcal G}}F(\Lambda))_{2}$, so $\dot\lambda_1 = \dot\lambda_2$ along the gradient flow.
That is, $\lambda_1(t) = \lambda_2(t)$, so the flow cannot converge to a solution where $\lambda_1 \neq \lambda_2$.
Consequently, if $\lambda_1(0) < \lambda_2(0)$, then $\lambda_1(t) < \lambda_2(t)$, so the flow of \eqref{eq:abstract_diag_flow} preserves ordering.
As a remedy, if $W_1$ is generic (all eigenvalues different), then it is possible for the functional $F$ to be strictly convex on the subset of ordered elements
\begin{equation}
  \diagord{n} = \{ \Lambda\in \diag{n}\mid \lambda_1 < \lambda_2 < \ldots < \lambda_n \}.
\end{equation}

Let us now consider the specific case
\begin{equation}\label{eq:Fhat_monomial_norm}
  \hat F(p) = \frac{1}{2}\pair{p-p_1,p-p_1}
\end{equation}
where $\pair{\cdot,\cdot}$ denotes the Euclidean inner product of vectors of monomial coefficients.

\begin{lemma}\label{lem:gradient_horizontal_diag}
  Let $\hat F$ be given by \eqref{eq:Fhat_monomial_norm}.
  Then
  \begin{equation}
    \nabla_{\bar{\mathcal G}} F(\Lambda) = \mathrm{diag}(\lambda_1^{2} y_1,\ldots,\lambda_n^{2 }y_n)
  \end{equation}
  where
  \begin{equation}
    y_k = \pair{\prod_{i\neq k}(\lambda-\lambda_i),p_1-\varpi(\Lambda)}.
  \end{equation}
\end{lemma}

\begin{proof}
  We have
  \begin{align}
    \frac{\ud}{\ud t}F(\Lambda) &= \frac{\ud}{\ud t}\hat F(\varpi(\Lambda)) = \pair{\frac{\ud}{\ud t}\varpi(\Lambda),\varpi(\Lambda)-p_1} \\
    &= \pair{\sum_{k=1}^{n}-\dot\lambda_k \prod_{i\neq k}(\lambda-\lambda_i),\varpi(\Lambda)-p_1} \\
    &= \sum_{k=1}^{n}\dot\lambda_k\pair{\prod_{i\neq k}(\lambda-\lambda_i),p_1-\varpi(\Lambda)}
  \end{align}
  \begin{eqnarray*}
  &= \sum_{k=1}^{n}\dot\lambda_k y_k = \tr\Big(\dot\Lambda\,\mathrm{diag}(y_1,\ldots,y_n) \Big) \\
   &= \bar{\mathcal{G}}_\Lambda(\dot\Lambda,\Lambda^{2}\,\mathrm{diag}(y_1,\ldots,y_n)) \\
    &= \bar{\mathcal{G}}_\Lambda(\dot\Lambda,\mathrm{diag}(\lambda_1^2 y_1,\ldots,\lambda_n^2 y_n)).
  \end{eqnarray*}
  This proves the result.
\end{proof}

The gradient flow \eqref{eq:abstract_diag_flow} is thereby given by
\begin{equation}\label{eq:diag_flow}
  \dot\Lambda = - \Lambda^2 Y(\Lambda),
\end{equation}
where
\begin{equation}
  Y(\Lambda) = \mathrm{diag}(y_1,\ldots,y_n).
\end{equation}

The question of convexity of $F$ and convergence towards a limit of \eqref{eq:diag_flow} is left for future work.
We shall, however, give a numerical example indicating convergence.

\begin{example}\label{ex:horizontal_diag}
  Let $W_1\in\Sym{3}$ have a spectral decomposition $W_1 = Q^{\top}\Lambda_1 Q$ with
  \begin{equation}
    \Lambda_1 = \begin{pmatrix}
      5 & 0 & 0 \\
      0 & 1/2 & 0 \\
      0 & 0 & 1
    \end{pmatrix}.
  \end{equation}
  (What $Q$ is does not affect the flow.)
  We then compute $p_1(\lambda) = \varpi(W_1) = \det(I\lambda - W_1)$ in the monomial basis
  \begin{equation}
    p_1(\lambda) = \lambda^{3} + b_2\lambda^{2} + b_1\lambda + b_0.
  \end{equation}
  From \autoref{lem:gradient_horizontal_diag} we then get the gradient flow
  \begin{equation}\label{eq:example_diag_flow}
    \dot\lambda_k = -\lambda_k^{2} y_k
  \end{equation}
  where
  \begin{equation}
    y_k = \pair{\prod_{k\neq i}(\lambda-\lambda_k),p_1-\varpi(\Lambda)}
  \end{equation}
  with the inner product $\pair{\cdot,\cdot}$ being the Euclidean inner product for the polynomial coefficients in the monomial basis.

  Since $W_1$ is generic (no eigenvalues are the same), it is conceivable that the flow converges to the ordered sequence of eigenvalues if we pick initial data in $\diagord{n}$.
  Thus, we need to come up with a strictly growing sequence of initial data.
  To this extent, we use that $\det(W_1) = (-1)^{n}p_1(0) = -b_0$ and we take $\Lambda_0$ to be the geometric sequence
  \begin{equation}
    \Lambda_0 = \mathrm{diag}(1,(-b_0)^{1/3},(-b_0)^{2/3}).
  \end{equation}
  Thus, the initial data fulfills $\det(\Lambda_0) = \det(W_0)$.\footnote{Possible, a better choice of initial data could be selected by also asserting $\tr(\Lambda_0) = \tr(W_0)$.}

  We discretize \eqref{eq:example_diag_flow} by the classical 4th order Runge--Kutta method \cite[\S\,322]{Bu2008}, with time-step~$\Delta t = 0.01$.
  The evolution of
  \begin{equation}
    \Lambda(t) = \begin{pmatrix}
      \lambda_1(t) & 0 & 0 \\
      0 & \lambda_2(t) & 0 \\
      0 & 0 & \lambda_3(t)
    \end{pmatrix}.
  \end{equation}
  is shown in \autoref{fig:diag_example_spectral};
  $\Lambda(t)$ appears to converge to the ordered sequence of eigenvalues of $W_1$, i.e., the sequence $(1/2,1,5)$.
  A plot of $F(\Lambda(t))$ is given in \autoref{fig:diag_example_spectral_convergence};
  it appears that $F(\Lambda(t))\to 0$ exponentially fast as $t\to \infty$.
  A more careful study is left as a future research topic.
\end{example}

\begin{figure}
  \iftoggle{arxiv}{\includegraphics[scale=\iftoggle{aims}{1.0}{1.2}]{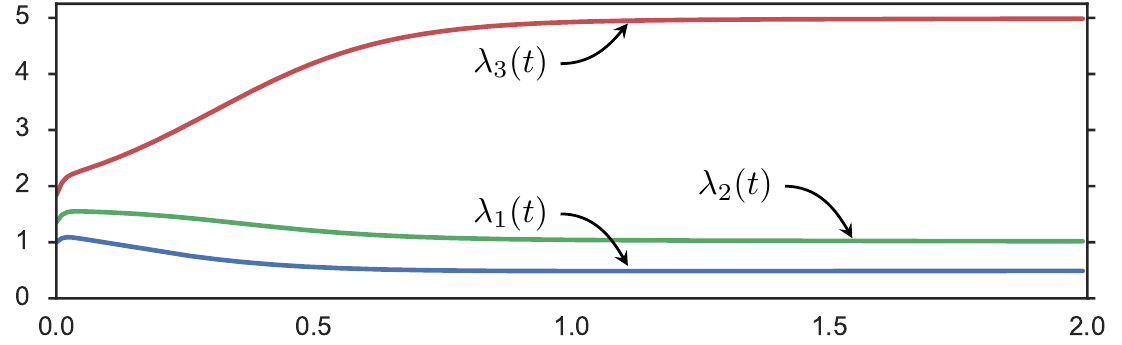}}{\input{fig_Dplot}}
  \caption{Evolution of the horizontal gradient flow in \protect\autoref{ex:horizontal_diag}.
  $\Lambda(t)$ appears to converge to the ordered sequence of eigenvalues $(1/2,1,5)$.
  }\label{fig:diag_example_spectral}
\end{figure}

\begin{figure}
  \iftoggle{arxiv}{\includegraphics[scale=\iftoggle{aims}{1.0}{1.2}]{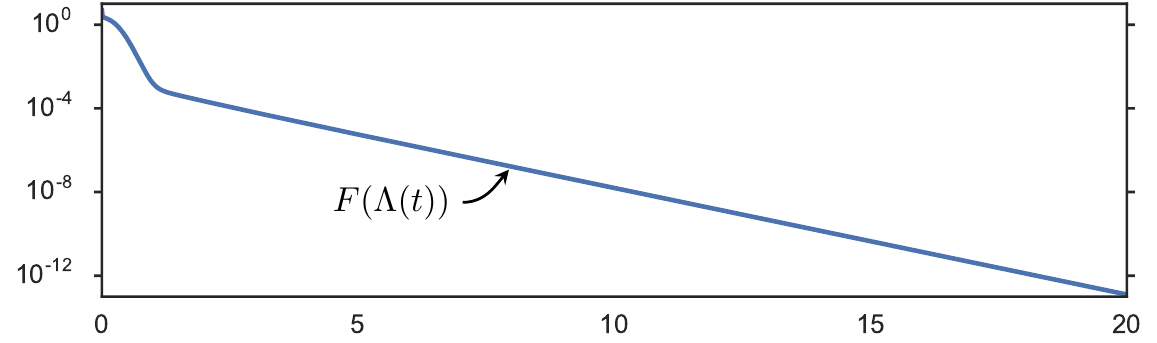}}{\input{fig_Derrorplot}}
  \caption{Convergence of $F(\Lambda(t))$ towards the minimum for the horizontal gradient flow in \protect\autoref{ex:horizontal_diag}.
  The convergence appears to be exponential.
  }\label{fig:diag_example_spectral_convergence}
\end{figure}

\iftoggle{final}{}{{\color{blue}

\begin{equation}
  \hat H(p) \coloneqq \frac{n}{2} - \frac{p'(0)p_1(0)}{2 p(0)} + \frac{1}{2} \log\left(\frac{p_1(0)}{p(0)}\right).
\end{equation}

\begin{equation}
  \frac{d}{d\lambda}\Big|_{\lambda=0} \log(\det(\Lambda-\lambda \Lambda_1)) = -\tr(\Lambda^{-1}\Lambda_1)
\end{equation}

Recall that the entropy of $\Lambda\in\diag{n}$ relative to $\Lambda_1\in\diag{n}$ is given by
\begin{equation}
  H(\Lambda) = \frac{n}{2} -\frac{1}{2}\tr(\Lambda_1\Lambda^{-1}) +\frac{1}{2}\log\left( \det(\Lambda_1 \Lambda^{-1})\right).
\end{equation}
We then have the following result, which shows that $\hat H(p)$ is the right definition of entropy of $p$ relative to $p_1$, i.e., $\hat H$ is the relative entropy functional on $\poly{n}$.

\begin{lemma}
  If $\Lambda_1 \in \Orb{W_1}$, then
  \begin{equation}
    H(\Lambda) = \hat H(\pi(\Lambda)).
  \end{equation}
\end{lemma}

\begin{proof}
  Let $\Lambda = \mathrm{diag}(\lambda_1,\ldots,\lambda_n)$ and $\Lambda_1 = \mathrm{diag}(\sigma_1,\ldots,\sigma_n)$.
  Since $\Lambda_1$ and $W_1$ belong to the same orbit, we have
  \begin{equation}
    p_1(\sigma) = \varpi(W_1)(\sigma) = \varpi(\Lambda_1)(\sigma) = \prod_{i=1}^{n}(\sigma_i-\sigma).
  \end{equation}
  Thus,
  \begin{equation}
    p_1(0) = \prod_{i=1}^{n}\sigma_i = \det(\Lambda_1).
  \end{equation}
  We also have
  \begin{equation}
    p(\lambda) = \varpi(\Lambda)(\lambda) = \prod_{i=1}^{n} (\lambda_i-\lambda),
  \end{equation}
  so $p(0) = \det(\Lambda)$.
  Thus,
  \begin{equation}
    \frac{1}{2}\log\left( \det(\Lambda_1 \Lambda^{-1})\right) = \frac{1}{2}\log\left( \frac{\det(\Lambda_1)}{\det(\Lambda)}\right) = \frac{1}{2}\log\left( \frac{p_1(0)}{p(0)}\right) .
  \end{equation}
  This shows that the last term of $\hat H(p)$ and $H(\Lambda)$ coincide.

  For the other term, notice that
  \begin{align}
    \frac{\ud}{\ud\lambda}\Big|_{\lambda = 0} \log\left( \frac{p(\lambda)}{p_1(\lambda)} \right)
    &= \frac{\ud}{\ud\lambda}\Big|_{\lambda = 0} \sum_{i=1}^{n} \log\left( \frac{\lambda_i-\lambda}{\sigma_i-\lambda} \right) \\
    &= -\sum_{i=1}^{n}\frac{\sigma_i}{\lambda_i}\frac{\ud}{\ud\lambda}\Big|_{\lambda = 0} \frac{\lambda_i-\lambda}{\sigma_i - \lambda}
  \end{align}
\end{proof}

%
%
%
%

\begin{lemma}
  The derivative $D\varpi(\Lambda)\colon T_\Lambda\diag{n} \to T_{\varpi(\Lambda)}\poly{n}$ is given by
  \begin{equation}
    D\varpi(\Lambda)\cdot\dot\Lambda =  \lambda \mapsto \sum_{k=1}^{n} \dot\lambda_k \frac{\varpi(\Lambda)(\lambda)}{\lambda-\lambda_k}.
  \end{equation}
\end{lemma}

\begin{proof}
  \begin{align}
    \frac{\ud}{\ud t} \varpi(\Lambda) &= \frac{\ud}{\ud t} \prod_{i=1}^{n}(\lambda-\lambda_i) \\
    &= \sum_{k=1}^{n} \frac{\dot\lambda_k}{\lambda-\lambda_k}\prod_{i=1}^{n}(\lambda-\lambda_i) \\
    &= \sum_{k=1}^{n} \frac{\dot\lambda_k}{\lambda-\lambda_k}\varpi(\Lambda)(\lambda).
  \end{align}
  \begin{align}
    \frac{\ud}{\ud t} \varpi(\Lambda) &= \frac{\ud}{\ud t}\det(\Lambda - \lambda I) \\
    &= \det(\Lambda - \lambda I) \tr( (\Lambda - \lambda I)^{-1}\dot\Lambda )
  \end{align}
  \begin{align}
    \frac{\ud}{\ud t} H(\varpi(\Lambda)) = \pair{\ud H(\varpi(\Lambda)), D\varpi(\Lambda)\cdot \dot\Lambda}
  \end{align}
\end{proof}

\begin{lemma}
  Let $\poly{n}$ be equipped with a Riemannian metric $\mathcal H$.
  Then the Riemannian transpose $D^{\top}\varpi(\Lambda)\colon T_{\varpi(\Lambda)}\poly{n}\to T_\Lambda\diag{n}$ is given by
  \begin{equation}
    hej
  \end{equation}
\end{lemma}

\begin{proof}
  \begin{align}
    \mathcal H(D\varpi(\Lambda)\cdot\dot\Lambda,\dot p) &= \sum_{i=1}^{n} \dot \lambda_i\underbrace{\mathcal H(\prod_{k\neq i}(\lambda_k-\lambda),\dot p(\lambda))}_{r_i} \\
    &= \tr(\dot\Lambda\, \mathrm{diag}(r_1,\ldots,r_n)) \\
    &= \tr(\Lambda^{-1}\dot\Lambda\, \Lambda^{-1}\Lambda\mathrm{diag}(r_1,\ldots,r_n)\Lambda) \\
    &= \bar{\mathcal{G}}\big(\dot\Lambda,\mathrm{diag}(\lambda_1^{2} r_1,\ldots,\lambda_n^{2 }r_n)\big).
  \end{align}
\end{proof}
}}

\subsection{Singular value decomposition}\label{sec:svd}

If we combine the geometry developed for the chain of projections
\begin{equation}
  \begin{tikzcd}
    \GL(n) \arrow{d}{\pi}  \\
    \Sym{n} \arrow{d}{\varpi} \\
    \poly{n}
  \end{tikzcd}
\end{equation}
we obtain the singular value decomposition.

\begin{theorem}[Singular value decomposition]
  Let $A\in\GL(n)$.
  Then there exist $U,V\in\OO(n)$ and $\Sigma\in\diag{n}$ such that
  \begin{equation}
    A = U\Sigma V^\top.
  \end{equation}
\end{theorem}

\begin{proof}
  Let $W=\pi(R)=A^\top A$.
  From \autoref{thm:spectral_decomposition} we get $W=Q_2^\top\Lambda Q_2$.
  Let $\sqrt{W} \coloneqq Q_2^\top\sqrt{\Lambda}Q_2$.
  Then $A$ and $\sqrt{W}$ belong to the same fiber (since $\pi(A) = \pi(\sqrt{W})$), so there exists a $Q_1\in\OO(n)$ such that $A=Q_1\sqrt{W}$.
  Thus, $A = Q_1Q_2^\top\sqrt{\Lambda}Q_2$.
  The result now follows by taking $U=Q_1Q_2$, $\Sigma=\sqrt{\Lambda}$, and $V=Q_2^\top$.
\end{proof}

\todo[inline]{Add \emph{Outlook}-section about infinite-dimensional analogue: (i) QR factorization of diffeomorphisms on the torus using optimal information transport, (ii) spectral decomposition of densities on the torus and connections to KAM theory.}

\section{Outlook}\label{sec:outlook}

In this section we give a few topics and future research directions directly related to the geometry described in this paper.


\subsection{New geometric techniques for optimal transport}
We start with some ideas related to the geometry of optimal mass transport.

\subsubsection{Limit of Lifted Gradient Flow}\label{subsub:limit}
In \autoref{subsub:lifted_gradient_omt_linear} we gave a geometric proof of \autoref{lem:shortest_geodesic_finite_dim}, the central result in proving existence and uniqueness of the optimal transport problem in the linear category (\autoref{prob:linear_omt_2}), or, equivalently, of the polar decomposition of matrices (\autoref{thm:polar_decomposition_matrices}).
The proof is based on existence of a unique limit for the lifted gradient flow \eqref{eq:lifted_gradient_flow_polar_omt_finite}.
In turn, this limit is obtained by showing that the Hessian of the functional for the gradient flow is negative definite.

It is a natural idea to try to extend the finite-dimensional geometric proof to the infinite-dimensional, smooth case.
That is, to prove that the infinite-dimensional lifted gradient flow \eqref{eq:explicit_lifted_gradient_flow_OMT_smooth} has a unique limit in the space of strictly convex functions.
If that can be established, \autoref{lem:shortest_geodesic} follows as in finite dimensions.

Of course, gradient flows on infinite-dimensional spaces are more subtle than on finite-dimensional spaces.
In addition, a Banach space setting might not suffice in the smooth category: for completeness one needs a Fréchet topology, for example as developed by Hamilton~\cite{Ha1982} or Kriegl and Michor~\cite{KrMi1997b}.
Regardless of these technicalities, however, the first step is to show that the Hessian of the functional $\hat F$ in \eqref{eq:Fhat_omt_smooth} is negative definite.
We shall now give some brief calculations in this direction.

Let $\phi=\phi(t)$ be a geodesic curve in the space of strictly convex functions.
Then, from \eqref{eq:geodesic_eq_omt} it follows that $\ddot\phi = 0$ and from \eqref{eq:dFdt_calc} we get
\begin{align}
  \hess(\hat F)(\phi) &= \frac{\ud^{2}}{\ud t^{2}} \hat F(\phi) \\
  &= \int_{\RR^{n}} \tr\left( \nabla^{2}\dot\phi \frac{\ud}{\ud t}(\nabla^{2}\phi)^{-1} \right)\mu_0 + \frac{\ud}{\ud t}\int_{\RR^{n}} (\nabla \log(\rho_1))\circ\nabla\phi \cdot \nabla\dot\phi \;\mu_0 \\
  &= \int_{\RR^{n}} \tr\left( \nabla^{2}\dot\phi \frac{\ud}{\ud t}(\nabla^{2}\phi)^{-1} \right)\mu_0 + \int_{\RR^{n}} \left(\left(\nabla^{2} \log(\rho_1)\right)\circ\nabla\phi  \cdot \nabla\dot\phi\right) \cdot \nabla\dot\phi \;\mu_0.
\end{align}
That the first term is negative follows from the same calculation as in the proof of \autoref{lem:pos_def_hessian_lifted_omt_finite}, since $\nabla^{2}\phi$ is positive definite at each point.
If $\rho_1$ is \emph{log-concave} then the Hessian matrix $\nabla^{2} \log(\rho_1)(x)$ is negative semi-definite for all $x\in\RR^{n}$, so the second term is non-positive.
We thereby come by the following conjecture.

\begin{conjecture}
  Let $\rho_1$ be log-concave (that is, $\log\circ\rho_1$ is concave).
  Then
  the flow \eqref{eq:explicit_lifted_gradient_flow_OMT_smooth} has a unique limit in the set of strictly convex functions.
\end{conjecture}

Although most standard probability distribution functions in statistics are log-concave, the restriction on $\rho_1$ is rather stringent.
(Notice that there is no restriction on $\rho_0$.)
Ideally, $\hess(\hat F) < 0$ for all $\rho_1$; to achieve this one might have to consider a different functional $\hat F$.

We aim to investigate the geometric approach for \autoref{lem:shortest_geodesic}, as sketched here, in a forthcoming publication.

\subsubsection{Numerical Method Based on the Lifted Gradient Flow}
If the lifted gradient flow \eqref{eq:explicit_lifted_gradient_flow_OMT_smooth} has nice convergence properties, it is natural to consider it for numerical computation of the solution to \autoref{prob:smooth_omt}.
This, of course, requires space and time discretization to yield a fully discrete flow.

The field of numerical methods for optimal transport has been growing steadily since the classical paper by Benamou and Brenier~\cite{BeBr2000}.
Surveys of state-of-the-art are given by Peyré~\cite[\S\,1.2]{Pe2015} and Benamou, Brenier, and Oberman~\cite{BeBrOb2015}.

\subsubsection{Geometric Analysis of the Vertical Gradient Flow}\label{subsub:outlook_vert_limit}
The numerical example in \autoref{subsub:vert_gradient_flow_linear_omt}, for the vertical gradient flow for the polar decomposition of matrices, suggests convergence.
A geometric approach for establishing this result could be to study the Hessian of the distance functional restricted to the fiber, as was done for the lifted entropy gradient flow in \autoref{subsub:lifted_gradient_omt_linear}.

\subsubsection{Wasserstein Analog of Ebin's Riemannian Metric}\label{sub:reviewers_idea}
On the space of Riemannian metrics there is a natural Riemannian structure, first studied by Ebin~\cite{Eb1967}, and later by Freed and Groisser~\cite{FrGr1989}, Gil-Medrano and Michor~\cite{GiMi1991}, and Clarke~\cite{Cl2010,Cl2013}.
This Riemannian structure induces the Fisher--Rao metric, by the projection taking a Riemannian metric to its corresponding volume form.
Now, the observation is that the formula \eqref{eq:FRsym}, for the Fisher--Rao metric restricted to Gaussian distributions, essentially recovers Ebin's metric (lacking only integration of the domain of interest).

The idea is then to study the metric on the space Riemannian metrics originating, analogously, from the Wasserstein metric restricted to Gaussian distributions \eqref{eq:bar_metric_linear_OMT}.
This would then yield a new geometry on the space of Riemannian metrics, related, not to the Fisher--Rao metric, but to the physically relevant Wasserstein geometry.


\subsection{New geometric techniques associated with Fisher--Rao}
As already mentioned, the finite-dimensional Fisher--Rao geometry studied in \autoref{sec:fisher_rao} naturally extends to infinite-dimensions (see \cite{Fr1991,KhLeMiPr2013,Mo2015}).
Therefore, it is natural to look for decompositions of diffeomorphisms analogous to the matrix decompositions studied in \autoref{sec:fisher_rao}.
In addition, the horizontal flow in \autoref{subsub:hor_flow_factorize_polynomials} to factorize the characteristic polynomial would be interesting to study further.

\subsubsection{$QR$ (or Iwasawa) Decomposition of Diffeomorphisms}
Let us recap the Wass- \break erstein geometry described in \autoref{sec:wasserstein}:
Brenier~\cite{Br1991} showed that the polar decomposition of matrices has an infinite-dimensional analogue, namely the polar decomposition of maps.
The geometric approach of Otto~\cite{Ot2001} makes the relation more transparent in terms of Riemannian metrics on diffeomorphisms and densities, and the observation that Gaussian distributions is a finite-dimensional submanifold of the space of all densities.

In Fisher--Rao geometry, there is a similar picture.
Indeed, an infinite-dimensional analogue of the $QR$ decomposition is developed in \cite{Mo2015} (for diffeomorphisms on a compact manifold, in the category of Banach manifolds).
It would be interesting to study this infinite-dimensional analogue of the $QR$ decomposition in more detail, to see how much of the structure in \autoref{sec:qr} that is retained.
Suggestively, one could take a flat compact manifold, such as the $n$--torus $\TT^{n}$, and try to relate the finite and infinite-dimensional $QR$ decompositions, analogously to how the finite and infinite-dimensional polar decompositions are related.

\iftoggle{final}{}{{\color{blue}
For example, consider diffeomorphisms on the flat $n$--torus $\TT^{n}$.
From \cite[Thm 5.6]{Mo2015} it follows\footnote{One has to slightly modify the metric in \cite[Thm\ 5.6]{Mo2015}, but the proof is the same.}
that any $\varphi\in\Diff(\TT^n)$ can be decomposed as\footnote{To be precise, the diffeomorphisms should be in the Sobolev class $H^s$ for $s>n/2+1$.}
\begin{equation}
  \varphi = \eta\circ \psi,
\end{equation}
where $\eta \in \Diffvol(\TT^n)$ and
\begin{equation}\label{eq:qr_diffeos}
  \psi(x_1,\ldots,x_n) = \left(\begin{array}{l}
    f_1(x_1) \\
    f_2(x_1,x_2) \\
    \;\,\vdots\qquad\quad
    \\
    f_n(x_1,\ldots,x_n)
  \end{array}
  \right)
  ,
  \qquad
  f_k\colon\TT^k\to\RR, \quad \frac{\pd f_k}{\pd x_k} > 0 .
\end{equation}
This is an infinite-dimensional analogue of the $QR$ decomposition in \autoref{thm:QR}, just as Brenier's decomposition in \autoref{thm:polar_decomposition_smooth} is an infinite-dimensional analogue of the polar decomposition of matrices in \autoref{thm:polar_decomposition_matrices}.
}}

\subsubsection{Spectral decomposition of densities}
The Fisher--Rao geometry is richer than the Wasserstein geometry: the Riemannian metric boasts more symmetries.
These extra symmetries allow the continuation of the reduction scheme to get the spectral decomposition of symmetric matrices in \autoref{sec:eigen}.

It is natural to look for an analogue of the spectral reduction in the infinite-dimensional setting, where the analogue of the space $\Sym{n}$ of symmetric matrices is $\Dens(\RR^n)$.
Just as $\Sym{n}$ can be viewed as the space of normalized inner products on $\RR^{n}$, we may think of $\Dens(\RR^{n})$ as the space of (smooth) normalized inner products on $L^{2}(\RR^{n})$.
Changing from $\RR^{n}$ to the torus $\TT^n$ (things are easier to prove for compact manifolds) we are then looking for a way to represent (or at least cover, as in the finite-dimensional case) the quotient $\Dens(\TT^{n})/\Diff(\TT^{n})$.
The formal decomposition thereby obtained would be an infinite-dimensional spectral decomposition of smooth densities.
Such a decomposition could have connections to integrable systems and the KAM theorem.

\iftoggle{final}{}{{\color{blue}
We come by the following conjecture.

\begin{conjecture}\label{conj:spectral}
  Given $\mu\in\Dens(\TT^n)$, there exist $\eta\in\Diffvol(\TT^n)$ (a volume-preserving change of coordinates) such that
  \begin{equation}
    (\eta^*\mu)(x_1,\ldots,x_n) = \lambda_1(x_1)\lambda_2(x_2)\cdots \lambda_n(x_n) \ud x
  \end{equation}
  where $\lambda_k\colon S^1\to\RR^{+}$.
\end{conjecture}

The functions $\lambda_k$ should be interpreted as `eigenvalues of $\mu$', and the set $\{\lambda_1,\ldots,\lambda_n\}$ corresponds to an infinite-dimensional ellipsoid.
Likewise, the diffeomorphism $\eta$ plays the role of a `matrix of eigenvectors' for $\mu$.

In view of the whole chain $\Diff(\TT^{n})\to $

Here are some questions related to Conjecture \ref{conj:spectral} that would be interesting to pursue:
}}

\subsubsection{New Numerical Methods for Spectral Decompositions}\label{subsub:new_num_methods_spectral}

As discussed in \autoref{sub:isospectral}, there are several vertical (isospectral) flows for obtaining spectral decompositions, e.g., the Toda and Brockett flows.
The horizontal flow \eqref{eq:diag_flow} in \autoref{subsub:hor_flow_factorize_polynomials} to factorize characteristic polynomials represents a new type of flows for spectral decompositions.
It would be interesting to study this class of flows in-depth.

Convergence to a limit is, of course, one important question.
Another is to develop numerical methods based on these flows.
For that, one should look for a strictly convex functional $\hat F$ on $\poly{n}$ such that the gradient vector field $\nabla_{\bar{\mathcal G}}(\Lambda)$ can be evaluated efficiently.
For sparse, banded matrices, there are $\mathcal{O}(n)$ algorithms for evaluating the characteristic polynomial \cite{KrTe2008}.
This might be helpful.
It is plausible that some known iterative methods for computing eigenvalues can be seen as discretizations of flows of the form in \autoref{subsub:hor_flow_factorize_polynomials}.
The computation of eigenvalues of symmetric matrices is discussed thoroughly in the monograph by Parlett~\cite{Pa1980}.
For a summary of numerical methods for spectral decompositions, see \cite{Wa1993,GoVo2000}.



\vspace*{-4pt}
\section*{Acknowledgments}

I would like to thank Darryl Holm, Boris Khesin, Peter Michor, and Olivier Verdier for helpful discussions.
Furthermore, I am grateful to the anonymous reviewers, who provided many excellent suggestions for improvement.
In particular, the first reviewer suggested the idea described in \autoref{sub:reviewers_idea}.

\providecommand{\href}[2]{#2}
\providecommand{\arxiv}[1]{\href{http://arxiv.org/abs/#1}{arXiv:#1}}
\providecommand{\url}[1]{\texttt{#1}}
\providecommand{\urlprefix}{URL }

\medskip
Received January 2016; revised August 2016.
\medskip

\end{document}